\documentclass[a4paper, 11pt, headings = small, abstract]{scrartcl}

\usepackage[english]{babel}
\usepackage[utf8]{inputenc}
\usepackage{geometry}

\usepackage[semibold]{libertine}
\usepackage[upint,amsthm]{libertinust1math}

\usepackage{amsmath}
\usepackage{amsfonts}
\usepackage{amssymb}
\usepackage{amsthm}
\usepackage{mathtools}
\usepackage{paralist}
\usepackage{comment}
\usepackage{graphicx}     
\usepackage{caption}      
\usepackage{subcaption}   
\usepackage{algorithm}
\usepackage{algpseudocode}
\usepackage{enumitem}

\usepackage[authoryear]{natbib}
\usepackage[dvipsnames]{xcolor}
\usepackage[colorlinks=true,linkcolor=blue,citecolor=blue,pdfborder={0 0 0}]{hyperref}
\usepackage{orcidlink}
\usepackage{xurl}

\usepackage{dsfont}
\usepackage{bm}

\usepackage{graphicx}
\usepackage{enumerate}
\usepackage{stackengine}
\stackMath

\newcommand{\triangleA}{\stackon[0.5pt]{A}{\kern0.2em\scalebox{0.6}{$\bm{\triangle}$}}}

\newcommand{\triangleAa}{\stackon[0.5pt]{A}{\kern0.2em\scalebox{0.6}{$\bm{\triangle}$}}\,\!^{(1)}}

\newcommand{\triangleAb}{\stackon[0.5pt]{A}{\kern0.2em\scalebox{0.6}{$\bm{\triangle}$}}\,\!^{(2)}}

\usepackage{tikz}
\usetikzlibrary{positioning}
\usetikzlibrary{backgrounds}

\newtheorem{theorem}{Theorem}[section]
\newtheorem{proposition}[theorem]{Proposition}
\newtheorem{lemma}[theorem]{Lemma}

\theoremstyle{definition}

\newtheorem{condition}[theorem]{Condition}

\newtheorem{model}[theorem]{Model}

\theoremstyle{remark}
\newtheorem{remark}[theorem]{Remark}

\allowdisplaybreaks[4]

\numberwithin{equation}{section}

\newcommand{\R}{\mathbb{R}}

\newcommand{\N}{\mathbb{N}}
\newcommand{\eps}{\varepsilon}

\newcommand{\Eb}{\mathbb {E}}

\newcommand{\Mb}{\mathbb {M}}

\newcommand{\Sb}{\mathbb {S}}

\newcommand{\Prob}{\mathbb{P}}   
\newcommand{\Exp}{\operatorname{E}}
\newcommand{\Var}{\operatorname{Var}}

\newcommand{\argmin}{\operatornamewithlimits{\arg\min}}

\newcommand{\diff}{\mathrm{d}}

\newcommand{\kn}{k_n}
\newcommand{\knon}{k}
\newcommand{\Acp}{\mathcal A_{c+}}
\newcommand{\Ap}{\mathcal A_{+}}

\newcommand{\Arp}{\mathcal A_{r+}}
\newcommand{\Apure}{\mathcal A_{\mathrm{pure}}}

\newcommand{\supp}{\textrm{supp}}
\newcommand{\hyperpurevar}{\kappa}
\newcommand{\hypersparsity}{\bar\kappa}

\begin{document}

\title{\fontsize{16}{19} Structured linear factor models for tail dependence}

\author{
Alexis Boulin\thanks{Ruhr-Universität Bochum, Fakultät für Mathematik, 44780 Bochum, Germany. Email: \href{mailto:alexis.boulin@rub.de}{alexis.boulin@rub.de}}
~\orcidlink{0000-0003-0548-2726}
\and
Axel B\"ucher\thanks{Ruhr-Universität Bochum, Fakultät für Mathematik, 44780 Bochum, Germany. Email: \href{mailto:axel.buecher@rub.de}{axel.buecher@rub.de}}
~\orcidlink{0000-0002-1947-1617}
}

\date{\today}

\maketitle

\begin{abstract}
A common object to describe the extremal dependence of a $d$-variate random vector $\bm X$ is the stable tail dependence function $L$. Various parametric models have emerged, with a popular subclass consisting of those stable tail dependence functions that arise for linear and max-linear factor models with heavy tailed factors. The stable tail dependence function is then parameterized by a $d \times K$ matrix $A$, where $K$ is the number of factors and where $A$ can be interpreted as a factor loading matrix. We study estimation of $L$ under an additional assumption on $A$ called the `pure variable assumption'.
Both $K \in \{1, \dots, d\}$ and $A \in [0, \infty)^{d \times K}$ are treated as unknown, which constitutes an unconventional parameter space that does not fit into common estimation frameworks. We suggest two algorithms that allow to estimate $K$ and $A$, and provide finite sample guarantees for both algorithms. Remarkably, the guarantees allow for the case where the dimension $d$ is larger than the sample size $n$. The results are illustrated with numerical experiments and two case studies. 
\end{abstract}


\smallskip



\section{Introduction}

Extreme value statistics is concerned with analyzing and modeling rare, extreme events, with vital applications in fields like finance, insurance, hydrology, and meteorology, among others \citep{Kat02, Bei04}. 
A primary challenge is understanding tail dependence, i.e, the interconnectedness of extreme events, which is particularly complex in high dimensional settings. 
To address this complexity, recent research aimed at integrating modern statistical and machine learning techniques into the field, with a particular focus on dimension reduction \citep{ chautru2015dimension, goix2017sparse, cooley2019decompositions, drees2021principal}, clustering \citep{janssen2020k, meyer2024multivariate, medina2024spectral} or (sparse) parametric modeling and estimation in high dimensions \citep{engelke2021learning}; see also \cite{Eng21} for a recent review.

This work goes in a similar direction, and relates to several of the aforementioned concepts. The key assumption is semiparametric: we assume that, after (nonparametric) marginal standardization, the tail dependence structure of an observable random vector $\bm X$ coincides with that obtained from a heavy-tailed linear factor model; this approach is akin to factor copula modeling in general dependence modeling \citep{OhPat17}. More specifically, the heavy-tailed linear factor model relies on the assumption that a random vector $\bm{Y} =(Y_1,\dots,Y_d)^\top$ either satisfies $Y_j = \sum_{a=1}^K A_{ja}Z_a + E_j$ or $Y_j = (\bigvee_{a=1}^K A_{ja}Z_a)\vee E_j$ for all $j\in\{1,\dots,d\}$, for some factor number $K\in \N$, some nonnegative factor loading matrix $A=(A_{ja}) \in [0,\infty)^{d\times K}$, some asymptotically independent, heavy-tailed random variables $Z_a$ called factors, and some random noise variables $E_j$ whose tails are lighter than the ones of the factors. 
Building on \cite{Wan11} and \cite{einmahl2012mestimator}, the model has been used in various contexts in extremes, with notable recent advances in graphical modeling \citep{gissibl2018tail, kluppelberg2021estimating}.
Beyond extreme value analysis, linear factor models have been widely used across scientific disciplines for nearly a century, generating extensive literature.
We refer to \cite[Chapter 15]{izenman2008modern} for a survey and applications. 

The semiparametric model assumption mentioned in the the previous paragraph is equivalent to a specific parametric assumption on the stable tail dependence function $L$ of $\bm X$. In the case where $K$ is assumed to be known, the parameter of $L=L_{\bar A}$ is a $d \times K$ loading matrix $\bar A$ which has non-zero coloumn sums and row sums equal to 1. 
Estimating that matrix turns out to be challenging, partly because of the lack of differentiability of $L_{\bar A}$. Existing approaches are based on minimizing suitable distances between $L_{\bar A}$ and the empirical stable tail dependence function $\hat L$ \citep{einmahl2012mestimator, einmahl2018continuous}, on connecting clustering methods to the parametric model \citep{janssen2020k, medina2024spectral} or on exploiting a relationship of $\bar A$ to the tail pairwise dependence matrix \citep{kiriliouk2022estimating}. However, all these methods rely on the assumption that $K$ is known, which is hardly the case in practice. 
Moreover, the methods often lack statistical guarantees,
in particular for the case where the dimension $d$ (and possibly the unknown number of factors $K$) is large and potentially exceeding the (effective) sample size; such situations may for instance arise in spatial or financial applications (e.g., when assessing tail dependence between the $d=500$ constituents of the S\&P500-index). These limitations are ongoing challenges in the field, and this paper offers partial steps toward addressing them on mathematical, modeling, and methodological fronts.

The contributions of this work are the following. 
Inspired by recent advances on statistical guarantees for high-dimensional linear factor models in the non-extreme setting in \cite{Bin20},
we prove identifiability of both the number of latent factors $K$ and the loading matrix $\bar{A}$ from pairwise extremal dependence measures under a condition called the ``pure variable assumption'', with uniqueness up to permutation. Building on this theoretical foundation, we develop a tailored estimator $\triangleA$ for the structured matrix $\bar{A}$ that depends on 
three interpretable tuning parameters: a threshold-exceedance number $k \in \mathbb{N}_{\geq 2}$, a latent factor control parameter $\hyperpurevar \in (0,1/2)$, and a sparsity parameter $\hypersparsity\in (0,1/2)$, with $\hypersparsity = \hyperpurevar$ being a reasonable choice.
As an important by-product, the methods allow for estimating extremal directions \citep{Sim20} and, hence, for soft-clustering of the variables.
Our estimation approach mirrors the constructive techniques used in our identifiability proof, and we provide non-asymptotic bounds ``in probability'' for both the exact recovery of $K$ and the estimation error between $\triangleA$ and the true $\bar{A}$. Moreover, we prove almost sure large-sample consistency for the case where $d=d_n$, $K = K_n$ may depend on $n$ and actually be larger than $n$, with properly specified tuning parameters. Extensive simulations confirm our theoretical results and showcase the estimator's practical performance. Finally, two case studies illustrate our methods: one where the model fits well, and another revealing potential limitations.

The remainder of this paper is organized as follows. Section \ref{subsec:tail-dependence} provides mathematical preliminaries on tail dependence and multivariate regular variation, followed by an extensive discussion of the tail dependence properties induced by a linear factor model in Section~\ref{subsec:linear-factor-model}. In Section \ref{sec:statistics-stdf-pure-variable}, we propose algorithms for estimating the model parameters and provide statistical guarantees for both finite samples and asymptotic regimes. The results of a Monte Carlo simulation study and two case studies are presented in Section \ref{sec:simulations} and \ref{sec:case_studies}, respectively, followed by a conclusion and outlook in Section~\ref{sec:conclusion}. 
All proofs are deferred to Sections \ref{sec:proofs}-\ref{sec:additional}, and some additional numerical results are presented in Sections~\ref{sec:additional-simulation-results} and \ref{sec:supplement_case_studies}.

\section{Mathematical Preliminaries}
\label{sec:preliminaries}

\subsection{Tail dependence and regular variation on the non-negative orthant}\label{subsec:tail-dependence}

Tail dependence of a $d$-variate random vector $\bm X=(X_1, \dots, X_d)^\top$ with continuous marginal cdfs $F_1, \dots, F_d$ is typically assessed by a standardization of the margins. A key object is the \emph{stable tail dependence function}  $L:[0,\infty)^d \to [0, \infty)$, which is defined by
\begin{align} 
\label{eq:def-stable-tail}
        L(\bm{x}) &= \lim_{n \to \infty} n \mathbb{P}\left\{ F_1(X_1) > 1-\frac{x_1}{n} \textrm{ or } \dots \textrm{ or } F_d(X_d) > 1-\frac{x_d}{n} \right\},
\end{align}
provided the limit exists (which we implicitly assume throughout this paper). Often, attention is restricted to simple bivariate coefficients like the \emph{tail correlation between $X_j$ and $X_\ell$}, which is defined, for $j, \ell \in [d]$, by
\begin{align*} 
\chi(j,\ell) = \lim_{n \to \infty} n \mathbb{P}\left\{ F_j(X_j)>1-1/n, F_\ell(X_\ell)>1-1/n\right\} = \begin{cases} 2 - L(\bm e_j + \bm e_\ell), & j \ne \ell, \\ 1, & \text{else}, \end{cases}
\end{align*}
where $\bm e_j$ denotes the $j$th unit vector in $\R^d$. 
These coefficients will be collected in a matrix $\mathcal X=(\chi(j,\ell))_{j,\ell\in[d]} \in [0,1]^{d \times d}$; here and throughout, we write $[n]=\{1, \dots, n\}$ for an integer $n$.

Existence of the stable tail dependence function $L$ in \eqref{eq:def-stable-tail} can equivalently be expressed through the concept of regular variation on the non-negative orthant $[0,\infty)^d$ of $\mathbb{R}^d$, as formally stated in Lemma~\ref{lemma:stable-tail-dependence-spectral-measure} below. Write $\mathbb{E}_0 = [0,\infty)^d \setminus \{\bm{0}\}$, which is a Polish space.  
Let $\mathbb{M}(\mathbb{E}_0)$ be the set of Borel measures on $\mathbb{E}_0$ which are finite on sets bounded away from zero. We denote by $\mathcal{C}(\mathbb{E}_0)$ the set of continuous, bounded, positive functions on $\mathbb{E}_0$ whose supports are bounded away from zero. For $\nu_n, \nu \in \mathbb{M}(\mathbb{E}_0)$, we say $\nu_n \rightarrow \nu$ in $\mathbb{M}(\mathbb{E}_0)$ if $\int f d \nu_n \rightarrow \int f d\nu$ for all $f \in \mathcal{C}(\mathbb{E}_0)$ (see Section 2 in \citealp{hult2006regular}, Section 1.4.1.2 in \citealp{resnick2024art} or Definition B.1.16 in \citealp{kulik2020heavy}, where the convergence is referred to as vague$^\#$-convergence with the boundedness consisting of all subsets of $\mathbb{E}_0$ that are bounded away from zero). A random vector $\bm{Y} \in [0,\infty)^d$ is regularly varying  on $\mathbb{E}_0$ if there exists a sequence $c_n \rightarrow \infty$ and a non-degenerate measure $\nu_{\bm{Y}} \in \mathbb{M}(\mathbb{E}_0)$, called the \emph{exponent measure associated with $(c_n)_n$}, such that, as $n \rightarrow \infty$,
\begin{equation}
    \label{eq:regular_variation}
    n \mathbb{P}\left\{ c_n^{-1}\bm{Y} \in \cdot \right\} \rightarrow \nu_{\bm{Y}}(\cdot), \quad \textrm{in } \mathbb{M}(\mathbb{E}_0)
\end{equation}
(Definition B.2.1 in \citealp{kulik2020heavy}). In this case, Theorem B.2.2 in \cite{kulik2020heavy} implies that
\begin{compactenum}[(a)]
    \item there exists an $\alpha>0$, called the tail index of $\bm{Y}$, such that $(c_n)_n$ is regularly varying of index $\alpha^{-1}>0$ and such that the limit measure $\nu_{\bm{Y}}$ satisfies the homogeneity property $\nu_{\bm{Y}}(tB) =t^{-\alpha} \nu_{\bm{Y}}(B)$ for all constants $t > 0$ and all Borel sets $B \subset \mathbb{E}_0$ bounded away from zero;
    \item the scaling sequence $(c_n)_n$ and the associated exponent measure $\nu_{\bm Y}$ are unique up to a factor in the sense that, if \eqref{eq:regular_variation} holds for another choice $(c_n', \nu_{\bm Y}')$, then there exists a constant $\xi>0$ such that $\nu_{\bm Y}= \xi \nu_{\bm Y}'$ and $\lim_{n \to \infty}c_n'/c_n= \xi^{1/\alpha}$;
    \item the above definition of regular variation is equivalent to the one given in Definition  2.1 in \cite{resnick2024art}.
\end{compactenum}
Next, fix a norm $\| \cdot\|$ on $\R^d$, let $\mathbb{S}_+^{d-1} = \{ \bm{y} \in [0, \infty)^d : \|\bm{y}\| =1\}$ and consider the bijective transform $T : \mathbb{E}_0 \mapsto (0,\infty) \times \mathbb{S}_+^{d-1}$ defined by
\[
    T(\bm{y}) = \Big(\| \bm{y} \|, \frac{\bm{y}}{\|\bm{y}\|} \Big),
\]
whose inverse $T^{-1} : (0,\infty) \times \mathbb{S}_+^{d-1} \mapsto \mathbb{E}_0$ satisfies $T^{-1}(r,\bm{\lambda}) = r \bm{\lambda}$. The homogeneity of $\nu_{\bm X}$ implies that, for any $r>0$ and any Borel set $A$ in $\mathbb{S}_+^{d-1}$,
\begin{align} \label{eq:relation-exponent-spectral-measure}
\nu_{\bm Y} \circ T^{-1} \big( (r,\infty) \times A \big)
&=
\nu_{\bm Y} \Big( \Big\{ \bm y \in \mathbb E_0: \| \bm y \| > r, \frac{\bm y}{\| \bm y\|} \in A \Big\} \Big)
\\&=\nonumber
r^{- \alpha}\nu_{\bm Y} \Big( \Big\{ \bm y \in \mathbb E_0: \| \bm y \| > 1, \frac{\bm y}{\| \bm y\|} \in A \Big\} \Big)
= \varsigma \cdot (\nu_\alpha \otimes \Phi_{\bm Y}) \big((r,\infty) \times   A \big),
\end{align}
where $\varsigma$ is the constant $\varsigma = \varsigma^{\| \cdot\|}  = \nu_{\bm Y} ( \{ \bm y \in \mathbb E_0: \| \bm y \| > 1 \})$,
$\nu_\alpha$ is the measure on $(0,\infty)$ defined by $\nu_\alpha(dr) = \alpha r^{-\alpha-1} dr$, 
\begin{align}
\label{eq:def-spectral-measure-new}
\Phi_{\bm Y}(A) := \Phi_{\bm Y}^{\|\cdot\|}(A)  := \varsigma^{-1} \nu_{\bm Y} \Big( \Big\{ \bm y \in \mathbb E_0: \| \bm y \| > 1, \frac{\bm y}{\| \bm y\|} \in A \Big\} \Big)
\end{align}
is the \emph{spectral measure of $\bm Y$ associated with the norm $\| \cdot \|$}, and $\otimes$ is the product measure. Note that the spectral measure is a probability measure on $\mathbb{S}_+^{d-1}$, and that, by property (a) above, it does not depend on the sequence $(c_n)_n$; see also Theorem 2.2.1 in \cite{kulik2020heavy}.

The claimed equivalence at the beginning of the previous paragraph can now be formally stated as follows; see also Chapter 6.1 in \cite{Res07} or Chapter 5.4.2 in \cite{Resnick1987} for similar statements based on a slightly different definition of regular variation.
For completeness, we give a proof in the appendix.

\begin{lemma} \label{lemma:stable-tail-dependence-spectral-measure}
Suppose $\bm X$ is a $d$-variate random vector with continuous marginal cdfs $F_1, \dots, F_d$. Then, the stable dependence function $L$ of $\bm X$ exists if and only if the random vector $\bm Y=(Y_1, \dots, Y_d)^\top$ defined by $Y_j=1/(1-F_j(X_j))$ is regularly varying. In that case:
\begin{compactenum}[(i)]
    \item the index of regular variation of $\bm Y$ is $1$;
    \item the scaling sequence $(c_n)_n$ can be chosen as $c_n=n$, and the exponent measure $\nu_{\bm Y}$ associated with $(c_n)_n$ is related to $L$ by $L(\bm x) = \nu_{\bm Y}([0, \bm 1/ \bm x]^c)$ for all $\bm x \in [0, \infty)^d$, where $\bm 1/\bm x$ has coordinates $1/x_j$;
    \item the spectral measure $\Phi_{\bm Y}$ of $\bm Y$ associated with the $\| \cdot \|$-norm is related to $L$ by
        \begin{align*} 
        L(\bm x) 
        = \varsigma
         \int_{\mathbb S_{+}^{d-1}} \max_{j \in [d]} \big( \lambda_j x_j \big)\, \diff  
        \Phi_{\bm Y}(\bm \lambda), \qquad \bm x \in [0,\infty)^d,
    \end{align*}
    with $\varsigma=\lim_{y \to \infty} y \Prob(\| \bm Y \|> y) = \nu_{\bm Y}(\{ \bm y : \| \bm y \| > 1 \})$, where $\nu_{\bm Y}$ is from (ii). For the $\| \cdot \|_\infty$-norm, we have $\varsigma =L(\bm 1)$, and for the $\| \cdot \|_1$-norm, we have $\varsigma=d$. Moreover, $\Phi_{\bm Y}$ is uniquely determined by $L$. 
    \end{compactenum}
\end{lemma}

Finally, inspired by \cite{Sim20} and \cite{Mou25}, a non-empty subset $D \subset [d]$ will be called an \textit{extremal direction of $\bm X$} if the spectral measure $\Phi_{\bm Y}$ of $\bm Y$ from Lemma~\ref{lemma:stable-tail-dependence-spectral-measure} satisfies $\Phi_{\bm Y}(\Sb_D)>0$, where $\Sb_D = \{ \bm x \in \Sb_{+}^{d-1}: x_j>0 \text{ iff } j \in D\}$.

\subsection{Linear factor models and the pure variable assumption}
\label{subsec:linear-factor-model}

For $K\in \N$, let $\Acp(K)$ and $\Arp(K)$ denote the set of matrices in $[0,\infty)^{d \times K}$ with non-zero column sums and non-zero row sums, respectively, and define $\Ap(K) = \Acp(K) \cap \Arp(K)$. Moreover, write $A_{\cdot a} \in \R^d$ for the $a$th column of $A$ and $A_{j\cdot} \in \R^K$ for the $j$th row of $A$.
The central model assumption in this paper is the following specific parametric model for the stable tail dependence function.

\begin{condition}[Max-linear Stable Tail Dependence Function]
\label{cond:stdf-max-linear}
The random vector $\bm X=(X_1, \dots, X_d)^\top$  has continuous marginal distribution functions $F_1, \dots, F_d$ and its stable tail dependence function $L$ exists and satisfies
\begin{align} \label{eq:stable-tail-dp-A}
    L(\bm x) = L_{K, \bar A}(\bm x) = \sum_{a \in [K]} \bigvee_{j \in [d]} \bar A_{ja} x_j, \qquad \bm x \in [0, \infty)^d,
\end{align}
for some unknown parameter $\theta = (K, \bar A) \in \Theta_L \subset \Theta$, where $\Theta=\big\{ (K,\bar A) \mid K \in [d], \bar A \in \Ap(K) \text{ has row sums } 1\big\}$ (the set $\Theta_L$ will be specified below). 
\end{condition}

In the remaining parts of this section, we will examine the model assumption in detail. In particular, we show that $L$ is indeed a stable tail dependence function, and that its parameter matrix $\bar A$ may be interpreted as a factor loading matrix in certain heavy tailed (max-) linear models. Note that this is akin to the interpretation of factor copula models \citep{OhPat17}. 

We start by recalling linear and max-linear factor models, and then show that their stable tail dependence function is of the form given in \eqref{eq:stable-tail-dp-A}.

\begin{model}[Linear Factor Model with Noise] \label{cond:factormodel}
The random vector $\bm Y \in [0,\infty)^d$ is said to follow a linear factor model with noise if there exists $\alpha>0$, $K \in \N$ and $A \in [0,\infty)^{d\times K}$ such that
\begin{align} \label{eq:factormodel}
\bm Y =_d A \bm Z + \bm E
\end{align}
for some random vectors $\bm Z=(Z_1, \dots, Z_K)^\top \in [0, \infty)^K$ and $\bm E=(E_1, \dots, E_d)^\top \in [0, \infty)^d$ that satisfy
\begin{compactenum}[(i)]
    \item $Z_1, \dots, Z_K$ are identically distributed, regularly varying of order $\alpha>0$, i.e., the exists a scaling sequence $c_n\to\infty$ such that, for all $a \in [K]$, $n \mathbb{P}\left\{ c_n^{-1} Z_a \in \cdot \right\} \rightarrow \nu_\alpha(\cdot)$ in $\mathbb{M}((0,\infty))$. 
    Moreover, $Z_1, \dots, Z_K$ are pairwise asymptotically independent, meaning that, for $a,b\in[K]$ with $a \ne b$
    \[
        n \mathbb{P}\left\{ Z_a > c_n x_a, Z_b > c_n x_b  \right\} \rightarrow 0, \quad \textrm{ for all } x_a, x_b > 0.
    \]
    \item $\bm E$ has a lighter tail than $\bm{Z}$, i.e., $\mathbb{P}\{ \|\bm{Z}\| > x \} / \mathbb{P}\{ \|\bm{E}\| > x \} = o(1)$ as $x \rightarrow \infty$.
\end{compactenum}
\end{model}

\begin{remark}\label{rem:fix-scale}
Under the above model assumption, the matrix $A$ is clearly not identifiable (not even up to column permutations). Indeed, if \eqref{eq:factormodel} holds with $\bm Z$ and $\bm E$ as in (i) and (ii), then we also have $\bm X=_d A' \bm Z' + \bm E$ with $A'=2A$ and $\bm Z'=\bm Z/2$, and $\bm Z'$ and $\bm E$ also satisfy (i) and (ii). To overcome this issue, the scale of $Z_i$ needs to be fixed. The following assumption is sufficient:

\smallskip
\begin{compactenum}[(i)']
    \item $Z_1, \dots, Z_K$ are identically distributed with tail function $\bar F_{Z}$ satisfying $\bar F_Z(x) = \ell(x) x^{-\alpha}$ for some function $\ell$ for which $\lim_{x \to \infty} \ell(x)=1$.
    Moreover, $Z_1, \dots, Z_K$ are pairwise asymptotically independent.     \label{cond_(i)'}
\end{compactenum}
\smallskip

\noindent Note that assuming $\bar F_Z(x) = \ell(x) x^{-\alpha}$ for some slowly varying function $\ell$ is equivalent to regular variation of $Z_i$ with index $\alpha$. Functions whose limit $\lim_{x\to\infty} \ell(x)$ exist are always slowly varying; in that sense, (i)' implies (i), but not vice versa. By fixing the limit to one, we essentially fix the scale of $Z_i$ as required.
\end{remark}

Slightly rephrasing \cite[Proposition 2.1.8]{kulik2020heavy} or \cite[Proposition 2.2]{resnick2024art}, the assumption on $\bm{Z}$ in (i) implies that $\bm{Z}$ is regularly varying with tail index $\alpha$ and exponent measure $\nu_{\bm{Z}}$ associated with $(c_n)_n$ given by
\begin{equation}
    \label{eq:Lambda_z}
    \nu_{\bm{Z}} = \sum_{a \in [K]} \delta_0^{\otimes(a-1)} \otimes \nu_\alpha \otimes \delta_0^{\otimes(K-a)},
\end{equation}
where $\delta_0$ is the Dirac measure at zero, where $\otimes$ is the product measure, and where $\mu^{\otimes n}$ is the $n$-fold product measure of $\mu$ with itself (for $n = 0$, the factor is simply omitted from the previous display).

\begin{model}[Max-Linear Factor Model with Noise] \label{cond:max-factormodel} 
The random vector $\bm Y \in [0,\infty)^d$ is said to follow a max-linear factor model if there exists $\alpha>0$, $K \in \N$ and $A \in [0,\infty)^{d\times K}$ such that
\begin{equation*}
    \bm Y =_d (A \times_{\max} \bm Z)\vee \bm{E} := \left( \Big(\max_{a=1}^K A_{1a} Z_a\Big) \vee E_1, \dots, \Big(\max_{a=1}^K A_{da} Z_a\Big) \vee E_d \right)^\top,
\end{equation*}
for some random vectors $\bm Z=(Z_1, \dots, Z_K)^\top \in [0, \infty)^K$ and $\bm E=(E_1, \dots, E_d)^\top \in [0, \infty)^d$ that satisfy (i) and (ii) from Model~\ref{cond:factormodel}.
\end{model}

Both Model~\ref{cond:factormodel} and Model~\ref{cond:max-factormodel} define regular varying random vectors whose stable tail dependence function is as in \eqref{eq:stable-tail-dp-A}.

\begin{proposition}
\label{prop:spectral_measure_max_lin}
Suppose $\bm Y$ is either as in Model~\ref{cond:factormodel} or as in Model~\ref{cond:max-factormodel}, with $A \in \Acp(K)$. 
Then $\bm Y$ is regularly varying of index $\alpha$ with spectral measure
\begin{align} \label{eq:Phi_A}
\Phi_{\bm{Y}}(\cdot) 
= 
\Phi_{A,\alpha} (\cdot)
= 
\varsigma^{-1} \sum_{a\in[K]} \| A_{\cdot a} \|^\alpha \delta_{A_{\cdot a}/ \| A_{\cdot a}\|}(\cdot),
\end{align}
where $\varsigma=\varsigma_{A, \alpha} =\sum_{a \in [K]} \| A_{\cdot a}\|^\alpha$.
If, additionally, $\bm Y$ has continuous marginal distribution functions $F_1,\dots,F_d$ and if $A \in \Arp(K)$, writing $\bar A = f(A, \alpha)$ for the matrix with entries
\begin{align} \label{eq:barA}
    \bar A_{ja} = \frac{A_{ja}^\alpha}{\sum_{b=1}^K A_{j b}^\alpha} =  \frac{A_{ja}^\alpha}{s_{j}^\alpha}, \qquad s_j=s_j(\alpha, A) := \Big(\sum_{a\in[K]} A_{ja}^\alpha\Big)^{1/\alpha}
\end{align}
(note that $\bar A$ has row sums 1),
the stable tail dependence function of $\bm Y$ satisfies $L =L_{K, \bar A}$ with $L_{K, \bar A}$ from \eqref{eq:stable-tail-dp-A}.
\end{proposition}
 
Among other things, Proposition~\ref{prop:spectral_measure_max_lin} shows that the parameter matrix $\bar A$ from Condition~\ref{cond:stdf-max-linear} may indeed by interpreted as a factor loading matrix for tail dependence in the following sense: the same tail dependence function $L_{K, \bar A}$ also arises in a (max-)linear factor model as in Model~\ref{cond:factormodel} or \ref{cond:max-factormodel}, with factor loading matrix $A=\bar A$ and with tail index $\alpha=1$. This follows from the fact that $\bar A=f(A, \alpha)$ defined in \eqref{eq:barA} satisfies $\bar A = f(\bar A,1)$.

We collect some further properties that can be deduced from Condition~\ref{cond:stdf-max-linear}.

\begin{lemma} \label{lem:tail-probabilities-maxlinear-stdf}
Suppose $\bm X$ satisfies Condition~\ref{cond:stdf-max-linear}, let $\bm x \in [0,\infty)^d$ and let $\mathcal J$ be a non-empty set of non-empty subsets of $[d]$. Then
\begin{align*}
R^\cup_{\mathcal J}(\bm x) \equiv  \lim_{n \to \infty} 
n
\mathbb P \Big (\exists\, J \in \mathcal J\,  \forall j \in J: F_j(X_j) > 1-\frac{x_j}n \Big)
&=
\sum_{a \in [K]} \bigvee_{J \in \mathcal J} \bigwedge_{j \in J} \bar A_{ja}{x_j},
\\
L^\cap_{\mathcal J}(\bm x) \equiv  \lim_{n \to \infty} 
n
\mathbb P \Big (\forall\, J \in \mathcal J\,  \exists j \in J: F_j(X_j) > 1-\frac{x_j}n \Big)
&=
\sum_{a \in [K]} \bigwedge_{J \in \mathcal J} \bigvee_{j \in J} \bar A_{ja}{x_j}.
\end{align*}
In particular, for all $\varnothing \ne J \subset [d]$, the tail copula and the stable tail dependence function of the subvector $\bm X_J=(X_j)_{j \in J}$ exist and satisfy
\begin{align*}
R_J(\bm x_J)  
&\equiv 
\lim_{n \to \infty} n \Prob\Big( \forall j \in J: F_j(X_j) > 1-\frac{x_j}n \Big) =  \sum_{a \in [K]}  \bigwedge_{j \in J} \bar A_{ja} x_j,
\\
L_J(\bm x_J) 
&\equiv 
\lim_{n \to \infty} n \Prob\Big( \exists j \in J: F_j(X_j) > 1-\frac{x_j}n \Big) = 
\sum_{a \in [K]}  \bigvee_{j \in J} \bar A_{ja} x_j.
\end{align*}
Moreover,  for each $j,\ell \in [d]$ with $j \ne \ell$, the tail correlation between $X_j$ and $X_\ell$ satisfies
\begin{align} \label{eq:chiA}
        \chi(j,\ell) = 
        R_{\{j,\ell\}}(1,1) = \sum_{a\in[K]} \bar{A}_{ja} \wedge \bar{A}_{\ell a}, \qquad  j \ne \ell.
\end{align}
Finally, $\bm X$ has at most $K$ extremal directions, and these are given by $D_1, \dots, D_K$, where
\begin{equation} \label{eq:softcluster}
D_a := D_a(A) := \big\{ j \in [d]: A_{ja}>0 \big\} \subseteq [d], \qquad a \in [K].
\end{equation}
\end{lemma}

In the remaining part of this section, we introduce an additional assumption on the matrix $\bar A$ called \emph{the pure variable assumption} under which we can `invert the relationship in \eqref{eq:chiA}'. More precisely, the number of columns of $\bar A$ (i.e., the number of factors) and the matrix $\bar A$ from \eqref{eq:barA} will be shown to be identifiable from the tail correlation matrix, which will be the basis for the estimators studied in Section~\ref{sec:statistics-stdf-pure-variable}. The resulting model's flexibility and the corresponding restrictiveness of the assumption is examined in Section~\ref{sec:case_studies} via empirical applications. These applications also motivate potentially more flexible, though mathematically more challenging, extensions that shall be studied in future work.

\begin{condition}[Pure Variable Assumption] \label{cond:purevar}
A matrix $A \in [0,\infty)^{d \times K}$ is said to satisfy the \textit{pure variable assumption} if,
for any factor index $a\in[K]$, there exists a variable index $j \in [d]$ such that $A_{ja}>0$ and $A_{jb}=0$ for all $b \ne a$.  
In other words, we have
\begin{align} \label{eq:Ia}
I_a := I_a(A) := \big\{ j \in [d]: A_{ja}>0 \text{ and } A_{jb}=0 \ \forall b \ne a \big\} \ne \emptyset.
\end{align}
The set of all matrices that satisfy the pure variable assumption and have positive row sums is denoted by $\Apure(K) \subset \Arp(K) \cap \Acp(K)$, and $I=I(A) := \bigcup_{a \in [K]} I_a$ denotes the \textit{pure variable set} and $J=[d] \setminus I$ the \textit{non-pure variable set}.
\end{condition}

Note that the pure variable assumption implies that $K\le d$. 
To formalize the identifiability statement from the above, let
\begin{align} \label{eq:def-phi-l}
     \Phi_L: &\Theta_L \to [0,1]^{d \times d}, \qquad  (K, \bar A) \mapsto \mathcal X = \Big(\sum_{a\in[K]} \bar A_{ja} \wedge \bar A_{\ell a} \Big)_{j,\ell \in [d]},
\end{align}
where
\begin{align}
\Theta_{L} = \Theta_L(d) 
    &= \label{eq:Theta-L}
    \big\{ (K,\bar A) \mid K \in [d], \bar A \in \Apure(K) \text{ has row sums } 1\big\}.
\end{align}
The following lemma shows that $\Phi_L$ is injective. In statistical terms, the parameter $\theta \in \Theta_L$ is identifiable from its image under $\Phi_L$, which will eventually allow for its estimation. Two matrices $A^{(1)}, A^{(2)} \in \mathbb{R}^{d \times K}$ will be called equivalent, notation $A^{(1)} \sim A^{(2)}$, if $A^{(1)} = A^{(2)}P$ for some permutation matrix $P \in \R^{K \times K}$.

\begin{proposition}[Identifiability]
\label{prop:injective-new}
The function $\Phi_{L}$ is injective up to column permutations of $ \bar A$, that is, $\Phi_{L}(K^{(1)}, \bar A^{(1)}) = \Phi_{L}(K^{(2)}, \bar A^{(2)})$ implies $K^{(1)}= K^{(2)}$ and $\bar {A}^{(1)} \sim \bar {A}^{(2)}$. 
\end{proposition}

The previous proposition is a consequence of the following result, a sample version of which will be the basis for the \textsc{PureVar}-algorithm introduced in Section~\ref{sec:statistics-stdf-pure-variable}.

\begin{proposition}
\label{prop:pur_var_ide-new2} 
Fix $d \in  \N$ and suppose that $\Xi =(\xi_{j,\ell})_{j,\ell=1}^d$ is an extremal correlation matrix. Recall $\Phi_L$ from \eqref{eq:def-phi-l}, and define
\begin{align*}
M_\Xi = \big\{\theta = (K, \bar A) \in \Theta_L \mid \Xi = \Phi_L(K, \bar A) \big\}
\end{align*}
as the set of parameters which are compatible with the tail correlation matrix $\Xi$. 
Let $G=G(\Xi) = (V,E)$ denote the undirected graph with vertex set $V=[d]$ and edge set $E$ consisting of all $(j,\ell)$ such that $\xi_{j,\ell}=0$. 
Let $C$ be a maximum clique in $[d]$ (that is, there does not exist any clique with more vertices), and write $L=|C|$.
Then, for any $\theta =(K, \bar A)\in M_\Xi$:
\begin{compactenum}[(a)]
\item We have $K=L$.
\item The clique $C$ only contains pure variables of $\bar A$.
\item For any $a\in[K]$ with $|I_a(\bar A)| \ge 2$ and any $j\in I_a(\bar A)$ we have, for all $\ell \ne j$,
\[
\ell \in I_a(\bar A) \Longleftrightarrow \xi_{j,\ell} =1.
\]
\item 
The set of pure variables $I=I(\bar A)$ can be uniquely determined from $\Xi$.
Moreover, up to label permutations, the partition $\{I_1(\bar A), \dots, I_K(\bar A)\}$  can be uniquely determined from $\Xi$.
\end{compactenum}
\end{proposition}

\section{Estimating max-linear stable tail dependence functions under a pure variable assumption}
\label{sec:statistics-stdf-pure-variable}

In this section, we assume that we observe an i.i.d.\ sample $\bm X_1, \dots, \bm X_n$ from a random vector $\bm X$ satisfying Condition~\ref{cond:stdf-max-linear}, that is, $\bm X$ has  continuous marginal cdfs and its stable tail dependence function $L$ is of the form given in \eqref{eq:stable-tail-dp-A}:
\[
L(\bm x) = L_{K, \bar A}(\bm x) = \sum_{a \in [K]} \bigvee_{j \in [d]} \bar A_{ja} x_j, \quad \bm x \in [0, \infty)^d
\]
for some unknown parameter $\theta_L = (K, \bar A) \in \Theta_L$, where the parameter space $\Theta_L=\big\{ (K,\bar A) \mid K \in [d], \bar A \in \Apure(K) \text{ has row sums } 1\big\}$ is as in \eqref{eq:Theta-L}; i.e., $\bar A$ satisfies Condition~\ref{cond:purevar}.
This setup provides a tail-dependence model that is free of any marginal assumptions (except for continuity), and it can be considered a special instance of the semiparametric models considered in \cite{einmahl2012mestimator}, but with a non-standard parameter space that does not fit into their framework.

Our goal is to define estimators for $K$ and $\bar A$, and to equip them with finite-sample statistical guarantees in arbitrary dimensions. We proceed in two steps, each of which is dealt with in a separate section.

\subsection{The PureVar-Algorithm}

We start by describing how to estimate $K, I$ and $\mathcal I=\{ I_1, \dots, I_K\}$. Inspired by the result in Proposition~\ref{prop:injective-new}, the method depends on the estimated tail correlation matrix, $\hat{\mathcal X}=(\hat{\chi}_{n, \knon}(j,\ell))_{j,\ell \in [d]}$, and a small threshold $\hyperpurevar \in (0,1/2)$ whose choice will be specified later. Here,  $\hat{\chi}_{n, \knon}(j,\ell)$ denotes the empirical extremal correlation, which is defined by
\begin{align}
\label{eq:empirical_correlation}
    \hat{\chi}_{n, \knon}(j,\ell) = \frac1{\knon} \sum_{i=1}^n  \bm 1 ( R_{i,j}>n-\knon, R_{i,\ell}>n-\knon ), \qquad j, \ell\in[d],
\end{align}
where $\knon \in [n]$ is a hyperparameter balancing bias and variance (in an asymptotic framework, it is typically required to satisfy $k_n \to \infty$ and $k_n/n = o(1)$ as $n \to \infty$) and $R_{i,j}$ denotes the rank of $X_{i,j}$ among $X_{1,j}, \dots, X_{n,j}$, breaking ties at random if necessary. Note that $\hat\chi_{n,\knon}(j,j)=1$ for all $j\in[d]$. 
Now, inspired by the population level result in Proposition~\ref{prop:pur_var_ide-new2}, we propose to proceed as follows: first, construct the graph $\hat G=(V,E)$ with $V=[d]$ and edge set $\hat E$ consisting of all $(j,\ell)$ such that $\hat{\chi}_{n, \knon}(j,\ell) \le \hyperpurevar$. Then, find a maximum clique $\hat C$ in that graph, and define $\hat K = |\hat C|$. Moreover, for any $j \in \hat C$, define $\hat I_j = \{ j \} \cup \{\ell \in [d]: \hat{\chi}_{n, \knon}(j, \ell) \ge 1-\hyperpurevar\}$, and let $\hat{\mathcal I} = \{ \hat I_1, \dots, \hat I_{\hat K}\}$. The method is summarized in Algorithm~\ref{alg:purevar}.

\begin{algorithm}
\caption{(Estimating $K, I$ and $\mathcal I$)} \label{alg:purevar}
\begin{algorithmic}[1]
\Require{$\hat {\mathcal X} \in[0,1]^{d \times d}, \hyperpurevar \in (0,1/2)$}
\Function{PureVar}{$\hat{\mathcal X}, \hyperpurevar$}
    \State{Construct the graph $\hat G=(V,\hat E)$ with $V=[d]$ and edge set $\hat E=\{ (j,\ell): \hat{\chi}_{n, \knon}(j,\ell) \le \hyperpurevar \}$}
    \State{Find a maximum clique $\hat C\subseteq [d]$, define  $\hat K=|\hat C|$, and write $\hat C=\{\hat j_1, \dots, \hat j_{\hat K}\}$}
    \For{$a = 1, \ldots, \hat K$} 
        \State{$\hat I_a \gets \{ \hat j_a \} \cup \{\ell \in [d]: \hat{\chi}_{n, \knon}(\hat j_a, \ell) \ge 1-\hyperpurevar\}$}
    \EndFor
    \State{$\hat{\mathcal I} \gets \{ \hat I_1, \dots, \hat I_{\hat K}\}$}
    \State{$\hat{I} \gets \hat I_1 \cup \dots \cup \hat I_{\hat K}$}
    \State{\Return $(\hat K, \hat I, \hat{\mathcal I})$}
    \EndFunction
\end{algorithmic}
\end{algorithm}

\begin{remark}
    Note that the maximum clique in Line 3 of Algorithm~\ref{alg:purevar} is not unique in general, meaning that a (possibly random) choice has to be made. For instance, in the case $d=2$ with $\hat \chi_{n,\knon}(1,2)=1/2$ and $\hyperpurevar=1/4$, both $\hat C=\{1\}$ and $\hat C=\{2\}$ are possible choices for the maximum clique. The former choice results in $\textsc{PureVar}(\hat {\mathcal X}, \hyperpurevar)=(1,\{1\}, \{\{1\}\})$, while the latter results in $\textsc{PureVar}(\hat {\mathcal X}, \hyperpurevar)=(1,\{2\}, \{\{2\}\})$. In the results to follow, we will show that the choice does not matter, provided that $\hyperpurevar$ is chosen appropriately. 
\end{remark}

Statistical guarantees for the \textsc{PureVar}-Algorithm can be obtained if the tuning parameter $\hyperpurevar$ is chosen appropriately, respecting the bias of $\hat{\mathcal X}$ and the size of certain entries in $\bar A$. The latter will be measured by the following signal strength parameter $\eta$:
\begin{align}
\label{eq:signal-strength-eta}
\eta := \eta(\bar A) := \min\big\{ \bar A_{ja} \wedge (1-\bar A_{ja}) \mid j \in [d], a \in [K] \text{ s.t. } \bar A_{ja} \notin\{ 0,1\}\big\} \in (0,1/2].
\end{align}
Note that $\bar{A}_{ja}\in \{0\} \cup [\eta, 1-\eta] \cup \{1\}$ for all $j \in [d]$ and $a \in [K]$. 
For controlling the bias of the empirical tail correlation, let
$\chi_{n,\knon}(j,\ell)$ with $j,\ell \in [d]$ and $j\ne \ell$ denote the pre-asymptotic extremal coefficient between $X_j$ and $X_\ell$, i.e.,
\begin{equation*}
    \chi_{n, \knon}(j,\ell) = \frac{n}{\knon} \mathbb{P}\left\{ F_j(X_j) > 1-\knon /n , F_\ell(X_\ell) > 1- \knon/n \right\},
\end{equation*}
and define
\begin{equation} \label{eq:dkn}
   D(\knon/n) := \underset{j, \ell \in [d], j \ne \ell}{\sup} \, \big|\chi_{n, \knon}(j,\ell) - \chi(j,\ell) \big|
\end{equation}
as the maximal bias. Combined with the concentration inequality from Proposition~\ref{prop:concentration-tail-correlation}, we find that, for each fixed small probability level $\delta\in(0,1)$, we have $\sup_{j, \ell \in [d], j \ne \ell} \big|\hat \chi_{n, \knon}(j,\ell) - \chi(j,\ell) \big| \le \kappa_0 = \kappa_0(\delta, n, \knon, d)$ with probability at least $1-\delta$, where
\begin{equation} \label{eq:kappa0}
    \kappa_0 = D(\knon/n) + \frac{\sqrt{2}+ \sqrt{16 \ln(d/\delta)} + \sqrt{16 \ln(4d/\delta)}}{\sqrt{\knon}} + \frac{6+4\ln(d/\delta) + 8\ln(4d/\delta)}{3\knon};
\end{equation}
see the beginning of the proof of the next result for details.
As a consequence, we can expect the PureVar-Algorithm to work well if the hyperparameter $\hyperpurevar$ is at least as large $\kappa_0$; this is one of the conditions in the following result.

\begin{theorem}[Statistical guarantees for the \textsc{PureVar}-Algorithm]
\label{thm:consistency_K_pure2}
Suppose $\bm X_1, \dots, \bm X_n$ is an i.i.d.\ sample from a random vector $\bm X$ satisfying Condition~\ref{cond:stdf-max-linear}, with unknown parameter $(K, \bar A) \in \Theta_L$ with $\Theta_L$ from \eqref{eq:Theta-L} (i.e., $\bar A$ satisfies Condition~\ref{cond:purevar}). 
Fix a small probability level $\delta \in (0,1)$.
Then, with probability at least $1-\delta$, if $\hyperpurevar$ from Algorithm~\ref{alg:purevar} is chosen such that $\hyperpurevar \in [\kappa_0, \eta/2)$,
the following statements are true for any choice of a maximum clique that is to be made in Algorithm~\ref{alg:purevar}:
\begin{compactenum}[(i)]
\item $\hat K = K$ and $\hat I = I$. \label{item:thm_consistency_K_pure_i}
\item There exists a permutation $\pi$ on the set $[K]$ such that, for any $a \in [K]$, $\hat I_a=I_{\pi(a)}$. \label{item:thm_consistency_K_pure_ii}
\end{compactenum}
\end{theorem}

The result can be visualized as follows: with high probability, provided the accuracy of the empirical tail correlation matrix as measured by $\kappa_0$ is smaller than half the signal strength (which is `easier' to satisfy for large sample size or small dimensionality), we recover $K, I$ and $\mathcal I$ for any choice of the hyperparameter $\hyperpurevar$ that falls into the range $[\kappa_0, \eta/2)$. Typically, one would choose $\knon=\kn$ as an intermediate sequence (i.e., $\kn \to \infty$ with $\kn=o(n)$), which implies that $n \mapsto \kappa_0(\delta, n, \kn, d)$ is decreasing for any fixed $(\delta, d)$. Hence, the condition $\kappa_0 < \eta/2$ is satisfied for sufficiently large $n$. Apparently, we may even allow for $d =d_n \to \infty$; details are provided in Section~\ref{subsec:high-dimensions}.

\begin{remark}
    A close look at the proof of Theorem~\ref{thm:consistency_K_pure2} shows that the same assertion holds for a (potentially) larger signal strength parameter $\eta$, namely, for $\eta$ replaced by
    \[
    \tilde \eta = \min\big\{ 1- \bar A_{ja} \mid j \in J, a \in [K] \big\} 
    \wedge 
    \min\big\{ \bar A_{j, (K-1:K)} \mid j \in J \big\},
    \]
    where $\bar A_{j,(K-1:K)}$ denotes the second largest entry in the $j$th row of $\bar A$ (which is necessarily positive for $j \in J$).
By definition, for all non-pure-variables $j\in J$, we have  $\bar A_{ja} \le 1-\tilde \eta$ for all $a\in [K]$ and $\bar A_{ja} \ge \tilde \eta$ and $\bar{A}_{jb}\ge \tilde \eta$ for at least two distinct $a,b \in [K]$. As this improvement tends to be incremental, we will not go into more detail in the following sections of this paper.
\end{remark}

\subsection{The `Hard Thresholding and Simplex Projector'-Algorithm}

Using the estimators $\hat{\mathcal X}=(\hat{\chi}_{n, \knon}(j,\ell))_{j,\ell \in [d]}$ from \eqref{eq:empirical_correlation} and $\hat{K}$, $\hat{I}$ and $\hat{\mathcal{I}}$ obtained from Algorithm~\ref{alg:purevar}, we next construct an estimator $\triangleA \in [0,\infty)^{d \times \hat K}$ for the matrix $\bar A$, where we aim to ensure that the estimator is row-sparse. The latter means that any observed variable should only be connected to a small number of latent variables. Further, in view of the fact that $\bar A$ has row sums 1, we want to impose the same constraint on the estimator.

The estimator $\triangleA$ will be constructed in three steps, where, in step 1, we construct an initial estimator $\triangleAa$ which is neither row-sparse nor has row sums 1. 
For that purpose, we proceed rowwise, and we start with $j \in \hat I$, i.e., with an index corresponding to an estimated pure variable. For such $j$, there exists a unique $a \in [\hat K]$ such that $j \in \hat I_a$, and we define
\begin{equation*}
    \triangleAa_{ja} := 1, \quad \triangleAa_{jb} := 0, \quad \text{for } b \neq a.
\end{equation*}
Next, consider $j \in \hat J = [d] \setminus \hat I$. By definition of  $\chi(j,\ell)$ and of $L=L_{K, \bar A}$, we have $\chi(j,\ell) = \bar{A}_{j a}$ for all $\ell \in I_a$ and $j\in J$; see also \eqref{eq:chiA}. As a consequence, averaging over $\ell \in I_a$, we obtain that
$
\bar{A}_{j a} = |I_a|^{-1} \sum_{\ell \in I_a} \chi(j,\ell) ,
$
which suggests the estimator
\begin{align}
\label{eq:Ahat1}
\triangleAa_{j a} 
:= 
\bar{\hat \chi}_{n,k}(j,a) 
:= 
\frac{1}{|\hat{I}_a|} \sum_{\ell \in \hat{I}_a} \hat{\chi}_{n,k}(j,\ell), \qquad j \in \hat J, a \in [\hat K].
\end{align}
Since $a\in [\hat K]$ was arbitrary, we have defined the $j$th row of $\triangleAa_{j\cdot}$ and hence the full matrix $\triangleAa$. 
Note that the row sums of $\triangleAa$ are non-zero: this is obvious for $j \in \hat I$, and for $j\in \hat J$, we claim that there exists $\ell \in \hat I$ such that $\hat{\chi}_{n, \knon}(j,\ell)\ge \hyperpurevar>0$ with $\hyperpurevar$ from Algorithm~\ref{alg:purevar}. Indeed, otherwise, if we had $\hat{\chi}_{n, \knon}(j,\ell) < \hyperpurevar$ for all $\ell \in \hat I$, then the clique $\hat C$ constructed in Algorithm~\ref{alg:purevar} cannot be maximal, since $\hat C \cup \{j\}$ is a clique of larger cardinality.

In the second step, we transform the initial estimator $\triangleAa$ into an estimator that is row-sparse. For that purpose, we employ hard thresholding: let $\hypersparsity \in [0,1]$ denote a threshold parameter controlling the level of sparsity; often, we will choose $\hypersparsity=\hyperpurevar$, but more general choices are possible as illustrated in Section~\ref{sec:case_studies}. We then define 
\begin{equation}
\label{eq:Ahat2}
\triangleAb_{j a }  = \triangleAa_{ja} \bm 1( \triangleAa_{ja}  > \hypersparsity), \qquad j \in [d], a \in [\hat K].
\end{equation}
Here, we implicitly assume that the threshold parameter is chosen sufficiently small such that $\triangleAb$ still has non-zero row sums; by construction, this holds for $\hypersparsity \in [0, \min_{j\in[d]} \max_{a \in [\hat K]} \triangleAa_{ja}]$.

In the third step, the row-sparse estimator $\triangleAb$ will be transformed into an estimator which has row sums 1, i.e., whose row vectors belong the unit simplex $\Delta_{\hat{K}-1} = \{ \bm{v} \in [0,1]^{\hat{K}} :  \sum_{a=1}^{\hat{K}} v_a = 1 \}$. For that purpose, consider the projection operator $\mathcal P: [0,1]^p \to \Delta_{p-1}$, defined, for arbitrary $p\in\N$ and  $\bm w \in [0,1]^p$, by
\begin{align}
\label{eq:P-and-tau}
\mathcal P(\bm w) = \big(\max(w_a-\tau(\bm w), 0) \big)_{a=1}^p,
\qquad
\tau(\bm w) = \frac1\rho \Big( - 1 + \sum_{b=1}^\rho w_{(b)} \Big)
\end{align}
where $w_{(1)} \ge  \dots  \ge  w_{(p)}$ are the sorted components of $\bm w$ and where $\rho = \rho(\bm w) := \max \{ b \in [p] : w_{(b)} > \frac{1}{b}(\sum_{a=1}^b w_{(a)}-1) \}$. Note that $\rho>0$ is the number of non-zero components of $\mathcal P(\bm w)$ and that
\begin{align*}
\mathcal P(\bm w) \in \argmin_{\bm x \in \Delta_{p-1}} \| \bm x - \bm w\|_2,
\end{align*}
see Section 3 in \cite{duchi2008efficient}.
Further note that $\mathcal P(\bm w)$ may have more or less zero entries than $\bm w$; for instance, we have $\mathcal P(.2, .2, .2, 0) = (.3, .3, .3, .1)$ and $\mathcal P(1, .4, .1, 0) = (.8, .2,0,0)$. The fact that the number of zeros of $\mathcal P(\bm w)$ may be larger than the number of zeros of $\bm w$ is a nuisance that can be avoided by projecting the non-zero entries only. Formally, for $\bm w \in [0,1]^p \setminus \{\bm 0\}$, let $S = \supp(\bm w)= \{ a \in [p]: w_a >0 \} \ne \emptyset$. We then define $\tilde{\mathcal P}(\bm w)$ by
\begin{align*}
\tilde{\mathcal P}(\bm w) |_S = \mathcal P(\bm w |_S), 
\qquad
\tilde{\mathcal P}(\bm w) |_{[p] \setminus S} = \bm 0.
\end{align*}
Formally, this defines another function $\tilde {\mathcal P}: [0,1]^p \setminus \{ \bm 0\} \to \Delta_{p-1}$, which has the property that $\supp(\tilde {\mathcal P}(\bm w)) \subset \supp(\bm w)$. This function will now be used to define the final estimator $\triangleA$: we set
\begin{equation}
    \label{eq:triangleA}
    \triangleA_{j \cdot} = \tilde{\mathcal P}(\triangleAb_{j\cdot}), \qquad j \in  [d].    
\end{equation}
Note that $\triangleA_{j\cdot}=\triangleAa_{j\cdot}$ for all $j\in \hat I$.
The sparsity level of $\triangleA$ is denoted by $\hat s = \hat s(\hypersparsity) = \max_{j\in[d]} |\supp(\triangleA_{j\cdot})|$, which is a decreasing function of $\hypersparsity$. 

The estimator $\triangleA$ may be called a ``Hard Thresholding and Simplex Projector'' (HTSP); the underlying algorithm is inspired by  \cite{kyrillidis2013sparse} and is summarized in Algorithm~\ref{alg:htsp}. It is worthwhile to  mention that
\begin{equation*}
    \triangleA_{j\cdot} \in \underset{\bm v \in \Delta_{\hat{K}-1}(\hat s)}{\textrm{argmin}} \, \| \bm{v} - \triangleAa_{j\cdot}\|_2, \quad \forall j \in \hat{J},
\end{equation*}
where $\Delta_{\hat{K}-1}(\hat s) = \{ \bm{v} \in \Delta_{\hat{K}-1} : |\supp(\bm v)| \leq \hat s \}$; see \cite[Theorem 1]{kyrillidis2013sparse}.

\begin{algorithm}
\caption{(Estimating $\bar{A}$: The Hard Thresholding and Simplex Projector)} \label{alg:htsp}
\begin{algorithmic}[1]
\Require{$\hat {\mathcal X} \in[0,1]^{d \times d}, \hyperpurevar, \hypersparsity  \in (0,1/2)$}
\Function{HTSP}{$\hat{\mathcal X}, \hyperpurevar$}
    \State $ (\hat{K}, \hat{I}, \hat{\mathcal I}) \gets \textsc{PureVar}(\mathcal X, \hyperpurevar)$ \Comment{See Algorithm~\ref{alg:purevar}}
    \For {$ j \in \hat{I}$}
    \For {$a \in [\hat K]$}
    \If{$j \in \hat I_a$}
		\State $\triangleA_{j\cdot} \gets e_a$ \Comment{$e_a$ the $a$th unit vector in $\R^{\hat K}$} 
    \EndIf
    \EndFor
    \EndFor
    \For {$ j \in \hat{J}=[d] \setminus \hat I$}
        \For {$a \in [\hat K]$}
		  \State $\triangleAa_{ja} \gets |\hat{I}_a|^{-1} \sum_{\ell \in \hat{I}_a} \hat{\chi}_{n,k}(j,\ell)$ 
            \State $\triangleAb_{ja} \gets \triangleAa_{ja} \bm 1(\triangleAa_{ja}>\hypersparsity)$ \Comment{Hard thresholding}
        \EndFor
        \State $\triangleA_{j} \gets \tilde{\mathcal P}(\triangleAb_{j\cdot})$ \Comment{Projection on the simplex}
    \EndFor
    \State{\Return $\triangleA$}
    \EndFunction
\end{algorithmic}
\end{algorithm}

We next derive statistical guarantees for the \textsc{HTSP}-algorithm. For that purpose, define the loss function
\begin{equation*}
        L_{\infty, 2}(A, A') = \begin{cases}
        \min_{P \in S_{K} } \| A- A'P \|_{\infty,2} & K= K', \\
        + \infty & K \ne K',
        \end{cases}
\end{equation*}
where $A \in \R^{d \times K}, A' \in \R^{d \times K'}$ are two matrices of possibly different column dimensions,  where
$S_K$ is the group of all $K \times K$ permutation matrices and where
\begin{equation*}
    \|A\|_{\infty,2} := \max_{j \in [d]} \|A_{j\cdot}\|_{2} 
    = 
    \max_{j \in [d]} \left( \sum_{a \in [K]} |A_{ja}|^2 \right)^{1/2}.
\end{equation*}
Note that the choice of the $\|\,\cdot\,\|_{\infty,2}$-norm is mostly made for mathematical convenience; by equivalence of norms, the statements to follow can be transferred to any other norm of interest at the cost of additional finite factors depending on $d$ and $K$.
Further, let
\begin{align}
\label{eq:s-sparsity}
s:=\max_{j \in [d]} |\{a \in [K]: \bar A_{ja}>0\}| \in [K]
\end{align}
denote the row-sparsity index of the population matrix $\bar A$, and, recalling Lemma~\ref{lem:tail-probabilities-maxlinear-stdf}, let $D_1 = \supp(\bar A_{\cdot 1}), \dots, D_K = \supp(\bar A_{\cdot K})$ denote the extremal directions of $\bm X$; note that these are pairwise different in view of the pure variable assumption.

\begin{theorem}[Statistical guarantees for the \textsc{HTSP}-Algorithm]
    \label{thm:stat_guarantees_Abar} 
    Suppose $\bm X_1, \dots, \bm X_n$ is an i.i.d.\ sample from a random vector $\bm X$ satisfying Condition~\ref{cond:stdf-max-linear}, with unknown parameter $(K, \bar A) \in \Theta_L$ with $\Theta_L$ from \eqref{eq:Theta-L} (i.e., $\bar A$ satisfies Condition~\ref{cond:purevar}). 
    Then, for any $\delta\in(0,1)$, if we choose $\hyperpurevar, \hypersparsity \in [\kappa_0, \eta/2)$ with $\kappa_0=\kappa_0(\delta, n, \knon, d)$ from \eqref{eq:kappa0} and $\eta$ from \eqref{eq:signal-strength-eta}, we have that the following statements are true with probability at least $1-\delta$:
    \begin{compactenum}[(i)]
        \item \label{item:thm_stat_guarantees_Abar_1}
        $L_{\infty, 2}(\triangleA, \bar{A}) \le 2\sqrt{s} \hypersparsity$; 
        \item \label{item:thm_stat_guarantees_Abar_3}
        There exists a permutation matrix $P \in \mathbb{R}^{K \times K}$ such that $\supp( (\bar{A}P)_{j \cdot}) = \supp(\triangleA_{j \cdot})$ for any $j \in J$, and hence $\hat{s} = s$. 
        \item \label{item:thm_stat_guarantees_Abar_5}
        There exists a permutation $\pi$ on the set [K] such that  $\{\hat  D_1, \dots, \hat D_{\hat K}\} = \{D_{\pi(1)}, \dots, D_{\pi(K)} \}$.
    \end{compactenum}
\end{theorem}

Recalling that $\kappa_0$ from \eqref{eq:kappa0} provides a uniform non-stochastic upper bound on the absolute estimation error of the pairwise tail correlations, Part (i) of Theorem~\ref{thm:stat_guarantees_Abar} essentially means that  $\bar A$ can estimated as accurately as the pairwise tail correlations (up to a linear dependence on the sparsity index $s$) if $\hypersparsity$ is chosen of the same order as $\kappa_0$. Practical advice on how to select $\hypersparsity$ (and $\hyperpurevar$ and $k$) is given in Section~\ref{sec:simulations} and \ref{sec:case_studies}. Furthermore, Part (iii) explicitly states that extremal directions are accurately estimated.

\subsection{Asymptotic statistical guarantees in high dimensions}
\label{subsec:high-dimensions}

The results in the previous two theorems were derived for data of fixed dimension and sample size. In this section, we provide consistency results for the case where the dimension $d= d_n$ may depend on $n$ and actually be larger than $n$. More precisely, we assume that, for each $n \in \N$, we observe an i.i.d.\ sample $(\bm{X}_1,\dots,\bm{X}_n) = (\bm{X}_1^{(n)},\dots,\bm{X}_n^{(n)})$ from a $d_n$-variate random vector $\bm X^{(n)}$ satisfying Condition~\ref{cond:stdf-max-linear}, that is, $\bm X^{(n)}$ has continuous marginal cdfs and its stable tail dependence $L^{(n)}$ is of the form given in \eqref{eq:stable-tail-dp-A} for some unknown parameter $\theta^{(n)} = (K^{(n)}, \bar{A}^{(n)}) \subset \Theta_L=\Theta_L(d_n)$, where $\Theta_L$ is as in \eqref{eq:Theta-L}. 
Since we are aiming for almost sure convergence statements, it is assumed that all random variables are defined on the same probability space.
For notational convenience, we occasionally suppress the upper index $n$.

The algorithms from the previous two sections depend on three hyperparameters, $k, \hyperpurevar$ and $\hypersparsity$. The consistency statements below are based on explicit choices that depend on suitable second order conditions.
To formally introduce the latter, let $\mathcal{C}_2$ denote the set of bivariate copulas for which the tail correlation coefficient $\chi$  exists. More precisely, if $\bar C(u,v) = u+v-1+C(1-u,1-v) = \Prob_{(U,V) \sim C}(1-U \le u, 1-V\le v)$ denotes the survival copula associated with $C$, it is assumed that the limit $\chi := \lim_{t \rightarrow \infty} \chi_t \in[0,1]$ exists, where $\chi_t = t \bar{C}(1/t,1/t)$ for $t >  1$. For parameters 
$\rho \in (0,\infty) $  and $M \in[0, \infty)$, let
\[
    \mathfrak{X}(\rho, M) := 
    \left\{ C \in \mathcal{C}_2 \mid \forall t > 1 : \left| \chi_t - \chi \right|  \le M t^{-\rho}\right\},
\]
Unless $M=0$ (in which case $\chi_t=\chi$ does not depend on $t$ and $\rho$ is uninformative), the parameter $\rho$ measures the speed of convergence of $\chi_t$ to $\chi$: the larger $\rho$, the faster the convergence. Note that $\mathfrak{X}(\rho, M)$ is increasing in $M$ and decreasing in $\rho$.

In a suitable subclass of the max-linear model from Condition~\ref{cond:max-factormodel} (without noise), all bivariate marginal copulas $C_{j,\ell}$ of $\bm X$ belong to $\mathfrak X(1, M_{j,\ell})$ for some $M_{j,\ell} \ge 0$.

\begin{proposition} \label{prop:max-linear-frechet-second-order}
For $d\in \N_{\ge 2}$ and $K\in \N$, consider the max-linear model defined by
\[
\bm{X} = A \times_{\max} \bm{Z}, \quad A \in \Acp(K),
\]
where \( \bm Z=(Z_1, \dots, Z_K)^\top \) has i.i.d. Fréchet random margins with tail index $\alpha > 0$. Then, for any pair of indices \( 1 \le j < \ell \le d \), the bivariate marginal copula \( C_{j,\ell} \) of $(X_j, X_\ell)$ belongs to the class 
$\mathfrak{X}(1, M_{j,\ell})$, where $M_{j,\ell} = 1 - \chi(j,\ell) = 1- \sum_{a \in [K]} \bar{A}_{ja} \wedge \bar{A}_{\ell a} \in [0,1]$, with $\bar A=f(A,\alpha)$ the matrix defined in \eqref{eq:barA}.
\end{proposition}

For the following result, let $C_{j,\ell}^{(n)}$ denote the bivariate copula of $(X^{(n)}_j, X_\ell^{(n)})$. Further, let
\begin{align*}
\eta = \eta_n = \eta(\bar A^{(n)}) 
&= 
\min\big\{ \bar A_{ja}^{(n)} \wedge (1-\bar A_{ja}^{(n)}) \mid j \in [d], a \in [K^{(n)}] \text{ s.t. } \bar A_{ja}^{(n)} \notin\{ 0,1\}\big\} \in (0,1/2], \\
s = s_n = s(\bar A^{(n)})
&=
\max_{j \in [d]} |\{a \in [K^{(n)}]: \bar A_{ja}^{(n)}>0\}| \in [K^{(n)}],
\end{align*}
denote the signal strength parameter and the the row-sparsity index of $\bar A^{(n)}$, respectively; see also \eqref{eq:signal-strength-eta} and \eqref{eq:s-sparsity}.

\begin{theorem}
    \label{thm:high_dim2}
    For each $n\in \N$, let $\bm X_1^{(n)}, \dots, \bm X_n^{(n)}$ be an i.i.d.\ sample from a $d_n$-variate random vector $\bm X^{(n)}$ satisfying Condition~\ref{cond:stdf-max-linear} with unknown parameter $\theta^{(n)}=(K^{(n)}, \bar A^{(n)}) \in \Theta_L(d_n)$, with $\Theta_L(d_n)$ from \eqref{eq:Theta-L}, i.e., $\bar A^{(n)}$ satisfies Condition~\ref{cond:purevar}. Here, all random variables are assumed to be defined on the same probability space. Moreover, assume that $d_n \ge 2$ satisfies $\ln(d_n)=o(n)$, that there exists $\rho_0>0$ and $M_0>0$ such that each bivariate copula $C_{j,\ell}^{(n)}$ is in $\mathfrak{X}(\rho_0,M_0) = \bigcup_{\rho \ge \rho_0}\bigcup_{M \in [0, M_0]} \mathfrak{X}(\rho,M)$ and  that $ \liminf_{n\to\infty} \eta_n>0$.
    Then, 
    for any $r=r_n \in [\rho_0, \infty)$ and $M=M_n\in (0, M_0]$ such that $C_{j,\ell}^{(n)} \in \mathfrak{X}(r_n,M_n)$ for all $1 \le j < \ell \le d_n$ (for instance, $r=\rho_0$ and $M=M_0$), choosing 
    \begin{align} \label{eq:def-k-asymptotic-2}
    k = \lfloor c_k \left( \ln(4dn^2) \right)^{\frac{1}{2r+1}} n^{\frac{2r}{2r+1}}\rfloor \qquad \text{ and }  \qquad \hyperpurevar = \hypersparsity = c_\hyperpurevar \left(\ln(4dn^2) / {n} \right)^{\frac{r}{2r+1}}
    \end{align}
    for some $c_k =c_k(n), c_\hyperpurevar=c_\hyperpurevar(n)>0$ that are bounded away from zero and infinity and satisfy $c_{\hyperpurevar} \ge M c_k^r + 9c_k^{-1/2} + 5 c_k^{-1}$ (for instance, $c_k=c_0>0$ fixed and $c_\hyperpurevar=M_0 c_0^{\rho_0}+9c_0^{-1/2}+5c_0^{-1}$), we have
    \[
    \Prob\big\{ \hat I_n = I_n, \hat K_n = K_n, \hat s_n=s_n \textnormal{ for all but finitely many $n$} \big\} = 1 
    \]
    and
\begin{align} \label{eq:rate-Ahat}
    L_{\infty,2}(\triangleA, \bar{A}) = O_{a.s.}\Big(\sqrt{s} \Big( \frac{\ln(nd)}{n} \Big)^{\frac{r}{2r+1}} \Big). 
\end{align}
\end{theorem}

Note the resemblance of the rate in \eqref{eq:rate-Ahat} with minimax-optimal rates in nonparameteric regression, and that rates of the form $n^{-r/(2r+1)}$ also appear as `optimal' estimation rates in bivariate extremes under second order conditions; see \cite{DreesHuang1998} or Section 4 in \cite{BucherZou2019}. Our procedure hence recovers such optimal rates  up to logarithmic terms and the sparsity index $\sqrt s$.
Finally, note that, in the situation from Proposition~\ref{prop:max-linear-frechet-second-order}, the condition ``$C_{j,\ell}^{(n)} \in \mathcal X(r_n,M_n)$ for all $1 \le j < \ell \le d_n$'' is for instance satisfied for the choice $r=M=1$, which is independent of $n$.

\section{Monte Carlo Simulations}
\label{sec:simulations}

A large scale Monte Carlo simulation study was conducted to evaluate the finite-sample performance of the proposed estimators, with particular attention paid to the sensitivity of the methods depending on the hyperparameters.  The simulation is based on $N_{\text{sim}} = 100$ independent replications for each parameter configuration described below.

\paragraph{Simulation design.}
The experimental framework is based on the following choice of key model parameters:
\begin{compactenum}
    \item Sample size: $n \in \{ 1000,2000,\dots, 10\,000\}$;
    \item Dimension: $d\in \{100, 1000\}$;
    \item Number of factors: $K \in \{5, 20\}$;
    \item Signal strength: $\eta = 0.2$;
    \item Sparsity index: $s=4$;
    \item Model: Linear-NoNoise, Linear-Noise, Max-Linear-NoNoise, Max-Linear-Noise.
\end{compactenum}
More precisely, regarding the last point,  we consider the linear and max-linear model with $K$ i.i.d.\ Pareto(1) factors from Model~\ref{cond:factormodel} and \ref{cond:max-factormodel}, respectively, each of the two either without noise, or with i.i.d.\ Pareto(2) noise variables. 
The true matrix $\bar{A}$ for each of the configurations is chosen randomly as follows:
\begin{compactenum}
    \item Without loss of generality, we structure $\bar{A}$ such that its first $K \times K$ block forms the identity matrix. This ensures the existence of at least $K$ pure variables, as required in Condition~\ref{cond:purevar}.
    \item The $(K+1)$-th row contains exactly $s$ non-zero entries, with their positions drawn uniformly at random from $\{1,\ldots,K\}$.
    \item For rows $i \in\{ K+2,\ldots,d\}$, the number of non-zero entries $s_i$ is drawn uniformly from $\{1,\ldots,s\}$, with their positions sampled randomly from $\{1,\ldots,K\}$. This construction may introduce additional pure variables beyond the first $K$.    
    \item For each non-pure row, all non-zero entries are sampled uniformly from the $(s_i-1)$-dimensional unit simplex. We enforce the minimum/maximum loading constraints $\eta \leq \bar{A}_{ja} \leq 1-\eta$ through rejection sampling, repeating the draws until all entries satisfy the condition.
\end{compactenum}
By construction, Condition \ref{cond:purevar} is met, the signal strength of $\bar A$ is at least $\eta$ and the sparsity index of $\bar A$ is $s$. This yields a full factorial design with
\[
|\text{Total scenarios}| = N_n \times N_d \times N_K \times N_\eta \times N_s \times N_{\text{model}} = 160,
\]
where for each scenario we generate $N_{\text{sim}} = 100$ independent replicates of sample size $n$.
\paragraph{Hyperparameter design.} For each sample, we apply our algorithms with the following hyperparameter configurations:
\begin{compactenum}
    \item \textbf{Effective number of large order statistics $k$}:
    \begin{compactitem}
        \item Fixed: $k \in \{0.01n, 0.05n\}$;
        \item Data-adaptive: $k = \lfloor 0.25 (\ln(4dn^2))^{1/3} n^{2/3} \rfloor$ (motivated by Theorem~\ref{thm:high_dim2}).
    \end{compactitem}
    
    \item \textbf{Threshold parameters $\hyperpurevar$ and $\hypersparsity$}: 
    \begin{compactitem}
        \item Fixed: $\hyperpurevar = \hypersparsity = 0.1$;
        \item Data-adaptive: $\hyperpurevar = \hypersparsity = 0.75 (\ln(4dn^2)/n)^{1/3}$ (per Theorem~\ref{thm:high_dim2}).
    \end{compactitem}
\end{compactenum}
We evaluate two distinct strategies: the fixed approach where $(k, \hyperpurevar) \in \{0.01n, 0.05n\} \times \{0.1\}$, and the data-adaptive approach where both $k$ and $\hyperpurevar$ are set via their adaptive formulas. 

\paragraph{Performance Metrics.}
Recall that the loading matrix $\bar{A}$ and our estimator $\triangleA$ are not directly comparable, since they may differ by a permutation matrix. To evaluate the performance of our method when $\hat{K} = K$, we align the estimated components $\triangleA$ with the ground truth $\bar{A}$ by solving a matching problem. Specifically, we seek the permutation matrix $P \in S_K$ that minimizes the Frobenius norm $\| \triangleA - \bar{A}P \|_F$. This corresponds to finding the optimal assignment between columns of $\triangleA$ and $\bar{A}$ that best aligns the two matrices. To solve this assignment problem efficiently, we employ the Hungarian algorithm \citep{kuhn1955hungarian}. The algorithm takes as input a cost matrix $C_H \in \mathbb{R}^{K \times K}$, where each entry is defined as
\[
    C_H(a, b) = \| \triangleA_{\cdot a} - \bar{A}_{\cdot b} \|_2^2, \quad \text{for } a, b \in [K].
\]
In this matrix, each element $C_H(a, b)$ quantifies the squared Euclidean-distance between the $a$-th column of the ground truth and the $b$-th column of the estimate. The Hungarian algorithm then identifies the column-wise permutation that minimizes the total cost, yielding the optimal mapping $P$ for evaluation. Thus, we can compare the performance of the estimator with the permuted ground truth $\tilde{A} = \bar{A} P$. Recall the notation $D_a$ for an extremal direction from \eqref{eq:softcluster}, and let $\tilde{D}_a = \{j \in [d] : \tilde{A}_{ja} > 0\}$ and $\hat{D}_a = \{j \in [d] : \triangleA_{ja} > 0\}$. The performance of the algorithms may then be assessed in terms of the following  performances metrics:
\begin{compactenum}
    \item \textbf{Exact recovery rates}:
    \begin{compactitem}
        \item Factor dimension: proportion of cases for which $\hat K = K$;
        \item Sparsity index: proportion of cases for which $\hat{s} = s$;
        \item Set of pure variables: proportion of cases for which $\hat{I} = I$.
    \end{compactitem}
    \item \textbf{Classification errors}:
    \begin{compactitem}
        \item Total False Negative Proportion (TFNP):
        \[ \text{TFNP} =\frac{\sum_{a \in [K]} |\tilde{D}_{a} \cap \hat{D}_a^c|}{\sum_{a \in [K]} |\tilde{D}_a|}, \]
        \item Total False Positive Proportion (TFPP):
        \[ \text{TFPP} = \frac{\sum_{a \in [K]} |\tilde{D}_{a}^c \cap \hat{D}_a|}{\sum_{a \in [K]} |\tilde{D}_a^c|}; \]
    \end{compactitem}
    \item \textbf{Matrix estimation error}:
    \[ \|\triangleA - \tilde{A}\|_{\infty,2}. \]
\end{compactenum}
Note that all metrics except the first two are only well-defined if $\hat{K} = K$.

\begin{figure}[thp!]
    \centering
    \includegraphics[scale=0.37]{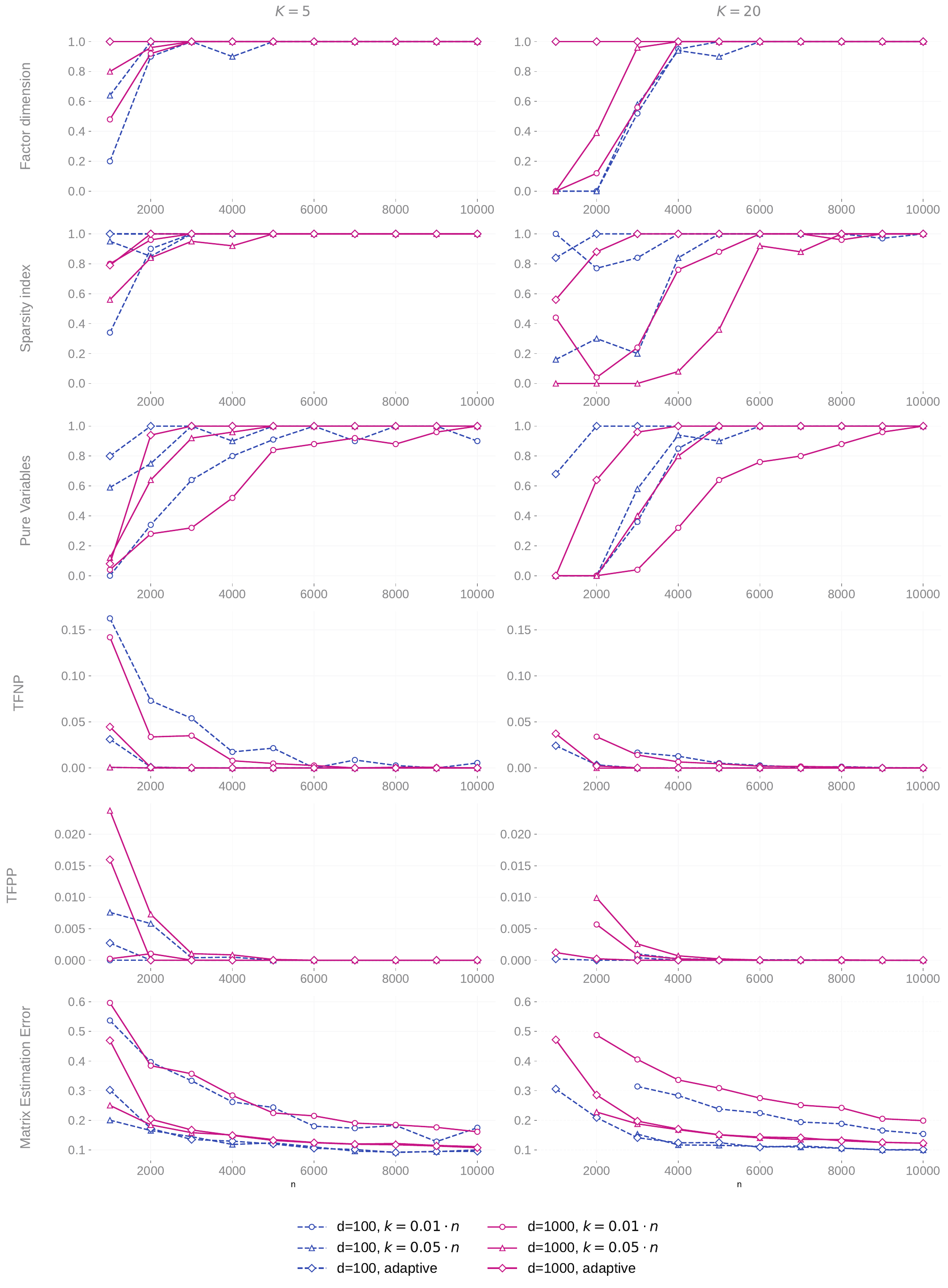}
    \caption{Performance metrics for the linear model with noise across different parameter combinations. Each rows depicts in order: (1) recovery rate of latent factors, 
        (2) recovery rate of sparsity, (3) recovery rate of pure variables, 
        (4) TFNP, (5) TFPP, and (6) matrix estimation Error. Each metric is plotted as a function of sample size $n$, 
        comparing dimensions $d \in \{100, 1000\}$, and the three choices of $(k,\hyperpurevar, \hypersparsity)$ described in the main text. Results are stratified by $K=5$ (left column) and $K=20$ (right column).}
    \label{fig:result_sum_noise}
\end{figure}

For the sake of brevity, the results are only partially reported; more specifically, we focus on the noisy linear factor model, as the other models provide qualitatively similar results (see Appendix~\ref{sec:additional-simulation-results} for details).
The results are summarized in Figure~\ref{fig:result_sum_noise}, which is a multi-panel plot where rows represent performance metrics and columns contrast number of latent factors ($K = 5$ vs.\ $K = 20$). It can be observed that the performance improves with sample size~$n$ for both $d = 100$ and $d = 1000$, though convergence is slower for $d = 1000$, reflecting higher sample complexity in high dimensions. The data-adaptive hyperparameter selection for $(k, \hyperpurevar, \hypersparsity)$ consistently matches or outperforms the two deterministic choices ($k \in \{0.01, 0.05\} \cdot n$, $\hyperpurevar=\hypersparsity = 0.1$), particularly for $d = 1000$, underscoring its advantage for scalability. These empirical trends align with Theorem~\ref{thm:high_dim2}, which guarantees dimension-free error bounds under adaptive tuning. 
The recovery task proves more challenging for $K = 20$ (right column) when using fixed tuning parameters. However, adaptive methods significantly narrow this performance gap. 
Remarkably, TFNP (false negative rate) exceeds TFPP (false positive rate) across all configurations, suggesting that the method yields smaller but more reliable extremal directions.
Under certain conditions, specifically for the sparsity recovery rate with $K=20$ and non-adaptive hyperparameters, the evaluation metrics can show non-monotonic behavior for small sample sizes. This is likely a consequence of the chosen parameters $k$ and $\kappa$ inadequately balancing the bias-variance tradeoff, thereby degrading performance. Consistent with this explanation, the adaptive estimator, which is optimized for this tradeoff, displays monotonic behavior.

\section{Case Studies}
\label{sec:case_studies}

We demonstrate our methods through two case studies. The first involves dietary data, where our model fits well and yields results, including estimated extremal directions, that align with those from other methods. The second case study focuses on wind speed gusts, revealing model limitations. We discuss these shortcomings and propose a potential model extension in the concluding Section~\ref{sec:conclusion}.

\subsection{Dietary intakes data}
\label{sec:dietary}

We evaluate our method on the NHANES 2015–2016 Day 1 total nutrient intake dataset (DR1TOT\_I; $N = 9544$), publicly available at \url{https://wwwn.cdc.gov/Nchs/Data/Nhanes/Public/2015/DataFiles/DR1TOT_I.xpt}. Prior work has examined dependence among extreme nutrient intakes due to potential links to adverse health effects \citep{janssen2020k,krali2025heavy}. Although the dataset reports intakes for 39 nutrients, additional diagnostics indicate that our model does not provide an adequate fit in the full 39-dimensional setting (see Section~\ref{sec:supplement_case_studies} for details), motivating a focused analysis on a lower-dimensional subset. Following \cite{krali2025heavy}, we focus on high intakes of the following $d=6$ specific nutrients: $\alpha$-carotene (AC), $\beta$-carotene (BC), lutein + zeaxanthin (LZ), retinol (RET), vitamin A (VA), and vitamin K (VK). This set of nutrients was identified by \cite{krali2025heavy} as the largest connected component of a directed acyclic graph generated by a Recursive Max-Linear Model under which hidden confounders, such as unobserved genetic factors jointly influencing extreme nutrient intakes, can be neglected. 

After keeping only complete observations, the sample size is reduced to $n = 8327$. The first step of the algorithm consists of selecting the threshold parameter $k$ to compute the empirical tail correlation matrix $\hat{\mathcal{X}}$. Similar as in Section~\ref{sec:simulations}, we set $k = \lfloor 0.25 (\ln(4dn^2))^{1/3} n^{2/3} \rfloor = 284$, which roughly corresponds to the largest $3\%$ of the observations. The second step is to determine suitable values for $\hyperpurevar$ and $\hypersparsity$, for which we proceed as follows. Let $\triangleA{}^{\hyperpurevar,\hypersparsity}$ and $\hat{K}{}^{\hyperpurevar,\hypersparsity}$ denote, respectively, the estimated matrix and the estimated number of factors obtained from Algorithm~\ref{alg:htsp} when using $\hyperpurevar,\hypersparsity \in (0,1/2)$ as tuning parameters. In view of \eqref{eq:chiA}, the corresponding fitted extremal correlation matrix is given by $\tilde{\mathcal{X}}{}^{\hyperpurevar,\hypersparsity}
= ( \sum_{a \in [\hat K^{\hyperpurevar,\hypersparsity}]} \triangleA{}^{\hyperpurevar,\hypersparsity}_{j a} \wedge \triangleA{}^{\hyperpurevar,\hypersparsity}_{\ell a} )_{j, \ell \in [d]}$. To select the tuning parameters, we consider a grid of candidate values for $\hyperpurevar$ and $\hypersparsity$. Specifically, we choose $100$ values for each parameter by scaling the theoretical rate $(\log(4dn^2)/n)^{1/3}$, as introduced in Section~\ref{sec:simulations}, with constants $c$ evenly spaced between $0.25$ and $1.5$. This yields candidate values for $\hyperpurevar$ and $\hypersparsity$ ranging approximately from $0.03$ to $0.2$. The optimal tuning parameters, denoted $\hyperpurevar^*$ and $\hypersparsity^*$, are then chosen as those for which the coefficient of determination $R^2$ between the empirical extremal correlation matrix $\hat{\mathcal{X}}$ and the fitted extremal correlation matrix $\tilde{\mathcal{X}}^{\hyperpurevar,\hypersparsity}$ is maximized. This procedure yields $\hyperpurevar^* = 0.092$ and $\hypersparsity^* = 0.12$; see the left panel of Figure \ref{fig:complete_analysis_dietary}.

As illustrated in the right panel of Figure \ref{fig:complete_analysis_dietary}, the adaptive choice of tuning parameters $\hyperpurevar^*$ and $\hypersparsity^*$ yields $\hat K =3$ factors, with associated pure variables $\hat I_1 = \{\textrm{RET}\}$,  $\hat I_2 = \{\textrm{AC}\}$ and $\hat I_3 = \{\textrm{LZ, VK}\}$. The extremal directions given by the algorithm are $\hat D_1 = \{\textrm{RET, VA}\}$, $\hat D_2 = \{\textrm{AC, BC, VA}\}$ and $\hat D_3 = \{\textrm{LZ, VK, BC, VA}\}$. This factor structure mirrors the source-node configuration in \cite[Figure 6]{krali2025heavy}, where RET and AC emerge as separate source nodes, while LZ and VK constitute an indistinguishable pair. Additionally, the directional pathway shown in their Figure 6 corresponds directly to our estimated extremal directions. To confirm such findings, we compare our extremal directions with two algorithm from the literature, namely CLEF \citep{chiapino2019identifying} and DAMEX \citep{goix2017sparse}. Following the same criterion for selecting extremes, we empirically transform the margins to standard Pareto and retain the $k = 284$ observations with the largest Euclidean norms. The extremal directions identified by the CLEF with tuning parameter of $0.1$ (the default choice in the implementations provided in \citealp{MLExtreme})
are 
\[
\{ \text{RET}, \text{VA} \}, \quad \{ \text{AC}, \text{VA}, \text{BC} \}, \quad \text{and} \quad \{ \text{LZ}, \text{BC}, \text{VK} \}.
\]
We note that the results are robust to the choice of the tuning parameter: any positive value smaller than 0.11 yields the same extremal directions. For the DAMEX algorithm with a tuning parameter of $0.3$, the detected extremal directions are
\[
\{ \text{RET}, \text{VA} \}, \quad \{ \text{AC}, \text{VA}, \text{BC} \}, \quad \text{and} \quad \{ \text{LZ}, \text{VA}, \text{BC}, \text{VK} \}.
\]
Again, the results are robust within a range: any tuning parameter between 0.24 and 0.41 produces the same set of extremal directions, while only two extremal directions \{RET, VA\} and \{AC,VA,BC,LZ,VK\} are found for parameters smaller than 0.24.
Overall, both alternative algorithms yield results consistent with our method: DAMEX produces identical extremal directions, and CLEF returns two matching clusters. The only divergence occurs in CLEF's third cluster, which excludes VA, precisely the feature with the weakest loading on our third factor.

\begin{figure}[htp!]
    \centering
    \begin{minipage}[t]{0.45\textwidth}
        \centering
        \includegraphics[width=\linewidth, height=0.35\textheight, keepaspectratio]{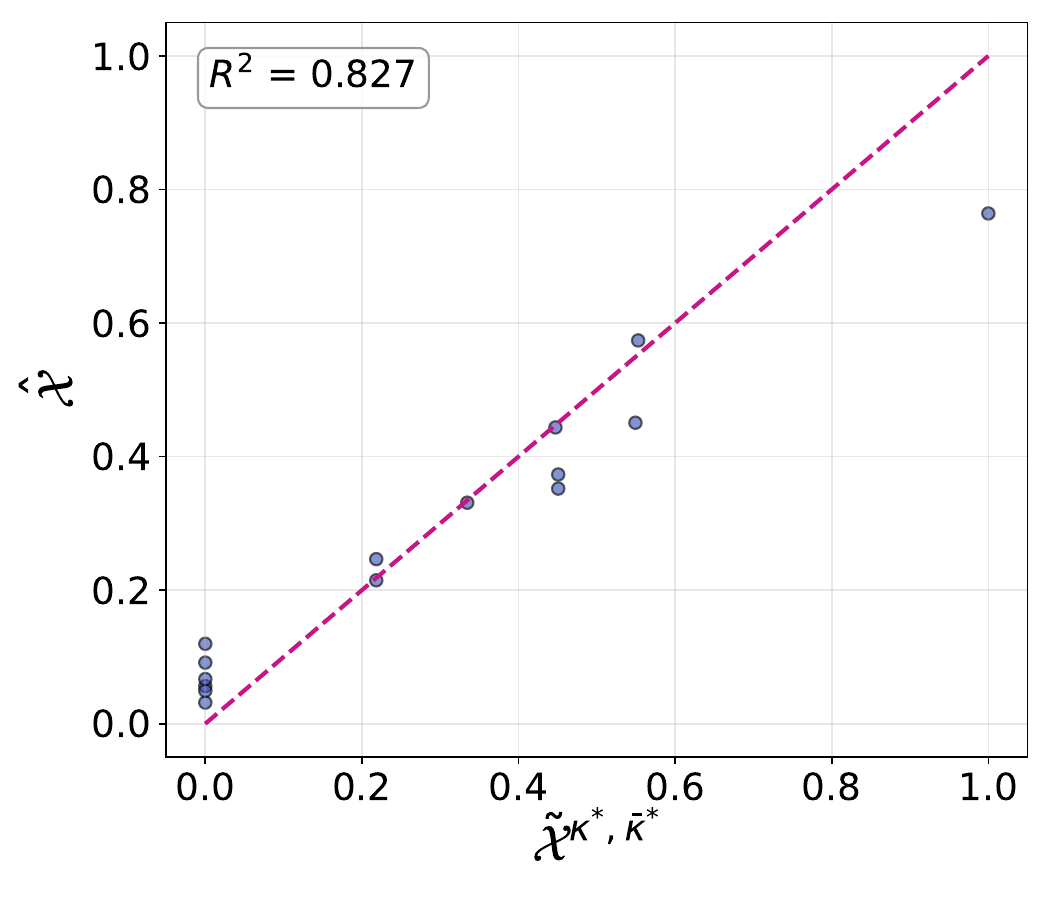}
    \end{minipage}
    \hfill
    \begin{minipage}[t]{0.45\textwidth}
        \centering
        \includegraphics[width=\linewidth, height=0.35\textheight, keepaspectratio]{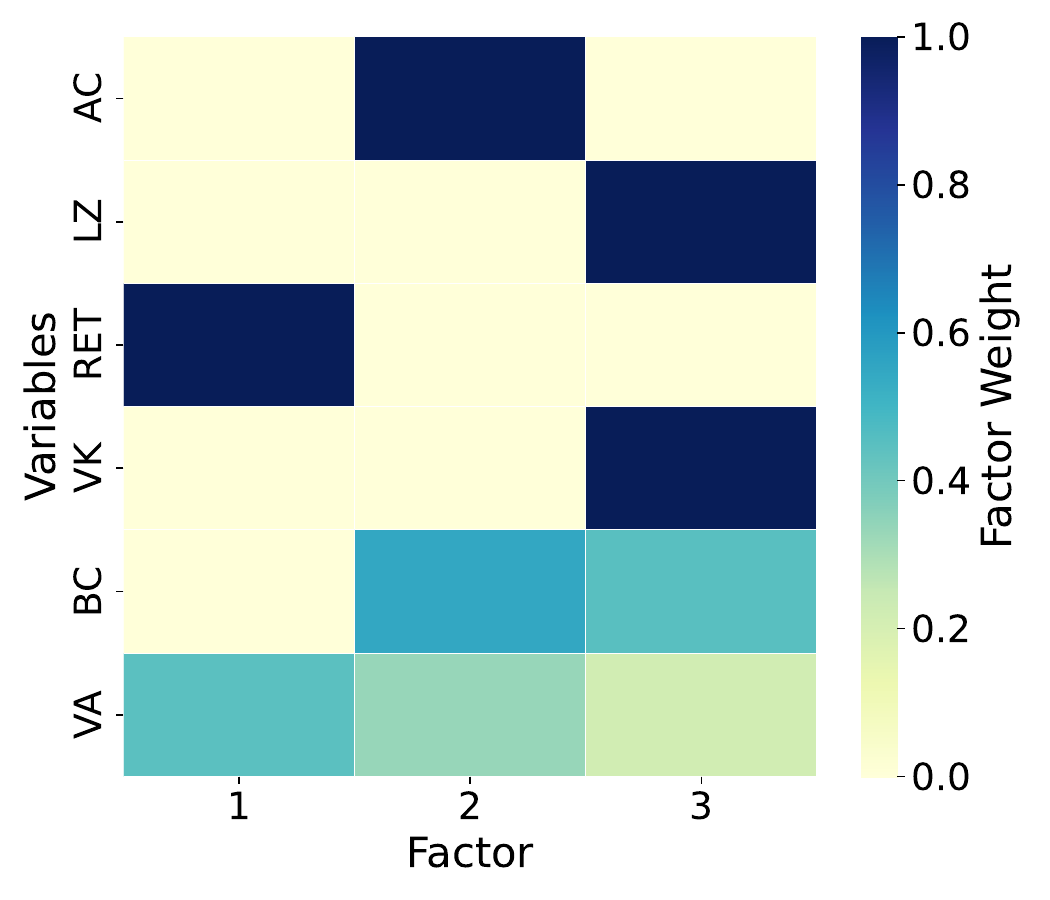}
    \end{minipage}    
    \caption{Results for the dietary data set with $d=6$. Left: empirical correlations from $\hat{\mathcal{X}}$ vs. fitted extremal correlations from $\tilde{\mathcal{X}}^{\hyperpurevar^*, \hypersparsity^*}$. Right: Estimated loading matrix $\triangleA{}^{\hyperpurevar^*, \hypersparsity^*}$.}
    \label{fig:complete_analysis_dietary}
\end{figure}

\subsection{Wind gusts in Schleswig-Holstein}
\label{sec:case-study-wind-speed}

Our second case study aims at analyzing tail dependence between hourly peak wind speeds (m/s) from $d = 22$ weather stations in Schleswig-Holstein (northern Germany) during winter seasons (December-January-February) from December 2013 through February 2018. Figure~\ref{fig:complete_analysis_spatial} (left) maps the station locations. The dataset is openly provided by the Deutscher Wetterdienst (DWD) at \url{https://opendata.dwd.de/climate_environment/CDC/observations_germany/climate/hourly/extreme_wind/historical/}. Restricting to complete records gives $n = 9447$ observations. As in the previous section, we set $k = \lfloor 0.25 (\ln(4dn^2))^{1/3} n^{2/3} \rfloor = 316$ for the number of threshold exceedances. Tuning parameters $\hyperpurevar$ and $\hypersparsity$ are chosen analogously to Section~\ref{sec:dietary}, resulting in $\hyperpurevar^* = \hypersparsity^* = 0.198$ and $\hat K=3$ estimated factors. Figure \ref{fig:complete_analysis_spatial} (right) compares empirical and model-implied extremal correlations. \emph{Pure–Pure} and \emph{Pure–Impure} pairs align well with the identity line, whereas \emph{Impure–Impure} pairs scatter below it, resulting in an overall poor fit.

The observed results can be explained as follows: the pure variable assumption requires that at least $K$ out of the $d$ variables are pairwise asymptotically independent (in particular, out of the $d(d-1)/2$ tail correlations, at least $K(K-1)/2$ must be zero). This restricts the size of $K$ in practice, with $\hat K=3$ for the wind gust data. However, a small $K$ often leads to many impure variables, all loaded by the same small set of factors. This typically results in large tail correlations among many impure variable pairs, explaining the numerous impure-impure points below the diagonal in Figure~\ref{fig:complete_analysis_spatial}.

This  suggests that requiring asymptotic independence among pure variables is too restrictive for many applications. Therefore, we outline a two-stage factor model in the Conclusion (Section~\ref{sec:conclusion}) that may allow to address this issue in future work.

\begin{figure}[thp!]
    \centering
    \begin{minipage}[t]{0.45\textwidth}
        \centering
        \includegraphics[width=\linewidth, height=0.35\textheight, keepaspectratio]{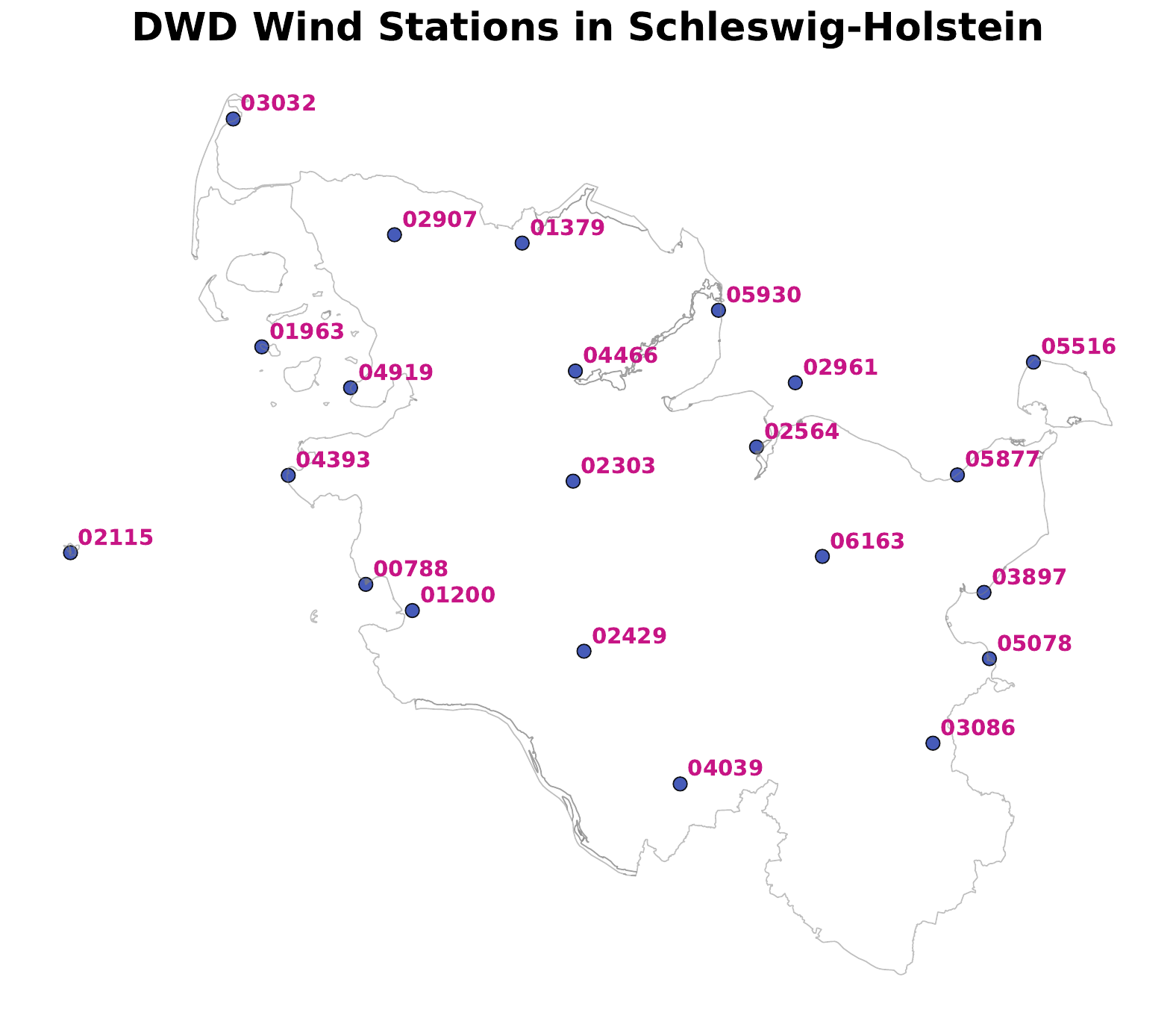}
    \end{minipage}
    \hfill
    \begin{minipage}[t]{0.45\textwidth}
        \centering
        \includegraphics[width=\linewidth, height=0.35\textheight, keepaspectratio]{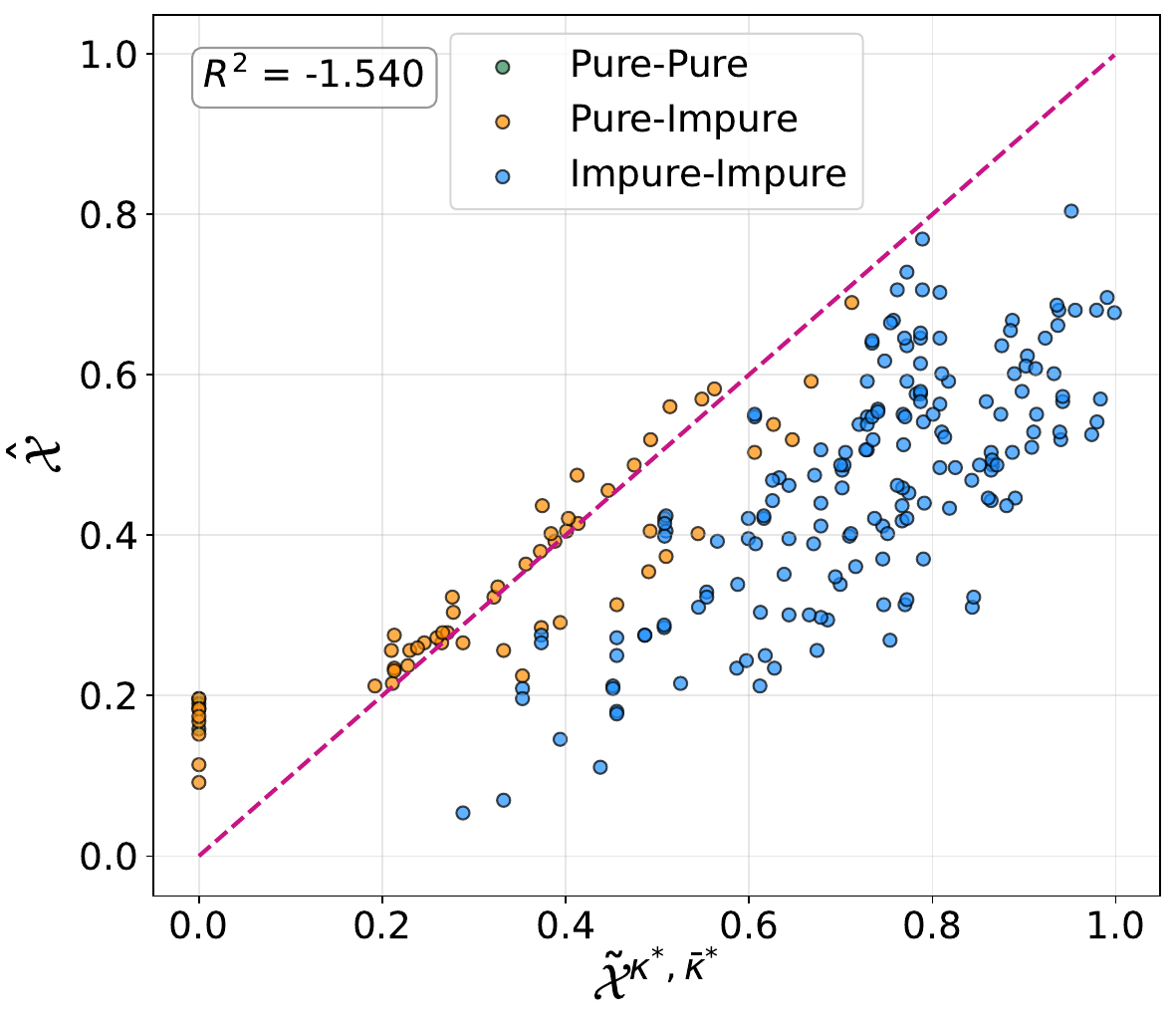}
    \end{minipage}    
    \caption{Results for the wind speed data set. Left: map of $d=22$ weather stations in Schleswig-Holstein. Right: empirical correlations from $\hat{\mathcal{X}}$ vs. fitted extremal correlations from $\tilde{\mathcal{X}}^{\hyperpurevar^*, \hypersparsity^*}$. }
    \label{fig:complete_analysis_spatial}
\end{figure}

\section{Conclusion}
\label{sec:conclusion}

A novel submodel of the classical heavy-tailed linear factor model in extremes has been introduced, leveraging a modified pure variable assumption introduced in \cite{Bin20}. Robust estimation methods for the unknown model parameters, including the number of latent factors $K$, have been developed. The methods have been equipped with finite sample statistical guarantees, accommodating cases where the dimension $d$ exceeds the sample size $n$. To the best of our knowledge, results of that kind are still scarce in the literature on extreme value analysis.

An important goal for future work is to extend the model and methods to better suit real-world data applications. The second case study in Section~\ref{sec:case_studies} illustrates that the pure variable assumption can be too restrictive, primarily due to the asymptotic independence between pure variables. To address this limitation, one may consider imposing a two-stage factor model of the form 
\[
\bm Y =   B_1  \bm{Z} + E,  \qquad \bm Z = B_2  \bm U,
\]
where the first-stage latent factors $\bm Z \in [0,\infty)^K$ are defined in terms of second-stage latent factors $\bm U \in [0,\infty)^K$. Asymptotic independence will only be imposed on the second-stage factor vector $\bm U$ (for instance as in (i)' of Remark~\ref{rem:fix-scale}), inducing asymptotic dependence in the first-stage factor vector $\bm Z$ unless $B_2 \in [0,\infty)^{K \times K}$ is diagonal. Note that our current model is a special case with $B_2$ as the identity matrix. 
Under suitable assumptions on $B_2$ (likely requiring it to be upper triangular) and the pure variable assumption $B_1 \in \Apure(K)$, we expect that $B_1$, $B_2$, and $A=B_1B_2$ are identifiable, subject to normalization constraints.

\appendix

\section{Proofs} \label{sec:proofs}

\subsection{Proofs for Section~\ref{sec:preliminaries}}

\begin{proof}[Proof of Lemma~\ref{lemma:stable-tail-dependence-spectral-measure}]
We start by proving `$\Longleftarrow$', so suppose that $\bm Y$ is regularly varying. We may hence choose a scaling sequence $(c_n)_n$ with associated non-degenerate exponent measure $\nu_{\bm Y}$ such that \eqref{eq:regular_variation} holds. Let $\alpha>0$ denote the tail index of $\bm Y$. Fix $j\in [d]$ and $y_j>0$. On the one hand we have,
by \eqref{eq:regular_variation} and homogeneity of $\bm \nu_Y$, 
\begin{align*} 
    \lim_{n\to\infty}  n \mathbb{P}\left\{ c_n^{-1} Y_j > y_j \right\}  
    =
    \nu_{\bm{Y}}\left( [0,\infty)^{j-1} \times (y_j, \infty) \times [0,\infty)^{d-j}\right)
    =
    y_j^{-\alpha} \kappa_j,
\end{align*}
where
$\kappa_j
:=
\nu_{\bm{Y}}\left( [0,\infty)^{j-1} \times (1, \infty) \times [0,\infty)^{d-j}\right) \ge 0.
$
On the other hand, by definition of $\bm Y$,
\[
n \mathbb{P}\left\{ c_n^{-1} Y_j > y_j \right\}
=
y_j^{-1}\frac{n}{c_n}.
\]
We conclude that $\lim_{n\to\infty} n/c_n =y_j^{1-\alpha} \kappa_j$ for all $y_j>0$ and $j\in[d]$. If $\alpha\ne 1$, this implies $\kappa_j=0$ for all $j\in[d]$, and hence $\nu_{\bm Y}([\bm 0, \bm 1]^c) \le \sum_{j\in[d]} \kappa_j =0$  by the union bound. This is a contradiction, as $\nu_{\bm Y}$ is non-degenerate by definition. Hence, $\alpha=1$, and then $\lim_{n\to\infty} n/c_n=\kappa_j$ for each $j \in [d]$; in particular, $\kappa=\kappa_j$ does not depend on $j$.

Now, for any $\bm x \in [0,\infty)^d$,
\begin{align} \label{eq:preasymptotic-stdf}
     n \mathbb{P}\left\{ F_1(X_1) > 1-\frac{x_1}{n} \textrm{ or } \dots \textrm{ or } F_d(X_d) > 1-\frac{x_d}{n} \right\} 
    &= \nonumber
    n \mathbb{P}\left\{ Y_1 > n/x_1 \textrm{ or } \dots \textrm{ or } Y_d > n/x_d \right\}
    \\&=
     n \Prob \{ c_n^{-1} \bm Y  \in (n/c_n)\cdot [\bm 0,  \bm 1/\bm x]^c\}.
\end{align}
A sandwich argument, invoking the convergence $\lim_{n \to \infty} n/c_n = \kappa$ and continuity of $t \mapsto \nu_{\bm Y}(tA)$ for any fixed $A$, implies that the left-hand side converges to
\[
\nu_{\bm Y}(\kappa [0, \bm 1/\bm x]^c)= \kappa^{-1} \nu_{\bm Y}([0,\bm 1/\bm x]^c).
\]
Hence, $L(\bm x)$ exists for all $\bm x \in [0,\infty)^d$.

We next prove `$\Longrightarrow$', so suppose that $L$ exists. Then, in view of \eqref{eq:preasymptotic-stdf} applied with $c_n=n$, we have, for any $\bm x \in \Eb_0$,
\[
n \Prob \{ c_n^{-1} \bm Y  \in  [\bm 0,  \bm x]^c\} = n \mathbb{P}\left\{ F_1(X_1) > 1-\frac{x_1^{-1}}{n} \textrm{ or } \dots \textrm{ or } F_d(X_d) > 1-\frac{x_d^{-1}}{n} \right\}, 
\]
which converges to $L(\bm 1/\bm x)$ for $\bm x \in (0, \infty)^d$, and to infinity if some coordinates of $\bm x$ are zero. 
Theorem 4.2 in \cite{hult2006regular} implies that $\bm Y$ is regularly varying with scaling sequence $c_n=n$, associated exponent measure defined by $\nu_{\bm Y}([\bm 0, \bm x]^c)=L(\bm 1/\bm x)$, and with tail index $1$.

The arguments in the previous paragraph also show that (i) and (ii) are met. It remains to prove (iii). Let $\nu_{\bm Y}$ denote the exponent measure from (ii). We then have, by the definition of the spectral measure and by \eqref{eq:relation-exponent-spectral-measure},
\begin{align*}
        L(\bm x) = \nu_{\bm Y}([0, \bm 1/ \bm x]^c) 
        &=
        \nu_{\bm Y} \Big( \Big\{ \bm y: \exists j \in [d] \text{ s.t. } y_j > x_j^{-1}  \Big\} \Big)
        \\&=
        \varsigma\cdot (\nu_{\alpha} \otimes \Phi_{\bm Y}) \left(\left\{ (r, \bm \lambda): \exists j \in [d] \text{ s.t. } r \lambda_j > x_j^{-1} \right\} \right)
        \\&=
        \varsigma
         \int_{\mathbb S_{+}^{d-1}} \max_{j \in [d]} \big( \lambda_j x_j \big)\, \diff  
        \Phi_{\bm Y}(\bm \lambda), \quad \bm x \in [0, \infty)^d,
    \end{align*}
as asserted. The claim for the $\| \cdot\|_\infty$-norm follows from  (ii), since $\varsigma=\nu_{\bm Y}(\{\bm y: \|\bm y\|_\infty > 1\}) = \nu_{\bm Y}([\bm 0, \bm 1]^c)=L(\bm 1)$. The claim for the $\|\cdot\|_1$-norm follows from the fact that $1=\int_{\mathbb S_{+}^{d-1}} \| \bm \lambda\|_ 1\, \diff  
        \Phi_{\bm Y}(\bm \lambda) = \sum_{j=1}^d \int_{\mathbb S_{+}^{d-1}} \lambda_j\, \diff  
        \Phi_{\bm Y}(\bm \lambda) = d \varsigma^{-1}$, where we have used that $\int_{\mathbb S_{+}^{d-1}} \lambda_j\, \diff  
        \Phi_{\bm Y}(\bm \lambda) = \varsigma^{-1}L(\bm e_j)=\varsigma^{-1}$ by the previous display.

        Finally, regarding uniqueness of $\Phi_{\bm Y}$, note that $\nu_{\bm Y}$ defined in (ii) is uniquely determined by $L$ as a consequence of Lemma B.1.32 in \cite{kulik2020heavy}.  Uniqueness of $\Phi_{\bm Y}$ then follows from the fact $\Phi_{\bm Y}$ is determined by $\nu_{\bm Y}$ by definition in \eqref{eq:def-spectral-measure-new}.
\end{proof}

\begin{proof}[Proof of Proposition~\ref{prop:spectral_measure_max_lin}]
Fix $(c_n)_n$ as in (i) of Model~\ref{cond:factormodel}, and note that $n \Prob(c_n^{-1} \bm Z \in \cdot) \to \nu_{\bm Z}$ in $\Mb(\Eb_0)$ with $\nu_{\bm Z}$ as in \eqref{eq:Lambda_z}. Define functions $f,g:[0,\infty)^K \setminus\{ \bm{0}\} \to \Eb_0$ by $f(\bm z):=A \bm z$ and $g(\bm z) = A \times_{\max} \bm{z}$.
The continuity and homogeneity of $f$ and $g$ imply that a set $B$ is separated from zero if and only if $f^{-1}(B)$, $g^{-1}(B)$ are separated from zero. Thus, Theorem B.1.21 in \cite{kulik2020heavy} implies that the vectors $f(\bm{Z})$ and $g(\bm{Z})$ are regularly varying  with respective exponent measures associated with $(c_n)_n$ given by $\nu_{A\bm{Z}} = \nu_{\bm{Z}} \circ f^{-1}$ and $\nu_{A \times_{\max} \bm{Z}} = \nu_{\bm{Z}} \circ g^{-1}$. Moreover, the tail index can easily seen to be $\alpha$ for both $f(\bm{Z})$ and $g(\bm{Z})$. 

Using that $f$ is a bounded linear operator, there exists a constant $L > 0$ such that $\|f(\bm{z})\| \le L \|\bm{z}\|$ for all $\bm{z} \in \mathbb{R}^K$, whence the tail of $f(\bm{Z})$ satisfies
\[
    \frac{\mathbb{P}\left\{ \| f( \bm{Z}) \| > x \right\}}{\mathbb{P}\left\{ \|\bm{E}\| > x \right\}} \le \frac{\mathbb{P}\left\{ \|\bm{Z} \| > x/L \right\}}{\mathbb{P}\left\{ \|\bm{E}\| > x \right\}} = \frac{\mathbb{P}\left\{ \|\bm{Z} \| > x/L \right\}}{\mathbb{P}\left\{ \|\bm{Z}\| > x \right\}} \frac{\mathbb{P}\left\{ \|\bm{Z} \| > x \right\}}{\mathbb{P}\left\{ \|\bm{E}\| > x \right\}} = O(1) o(1) = o(1),
\]
as $x \rightarrow \infty$. Thus, we can apply Lemma \ref{lemma:vanishing_noise} in Section \ref{sec:additional} to conclude that $\bm{Y}=_df(\bm Z) + \bm E$ is regularly varying with tail index $\alpha$ and exponent measure $\nu_{\bm{Y}} = \nu_{A \bm{Z}}$ associated with $(c_n)_n$. Likewise, since $\|g(\bm{z})\|_\infty \le (\max_{j \in [d], a\in [K]} A_{ja}) \|\bm{z}\|_\infty$ for any $\bm{z} \in \mathbb{R}^K$, we may replace $f$ by $g$ in the last display and then obtain, also by Lemma \ref{lemma:vanishing_noise}, that $\bm{Y}=_d g(\bm Z) \vee \bm E$ is regularly varying with tail index $\alpha$ and exponent measure $\nu_{\bm{Y}} = \nu_{A \times_{\max} \bm{Z}}$ associated with $(c_n)_n$.

We now compute the spectral measures of $A \bm{Z}$ and $A \times_{\max} \bm{Z}$. Since the involved computations are the same for both cases (mainly due to asymptotic independence), we only treat the first case. By the change-of-variable formula and the representation of $\nu_{\bm Z}$ in \eqref{eq:Lambda_z}, we have, 
\begin{align*}
\nu_{\bm Y}(B)  
=
\nu_{\bm Z} \circ f^{-1}(B)
=
\int_{[0,\infty)^d \setminus\{ \bm{0}\}} \bm 1_{B}(\bm{x}) \,\diff (\nu_{\bm Z} \circ f^{-1})(\bm{x})
&= \nonumber
\int_{[0,\infty)^K \setminus\{ \bm{0}\}} \bm 1_B \circ f(\bm{z}) \,\diff \nu_{\bm Z}(\bm{z})
\\&=  \nonumber
\sum_{a \in [K]} \int_{(0,\infty)} \bm 1_B(A_{1a}z_a,\dots,A_{da}z_a) \nu_\alpha(\diff z_a)
\\&=
\sum_{a \in [K]} \int_{(0,\infty)} \bm 1_B(z_a A_{\cdot a}) \nu_\alpha(\diff z_a)
\end{align*}
for every Borel set $B$ in $\mathbb{E}_0$ that is bounded away from zero. As a consequence, for any Borel set $C$ in $\mathbb{S}_+^{d-1}$, we have
\begin{align*}
\nu_{\bm{Y}} \circ T^{-1} ((1,\infty) \times C) 
&= 
\sum_{a \in [K]} \int_{(0,\infty)} \bm 1_{T^{-1}((1,\infty)\times C)}(z_a A_{\cdot a}) \nu_\alpha(\diff z_a)
\\&= 
\sum_{a \in [K]} \int_{(0,\infty)} \bm 1_{(1,\infty) \times C} \big(\| A_{\cdot a} \| z_a, A_{\cdot a} / \| A_{\cdot a}\| \big) \nu_\alpha(\diff z_a) 
\\ &= 
\sum_{a \in [K]} \| A_{\cdot a} \|^\alpha \delta_{A_{\cdot a}/ \| A_{\cdot a}\|}(C),
\end{align*}
which implies $\varsigma \equiv \nu_{\bm Y}(\{ \bm x \in \Eb_0: \| \bm x\| >1\}) =\sum_{a \in [K]} \| A_{\cdot a} \|^\alpha$ and hence \eqref{eq:Phi_A} by the definition of $\Phi_{\bm X}$ in \eqref{eq:def-spectral-measure-new}.

The assertion about $L$ in \eqref{eq:stable-tail-dp-A} now follows from Lemma~\ref{lemma:stable-tail-dependence-spectral-measure-old} and a straightforward calculation, 
after observing that $t_j$ from \eqref{eq:t_j-definition} can be written as $t_j=\kappa_j/\varsigma$ by \eqref{eq:t_j-representation}, with
\[
\kappa_j 
\equiv
\nu_{\bm{X}}\left( [0,\infty)^{j-1} \times (1, \infty) \times [0,\infty)^{d-j}\right)
=
\sum_{a \in [K]} A_{ja}^\alpha>0.
\]
    The proof is finished.
\end{proof}

\begin{proof}[Proof of Lemma~\ref{lem:tail-probabilities-maxlinear-stdf}]
Let $\bm Y$ be defined as in Lemma~\ref{lemma:stable-tail-dependence-spectral-measure}. Then, by elementary transformations and \eqref{eq:def-spectral-measure-new},
\begin{align*}
R^\cup_{\mathcal J}(\bm x)
&=
\lim_{n \to \infty}  n
\mathbb P \Big (\exists\, J \in \mathcal J\,  \forall j \in J: Y_j > \frac{n}{x_j} \Big)
\\&=
\nu_{\bm Y} \Big( \big\{ \bm y \in [0, \infty)^d \mid \exists\, J \in \mathcal J\,  \forall j \in J : y_j > x_j^{-1} \big\} \Big) 
\\&=
\varsigma (\nu_\alpha \otimes \Phi_{\bm Y} ) \Big( \big\{ (r, \bm \lambda) \in [0, \infty) \times \mathbb S_+^{d-1} \mid \exists\, J \in \mathcal J\,  \forall j \in J : r\lambda_j > x_j^{-1}  \big\} \Big) 
\\&=
\varsigma (\nu_\alpha \otimes \Phi_{\bm Y} ) \Big( \big\{ (r, \bm \lambda) \in [0, \infty) \times \mathbb S_+^{d-1} \mid  r >  \min_{J \in \mathcal J} \max_{j \in J} x_j^{-1} \lambda_j^{-1} \big\} \Big) 
\\&=
\varsigma \int_{\mathbb S_{+}^{d-1}} \int_{\min_{J \in \mathcal J} \max_{j \in J} x_j^{-1} \lambda_j^{-1}}^\infty \, \diff \nu_{\alpha}(r) \, \diff \Phi_{\bm Y}(\bm \lambda)
\\&=
\varsigma \int_{\mathbb S_{+}^{d-1}} \max_{J \in \mathcal J} \min_{j \in J} (x_j \lambda_j)\, \diff \Phi_{\bm Y}(\bm \lambda).
\end{align*}
Next, observe that Lemma~\ref{lemma:stable-tail-dependence-spectral-measure} and Proposition~\ref{prop:spectral_measure_max_lin} imply that the spectral measure of $\bm Y$ is uniquely determined by $L$ and hence by $(K, \bar A)$; and more specifically, 
\[
\Phi_{\bm Y}(\cdot) = \varsigma^{-1}  \sum_{a \in [K]} \| \bar A_{\cdot a} \| \delta_{\bar A_{\cdot a}/\| \bar A_{\cdot a} \|} (\cdot),
\]
where $\varsigma = \sum_{a \in [K]} \| \bar A_{\cdot a}\|$. As a consequence, by the previous two displays,
\begin{align*}
R^\cup_{\mathcal J}(\bm x)
&=
\sum_{a \in [K]} \| \bar A_{\cdot a} \| \max_{J \in \mathcal J} \min_{j \in J} (x_j \bar A_{ja}/\| \bar A_{\cdot a}\| )
=
\sum_{a\in[K]} \bigvee_{J \in \mathcal J} \bigwedge_{j \in J} \bar A_{ja}x_j
\end{align*}
as asserted.
The claim for $L^\cap_{\mathcal J}$ follows along similar lines.

Finally, regarding the extremal directions, we have
$
\Phi_{\bm Y}(\Sb_D) = \varsigma^{-1} \sum_{a \in [K]} \| \bar A_{\cdot a} \| \bm 1(\bar A_{\cdot a} \in \mathbb S_D),
$
which is positive if and only if, for some $a \in [K]$, $\supp(\bar A_{\cdot a}) = D$.
\end{proof}

\begin{proof}[Proof of Proposition~\ref{prop:injective-new}]
Let $\Xi \in \Phi_L(\Theta_L)$; i.e., there exists $(K,\bar A) \in \Theta_L$ such that  $\Phi_L(K,\bar A) = \Xi$. We need to show that $K$ is unique, and that $\bar A$ is unique up to column permutations. 
Let $I(\bar A)$ denote the pure variable set of $\bar A$, and let $I_a(\bar A)$ be defined as in \eqref{eq:Ia}, for $a\in[K]$.
Proposition \ref{prop:pur_var_ide-new2} implies that $\Xi$ uniquely determines $K$, $I(\bar A)$ and the partition $\{I_1(\bar A), \dots, I_K(\bar A)\}$ of $I(\bar A)$, up to permutation of labels. Subsequently, we fix one such permutation of labels.
We then obtain that the rows of $\bar A$ corresponding to $I(\bar A) \subset [d]$ are uniquely determined. Indeed, for any $\ell \in I(\bar A)$, there exists a unique $a=a_\ell$ such that $\ell \in I_a(\bar A)$. 
Hence, observing that $I_a(\bar A)= \{ j \in [d]: \bar A_{ja}=1, \bar A_{jb}=0 \forall b \ne a\}$, we obtain that $\bar A_{\ell \cdot}$ must be the unit vector $e_a \in \R^K$, with $1$ at the $a$th position.
Next, consider the rows of $\bar A$ corresponding to $[d] \setminus I(\bar A)$, so fix $j\in [d] \setminus I(\bar A)$. Then, for any $a\in[K]$, we have 
$
\Xi_{j,\ell} = \sum_{b \in [K]} \bar A_{jb} \wedge \bar A_{\ell b} = \bar A_{ja}
$
for all $\ell \in I_a(\bar A)$ by definition of $\mathcal X$ in \eqref{eq:def-phi-l}. Overall, we found that $\bar A$ is uniquely determined, up to the label permutation fixed earlier in the proof. Such a label permutation can be identified with a column permutation of $\bar A$.
\end{proof}

\begin{proof}[Proof of Proposition~\ref{prop:pur_var_ide-new2}]
(a) Fix $a\in [K]$. By the pure variable assumption, we may choose  an index $j=j_a \in [d]$ such that $\bar A_{ja}>0$ and $\bar A_{jb}=0$ for all $b\ne a$; since $\bar A$ has row sums one, we hence have  $\bar A_{ja}=1$ and $\bar A_{jb}=0$ for all $j\ne b$. We start by showing that the set $S=\{j_1, \dots, j_K\} \subseteq [d]$ is a clique in $G$. Indeed, since $\Xi=\Phi_L(K, \bar A)$ by assumption, for any distinct $j_a,j_{a'}\in S$, we have $\xi_{j_a,j_{a'}} = \chi(j_a, j_{a'}) = \sum_{b=1}^K \bar A_{j_ab} \wedge \bar A_{j_{a'} b} = \bar A_{j_aa}  \wedge \bar A_{j_{a'} a} + \bar A_{j_aa'}  \wedge \bar A_{j_{a'} a'}= 2 (1 \wedge 0)=0$ by Proposition~\ref{prop:spectral_measure_max_lin}. Hence, $S$ is a clique of cardinality $K$. In particular, $L \ge K$.

It remains to show that $L \le K$. For that purpose, let $C'$ be an arbitrary clique.
We need to show that $K':= |C'| \le K$.
For $j \in [d]$, let $S(j) = \{ a \in [K] : \bar A_{ja} > 0 \}$ denote the support of the $j$th row of $\bar A$. 
Now, again since $\Xi=\Phi_L(K, \bar A)$ by assumption and by Proposition~\ref{prop:spectral_measure_max_lin}, we have, for all distinct $j, \ell \in C'$,  
$
    0 = \xi_{j,\ell} = \chi(j, \ell) 
    = 
    \sum_{a \in [K]} \bar{A}_{ja} \wedge \bar{A}_{\ell a}.
$
As a consequence, ``for all $a \in [K]$: $\bar A_{ja}=0$ or $\bar A_{\ell a}=0$'', which in turn is equivalent to $S(j) \cap S(\ell) = \emptyset$. Hence, the sets $S(j)$ with $j\in C'$ are pairwise disjoint. Therefore, since $|S(j)| \ge 1$ by the non-zero row sum condition,
\begin{align} \label{eq:k-k}
K'=|C'| = \sum_{j \in C'} 1 \le \sum_{j\in C'} |S(j)| = \Big|\bigcup_{j\in C'} S(j)\Big| \le K,
\end{align}
which is the desired inequality.

(b) In view of (a), the clique $C$ has size $K$. Hence, by \eqref{eq:k-k} with $C'$ replaced by $C$, we have $K=|C| = \sum_{j \in C} 1 \le \sum_{j\in C} |S(j)| = |\bigcup_{j\in C} S(j)| \le K$, which is only possible if $|S(j)|=1$ for all $j\in C$. But $|S(j)|=1$ just means that $j$ is pure variable of $\bar A$.

(c) Fix $a \in [K]$ with $|I_a(\bar A)|\ge 2$ and $j\in I_a(\bar A)$; hence, $\bar A_{ja}=1$ and $\bar A_{jb}=0$ for all $b\ne a$. As a consequence, for any $\ell\in [K]$ with $\ell\ne j$, we have $\xi_{j,\ell} = \chi(j,\ell) = \sum_{b =1}^K \bar A_{jb} \wedge \bar A_{\ell b} = 1 \wedge \bar A_{\ell a} = \bar A_{\ell a}$. If $\ell \in I_a(\bar A)$, this expression is equal to 1. Conversely,  if $1=\xi_{j,\ell} = \chi(j,\ell) = \bar A_{\ell a}$, then $\bar A_{\ell a}>0$ and $\bar A_{\ell b}=0$ for all $b \ne a$, as $\bar A$ has row sums 1.

(d) By (a) and (b), the clique $C$ (which is uniquely determined by $\Xi$ by definition) has size $K$ and only consists of pure variables. Further, note that we have $|I_a(\bar A) \cap C| =1$ for all $a \in [K]$, because otherwise we would have $|I_a(\bar A) \cap C| \ge 2$ for some $a$, which is impossible by (c).
Now, fix $a \in [K]$, and let $j_a$ denote the unique element of $I_a(\bar A) \cap C$. Then, by part $c$, $I_a(\bar A) = \{j_a\} \cup \{\ell \in  [d]: \xi_{j_a,\ell}=1\}$, which implies the assertion.
\end{proof}

\subsection{Proofs for Section~\ref{sec:statistics-stdf-pure-variable}}

\begin{proof}[Proof of Theorem~\ref{thm:consistency_K_pure2}] 
    Let us consider the event
    \begin{equation}
    \label{eq:E-hp-event}
        \mathcal{E} := \mathcal{E}(\kappa_0) := \left\{ \underset{1 \leq j < \ell \leq d}{\max} \, |\hat{\chi}_{n, \knon}(j,\ell) - \chi(j,\ell)| \le \kappa_0 \right\}.
    \end{equation}
    Invoking the union bound 
    we know from Proposition \ref{prop:concentration-tail-correlation} that, with probability at least $1-\delta$, 
    \begin{multline*}
        \underset{1 \leq j < \ell \leq d}{\max} \, |\hat{\chi}_{n, \knon}(j,\ell) - \chi_{n, \knon}(j,\ell)| \le \frac{\sqrt{2}+ \sqrt{8 \ln(d(d-1)/\delta)} + \sqrt{8 \ln(4d(d-1)/\delta)}}{\sqrt{\knon}} \\ + \frac{6+2\ln(d(d-1)/\delta) + 4\ln(4d(d-1)/\delta)}{3\knon}.
    \end{multline*}
    Using the triangle inequality and the fact that $\ln(ad(d-1)) \leq 2\ln(ad)$ for any $a \geq 1$ we obtain that, with probability at least $1-\delta$,
    \begin{equation*}
        \underset{1 \leq j < \ell \leq d}{\max} \, |\hat{\chi}_{n, \knon}(j,\ell) - \chi(j,\ell)| \le \kappa_0,
    \end{equation*}
    or equivalently stated,
    $\mathbb{P}\left\{ \mathcal{E} \right\} \ge 1-\delta.$
    It is hence sufficient to show that (i) and (ii) from Theorem~\ref{thm:consistency_K_pure2} hold on the event $\mathcal E$.

Recall that $\hat{C}=\{\hat j_1, \dots, \hat j_{\hat K}\}$ is a maximum clique of size $\hat{K} = |\hat{C}|$ in the graph $\hat{G} = (V,\hat{E})$ with $V = [d]$ and edge set $\hat{E}=\{ (j,\ell) \, : \, \hat{\chi}_{n, \knon}(j,\ell) \le \hyperpurevar\}$. We start by pointing out that the following five claims are sufficient to prove \eqref{item:thm_consistency_K_pure_i} and \eqref{item:thm_consistency_K_pure_ii}:
    \begin{compactenum}[(a)]
\item \label{item:proof_thm_consistency_K_pure_1}
$\hat{K} = K$;
\item \label{item:proof_thm_consistency_K_pure_2}
$\hat{C} \subset I$;

\item \label{item:proof_thm_consistency_K_pure_3}
$\hat I \subset I$;

\item \label{item:proof_thm_consistency_K_pure_4}
There exists a permutation $\pi$ on $[K]$ such that $\hat j_{a} \in I_{\pi(a)}$ for all $a\in[K]$.

\item \label{item:proof_thm_consistency_K_pure_5}
With the permutation $\pi$ from \eqref{item:proof_thm_consistency_K_pure_4}, we have $I_{\pi(a)} \subset \hat I_{a}$ for all $a\in[K]$.

\end{compactenum}
Indeed, by \eqref{item:proof_thm_consistency_K_pure_3} and \eqref{item:proof_thm_consistency_K_pure_5}, we then have
\begin{align*}
    \hat I  \subset I = \bigcup_{a\in[K]} I_a \subset \bigcup_{a\in[K]} \hat I_{\pi^{-1}(a)} = \hat I,
\end{align*}
which yields $I = \hat I$ and hence \eqref{item:thm_consistency_K_pure_i}. Moreover, since both $(I_a)_a$ and $(\hat I_{\pi^{-1}(a)})_a$ are disjoint with $I_a \subset \hat I_{\pi^{-1}(a)}$ by \eqref{item:proof_thm_consistency_K_pure_5}, we must have equality $I_a = \hat I_{\pi^{-1}(a)}$ and hence \eqref{item:thm_consistency_K_pure_ii}.

To prove \eqref{item:proof_thm_consistency_K_pure_1}, let us first prove that $\hat{K} \ge K$, for which it is sufficient to find a clique in $\hat G$ of size $K$. By the pure variable assumption in Condition~\ref{cond:purevar}, we may choose, for any $a\in[K]$, an index $j = j_a \in [d]$ such that $\bar{A}_{ja} = 1$ and $\bar{A}_{jb} = 0$ for any $b \neq a$. Then ${S} = \{j_1,\dots,j_K\}$ is a clique in $\hat{G}$: indeed, for any $j, \ell \in {S}$ with $j \neq \ell$, we have $\chi(j,\ell) = 0$ by \eqref{eq:chiA}, whence, on the event $\mathcal{E}$,
$
    \hat{\chi}_{n, \knon}(j,\ell) \leq \chi(j,\ell) + \kappa_0 = \kappa_0 \leq \hyperpurevar
$
by the choice of $\hyperpurevar$. Hence, $S$ is a clique in $\hat G$ by the definition of $\hat G$.

We now show that $\hat{K} \leq K$. For that purpose, let $\hat{C}'$ be an arbitrary clique in $\hat{G}$.
We need to show that $\hat{K}' := |\hat{C}'| \leq K$. For $j\in [d]$, let $N(j) = \{a \in [K]: \bar{A}_{ja} \ge \eta\}$ denote the $\eta$-support of row $j$ of $\bar A$, i.e., the components in the $j$th row of $\bar{A}$ that are at least $\eta$. We start by showing that the sets $N(j)$ for $j \in \hat{C}'$ are disjoint. Indeed, suppose that there exist $j, \ell \in \hat{C}'$ with $N(j) \cap N(\ell) \ne \emptyset$. Then, by \eqref{eq:chiA}, 
\begin{equation*}
    \chi(j,\ell) = \sum_{a \in [K]} \bar{A}_{ja} \wedge \bar{A}_{\ell a} \geq \sum_{a \in N(j) \cap N(\ell)} \bar{A}_{ja} \wedge \bar{A}_{\ell a} \ge \eta,
\end{equation*}
which in turn implies that, on the event $\mathcal{E}$, $\hat{\chi}_{n, \knon}(j,\ell) \geq \chi(j,\ell) - \kappa_0 \ge \eta - \kappa_0 > \hyperpurevar$ by the assumption $\hyperpurevar < \eta/2 = \eta - \eta/2 < \eta- \hyperpurevar \le \eta -\kappa_0$.
This is a contradiction, as $j,\ell\in \hat C'$ implies that $\hat{\chi}_{n, \knon}(j,\ell) \leq \hyperpurevar$ by the clique property.
Next, observe that $|N(j)| \geq 1$ for all $j\in[d]$: indeed, if $j$ is pure, then $j\in I_a$ for some $a\in [K]$ and hence $\bar A_{ja}=1 \ge \eta$, and if $j$ is impure then $|N(j)| \geq 2$ by definition of $\eta$. Overall,
\begin{equation*}
    \hat{K}' =|\hat{C}'| = \sum_{j \in \hat{C}'}1 \leq \sum_{j \in \hat{C}'} |N(j)| = \Big|\bigcup_{j \in \hat{C}'} N(j)\Big|  \leq \big| [K] \big| = K.
\end{equation*}

To prove \eqref{item:proof_thm_consistency_K_pure_2} observe that, since $|\hat C|=K$, the previous display applied with $\hat C' = \hat C$ implies that $K=\sum_{j \in \hat C} |N(j)|$. Since each summand in the sum on right-hand side is at least 1, this is only possible if $|N(j)|=1$ for all $j \in \hat C$. Hence, since $|N(\ell)| \ge 2$ for $\ell\in J$, we must have $j\in I$ for all $j \in \hat C$.

To prove \eqref{item:proof_thm_consistency_K_pure_3}, or, equivalently, $J \subset \hat I^c$,
note that is sufficient to show that $\hat{\chi}_{n, \knon}(j,\ell) <1- \hyperpurevar$ for all $j \in I$  and  $\ell \in J$. Indeed, for any $\ell \in J$, we then have $\hat{\chi}_{n, \knon}(j,\ell) <1- \hyperpurevar$ for all $j \in \hat C$ by  \eqref{item:proof_thm_consistency_K_pure_2}, whence, by Step 5 in Algorithm~\ref{alg:purevar}, $\ell \notin \hat I$.
Hence, fix $j \in I$  and  $\ell \in J$. Choose $a\in[K]$ such that $j \in I_a$, which implies that $\bar A_{ja}=1$ and $\bar A_{jb}=0$ for all $b \ne a$ by Condition~\ref{cond:purevar}. Moreover, note that $\ell \in J$ implies that $\bar A_{\ell b} \le 1-\eta$ for all $b \in [K]$ by definition of $\eta$ in \eqref{eq:signal-strength-eta}. As a consequence, we obtain that 
$\chi(j,\ell) = \bar{A}_{\ell a} \le 1-\eta$ by \eqref{eq:chiA}. Hence, on the event $\mathcal E$, 
\begin{equation*}
    \hat{\chi}_{n, \knon}(j,\ell) \le \chi(j,\ell) +  \kappa_0 
    \le 1-\eta + \kappa_0 < 1-\hyperpurevar
\end{equation*}
by the assumption $\hyperpurevar < \eta-\kappa_0$.

Let us show \eqref{item:proof_thm_consistency_K_pure_4}. 
We start by showing that $|I_a \cap \hat{C}| = 1$ for all $a \in [K]$. Indeed, if this were not the case, there would exist some $a \in [K]$ such that $|I_a \cap \hat{C}| \geq 2$. This is impossible because, for $j, \ell \in I_a \cap \hat{C}$, the equality $\chi(j, \ell) = 1$ must hold by Proposition~\ref{prop:pur_var_ide-new2}(c). Hence, on the event $\mathcal{E}$ and observing that $\eta<1$ and $\hyperpurevar < \eta-\kappa_0$, we then have
\begin{equation*}
\hat{\chi}_{n, \knon}(j, \ell) \geq \chi(j,\ell) - \kappa_0 = 1 - \kappa_0 > \eta - \kappa_0 > \hyperpurevar,
\end{equation*}
which is a contradiction, as it is not possible for $j, \ell \in \hat{C}$.

Next, we have $\hat{K} = K$ by \eqref{item:proof_thm_consistency_K_pure_1}, which allows to write $\hat C=\{ \hat j_1, \dots, \hat j_K\}$. Moreover, $\hat C \subset I$ by \eqref{item:proof_thm_consistency_K_pure_2}, whence $\{ \hat j_1, \dots, \hat j_K\} = \hat C = \hat C \cap I = \bigcup_{b \in [K]} (\hat C \cap I_b)$. Hence, since the $K$ sets in the union on the right are disjoint and contain only one element as shown above, any $\hat j_a$ ($a \in [K]$) may be identified with exactly one $\hat C \cap I_b$ ($b \in [K]$); this defines the permutation $b=\pi(a)$ we are looking for.

It remains to prove \eqref{item:proof_thm_consistency_K_pure_5}. 
Fix $a\in[K]$ and note that $\hat j_a \in I_{\pi(a)}$ by \eqref{item:proof_thm_consistency_K_pure_4}.
If $|I_{\pi(a)}|=1$, we obtain that $I_{\pi(a)}=\{\hat j_a\} \subset \hat I_a$ . If $|I_{\pi(a)}| \ge 2$, 
let  $\ell \in I_{\pi(a)}$ be arbitrary with $\ell \ne \hat j_a$. 
Then, by Proposition~\ref{prop:pur_var_ide-new2}(c), we have $\chi(\hat j_a, \ell)=1$. Hence, on the event $\mathcal{E}$,
\[
\hat{\chi}_{n, \knon}(\hat{j_a},\ell) \ge \chi(\hat j_a, \ell) - \kappa_0 = 1- \kappa_0 \ge 1- \hyperpurevar
\]
by the assumption  $\hyperpurevar \ge \kappa_0$.
Therefore, by Step 5 in Algorithm~\ref{alg:purevar}, we have $\ell \in \hat I_a$. Since 
$\ell \in I_{\pi(a)} \setminus \{ \hat j_a\}$ was arbitrary and since $\hat j_a \in \hat I_a$ by definition, we have shown that  
$I_{\pi(a)} \subset \hat I_{a}$ as asserted.
\end{proof}

\begin{proof}[Proof of Theorem~\ref{thm:stat_guarantees_Abar}] 
Recall the event $\mathcal E = \mathcal E(\kappa_0)$ from \eqref{eq:E-hp-event} in the proof of Theorem~\ref{thm:consistency_K_pure2}. Since that event has probability at least $1-\delta$, it is sufficient to show that \eqref{item:thm_stat_guarantees_Abar_1} and \eqref{item:thm_stat_guarantees_Abar_3} hold on $\mathcal E$. Subsequently, we tacitly assume to work on this event. Recall that we then have $\hat{K} = K$, $\hat{I} = I$, and that there exists a permutation $\pi$ on the set $[K]$ such that $\hat{I}_a = I_{\pi(a)}$ for all $a \in [K]$; this follows from Theorem~\ref{thm:consistency_K_pure2}. In particular, we may define a matrix $\hat{P} \in S_K$ with entries $\hat{P}_{ba} = 1$ if $b = \pi(a)$ and $\hat{P}_{ba} = 0$ otherwise.

We start by proving \eqref{item:thm_stat_guarantees_Abar_1}, which will be an immediate consequence of the following two assertions:

\smallskip

\begin{compactenum}[(a)]
    \item \label{item:proof_thm_stat_guarantees_Abar_1}
    We have $\triangleA_{j\cdot} = (\bar A \hat P)_{j\cdot}$ for all $j \in I$.  As a consequence, 
     $
     \max_{j \in I} \| \triangleA_{j\cdot} - (\bar A \hat P)_{j\cdot} \|_2 = 0.
     $
    \item \label{item:proof_thm_stat_guarantees_Abar_3}
    We have
     $     
     \max_{j \in J} \| \triangleA_{j\cdot} - (\bar A \hat P)_{j\cdot} \|_2 \le  2\sqrt{s} \hyperpurevar.
     $
\end{compactenum}

\smallskip
For the proof of \eqref{item:proof_thm_stat_guarantees_Abar_1}, set $\tilde{A} = \bar A\hat{P}$.
We need to show that  $\triangleA_{ja} = \tilde{A}_{ja}$ holds for any $a \in [K]$ and $j \in \hat{I}$. 
For that purpose, fix $a \in [K]$ and consider the cases $j \in \hat I_a$ and $j \in \hat I \setminus \hat I_a$ separately. In the former case we have $\triangleA_{ja} = 1$ by the definition of $\triangleA$. 
Moreover, we have $\tilde{A}_{ja} = \sum_{b=1}^K \bar{A}_{jb} \hat{P}_{ba} = \bar{A}_{j\pi(a)} \hat{P}_{\pi(a)a} = \bar{A}_{j\pi(a)} =1$, where the final equality follows from the fact that $j \in \hat{I}_a = I_{\pi(a)}$. Overall, $\triangleA_{ja} = 1 = \tilde{A}_{ja}$ as asserted.
In the latter case, if $j \in \hat{I} \setminus \hat{I}_a$, we have $\triangleA_{ja} = 0$ by the definition of $\triangleA$. Following the same reasoning as above, we also find that $\tilde{A}_{ja} = \sum_{k=1}^K \bar{A}_{jk} \hat{P}_{ka} = \bar{A}_{j\pi(a)} \hat{P}_{\pi(a)a} = \bar{A}_{j\pi(a)} = 0$, since $j \in \hat I \setminus \hat I_a = I \setminus I_{\pi(a)}$. 

It remains to prove \eqref{item:proof_thm_stat_guarantees_Abar_3}.
For the ease of notation and without loss of generality, we make the blanket assumption that the permutation matrix $\hat{P}$ is the identity for the remainder of the proof. We will show below that the following three statements are true:

\smallskip
\begin{compactenum}
\renewcommand{\theenumi}{(B1)}
\renewcommand{\labelenumi}{\theenumi}
    \item \label{item:lemma_proof_thm_stat_guarantees_Abar_i}
        $\bar A_{\ell a} = 0 \implies \triangleAb_{\ell a } = 0$ for any $\ell \in J$ and $a \in [K]$;
\renewcommand{\theenumi}{(B2)}
\renewcommand{\labelenumi}{\theenumi}
        \item \label{item:lemma_proof_thm_stat_guarantees_Abar_ii}
        $\bar A_{\ell a} > 2 \hypersparsity  \implies \triangleAb_{\ell a } > 0$ for any $\ell \in J$ and $a \in [K]$;
\renewcommand{\theenumi}{(B3)}
\renewcommand{\labelenumi}{\theenumi}
        \item \label{item:lemma_proof_thm_stat_guarantees_Abar_iii}
        $\tau(\triangleAb_{\ell\cdot}) \le \kappa_0$ for all $\ell\in J$ with $\tau(\cdot)$ from \eqref{eq:P-and-tau}.
\end{compactenum}

\smallskip
\noindent
With these statements at hand, fix $\ell \in J$, and define $S_\ell = \supp(\bar{A}_{\ell \cdot})$. We then have
\begin{align}
\label{eq:alal}
\bar A_{\ell a} = \triangleA_{\ell a} = 0 \quad \forall a \in [K] \setminus S_\ell.
\end{align}
Indeed, for any $a\in [K] \setminus S_\ell$, we have $\bar A_{\ell a}=0$ by definition of $S_\ell$. Hence, by \ref{item:lemma_proof_thm_stat_guarantees_Abar_i}, we have $\triangleAb_{\ell a}=0$, which implies $\triangleA_{\ell a}=0$ by the definition in \eqref{eq:triangleA}. 
As consequence of \eqref{eq:alal}, we obtain that
\begin{equation*}
    \| \triangleA_{\ell \cdot} - \bar{A}_{\ell \cdot} \|_2^2 = \sum_{a \in S_\ell} \left(\triangleA_{\ell a} - \bar{A}_{\ell a} \right)^2.
\end{equation*}
Next, let $\hat S_\ell = \supp(\triangleAb_{\ell \cdot})$ and note that $\hat S_\ell \subset S_\ell$ by \ref{item:lemma_proof_thm_stat_guarantees_Abar_i}. Hence, 
\begin{equation}
\label{eq:Al-bound}
    \| \triangleA_{\ell \cdot} - \bar{A}_{\ell \cdot} \|_2^2 
    =  
    \sum_{a \in \hat S_\ell} \left(\triangleA_{\ell a} - \bar{A}_{\ell a} \right)^2 + \sum_{a \in S_\ell  \setminus \hat S_\ell} \left(\triangleA_{\ell a} - \bar{A}_{\ell a} \right)^2
    =: E_1 + E_2,
\end{equation}
where the second sum should be understand as zero if $S_\ell = \hat S_\ell$. 

We proceed by bounding $E_2$. For that purpose, note that $\triangleA_{\ell a}=0$ for any $a \in [K] \setminus \hat S_\ell$ by the definition of $\hat S_\ell$ and by definition of $\triangleA$ in \eqref{eq:triangleA}. Further, by \ref{item:lemma_proof_thm_stat_guarantees_Abar_ii}, we also have $\bar{A}_{\ell a} \le 2\hypersparsity $ for any $a \in [K] \setminus \hat S_\ell$. Hence,
\begin{equation}
\label{eq:E_2-bound}
    E_2 
    = \sum_{a \in S_\ell \setminus \hat S_\ell} \bar{A}_{\ell a}^2 
    \leq 
    |S_\ell \setminus \hat S_\ell| \cdot 4 \hypersparsity ^2.
\end{equation}

It remains to bound $E_1$. For that purpose, fix $a \in \hat S_\ell$ and note that $\tau(\triangleAb_{\ell\cdot}) \le \kappa_0 \le \hypersparsity $ by \ref{item:lemma_proof_thm_stat_guarantees_Abar_iii} and the choice of $\hypersparsity $. Hence, using the definitions of $\triangleAb$ and $\triangleA$ as given in \eqref{eq:Ahat2} and \eqref{eq:triangleA}, respectively, we have $\triangleA_{\ell a} = [\triangleAb_{\ell a} - \tau(\triangleAb_{\ell\cdot})]_+ = \triangleAb_{\ell a} - \tau(\triangleAb_{\ell\cdot}) =  \triangleAa_{\ell a} - \tau(\triangleAb_{\ell\cdot})$. Using the inequality $(x+y)^2 \le 2x^2 + 2y^2$,  which holds for any real numbers $x,y$, we obtain
\begin{equation}
\label{eq:E_1-bound1}
    E_1 
    = 
    \sum_{a \in \hat S_\ell} \left( \triangleAa_{\ell a} - \tau(\triangleAb_{\ell\cdot}) -\bar{A}_{\ell a} \right)^2 
    \le 
    2|\hat S_\ell| \hypersparsity ^2 +
    2 \sum_{a \in \hat S_\ell} \left( \triangleAa_{\ell a} -\bar{A}_{\ell a} \right)^2 . 
\end{equation}
Since we work on the event $\mathcal E$ from \eqref{eq:E-hp-event}, we have $|\hat{\chi}_{n,k}(j,\ell) - \chi(j,\ell)| \le \kappa_0$ for all $j\ne \ell$. Hence, 
since $\bar{A}_{\ell a} = |I_a|^{-1} \sum_{j \in I_a} \chi(j,\ell)$ and by the definition of $\triangleAa$ in \eqref{eq:Ahat1}, we obtain that $|\triangleAa_{\ell a} - \bar{A}_{\ell a} | \le |I_a|^{-1} \sum_{j \in I_a}|\hat{\chi}_{n,k}(j,\ell) - \chi(j,\ell)| \le \kappa_0 \le \hypersparsity $ for all $a \in [K]$. As a consequence, by \eqref{eq:E_1-bound1},
\begin{equation}
\label{eq:E_1-bound2}
E_1 \le  4|\hat S_\ell| \hypersparsity ^2,
\end{equation}
where we have used again that $\hat S_\ell \subset S_\ell$ by \ref{item:lemma_proof_thm_stat_guarantees_Abar_i}.
Overall, by \eqref{eq:Al-bound}, \eqref{eq:E_2-bound} and \eqref{eq:E_1-bound2}, we obtain that $$
    \| \triangleA_{\ell \cdot} - \bar{A}_{\ell \cdot}\|_2  
    \le 
    \sqrt{4 (|S_\ell \setminus \hat{S}_\ell| + |\hat{S}_\ell|) \hypersparsity ^2} 
    = 
    2 \sqrt{|S_\ell|} \hypersparsity 
    \le 
    2\sqrt{s} \hypersparsity ,
$$
and since $\ell \in J$ was arbitrary, we obtain the desired bound in \eqref{item:proof_thm_stat_guarantees_Abar_3}, which then in turn implies \eqref{item:thm_stat_guarantees_Abar_1}, 

It remains to prove \ref{item:lemma_proof_thm_stat_guarantees_Abar_i}-\ref{item:lemma_proof_thm_stat_guarantees_Abar_iii}.
    Regarding \ref{item:lemma_proof_thm_stat_guarantees_Abar_i}, let us fix $\ell \in J$ and $a \in [K]$. We then have $\chi(j,\ell) = \bar{A}_{\ell a}$ for any $j \in I_a$ by \eqref{eq:chiA}. As a consequence, since we work on the event $\mathcal{E}$, we have $\hat{\chi}_{n,k}(j,\ell) \le \chi(j,\ell) + \kappa_0 \le \hypersparsity $ where the last inequality follows from the choice of $\hypersparsity $ and $\bar{A}_{\ell a} = 0$. Since $j \in I_a = \hat{I}_a$ was arbitrary, we obtain that $\triangleAa_{\ell a} \le \hypersparsity $ by definition in \eqref{eq:Ahat1}, which in turn implies $\triangleAb_{\ell a} = 0$ by definition, see \eqref{eq:Ahat2}.

    To prove \ref{item:lemma_proof_thm_stat_guarantees_Abar_ii}, let us fix $\ell \in J$ and $a\in [K]$, we then have $\chi(j,\ell) = \bar{A}_{\ell a}$ for any $j \in I_a$ by \eqref{eq:chiA}. Then, on the event $\mathcal{E}$, we have $\hat{\chi}_{n,k}(j,\ell) \ge \bar{A}_{\ell a} - \kappa_0 > 2 \hypersparsity  - \kappa_0 \ge \hypersparsity $ where the last inequality follows from the choice of $\hypersparsity $. Since $j \in I_a = \hat{I}_a$ was arbitrary, we obtain that $\triangleAa_{\ell a} > \hypersparsity $ by definition in \eqref{eq:Ahat1}, which in turn implies that $\triangleAb_{\ell a} > \hypersparsity  > 0$ by definition given in \eqref{eq:Ahat2}.

    To prove \ref{item:lemma_proof_thm_stat_guarantees_Abar_iii}, let us fix $\ell \in J$. We can assume without loss of generality that the vector $\triangleAb_{\ell \cdot}$ is arranged in descending order, possibly after a reordering, i.e., $\triangleAb_{\ell 1} \ge \dots \ge \triangleAb_{\ell K}$, which in turn implies $\hat S_{\ell} := \supp(\triangleAb_{\ell \cdot}) = [p_\ell]$ where $p_\ell = |\hat S_\ell|$. With this notational convenience, we obtain, by \eqref{eq:P-and-tau},
    \begin{equation*}
        \tau(\triangleAb_{\ell \cdot}) 
        = \frac{1}{\rho}\Big(-1 + \sum_{a=1}^\rho \triangleAb_{\ell a} \Big),
        \qquad
        \rho = \max \Big\{ b \in [p_\ell] : \triangleAb_{\ell b} > \frac{1}{b} \Big( -1 + \sum_{a=1}^b \triangleAb_{\ell a} \Big) \Big\}.
    \end{equation*}
    For any $a \in [K]$, fix $j \in I_a$, and note that $\chi(j,\ell) = \bar{A}_{\ell a}$ by \eqref{eq:chiA} As a consequence, on the event $\mathcal{E}$, $\hat{\chi}_{n,k}(j,\ell) \le \bar{A}_{\ell a} + \kappa_0$. Since $j \in I_a =\hat{I}_a$ was arbitrary, we obtain that $\triangleAa_{\ell a} \le \bar{A}_{\ell a} + \kappa_0$ by definition in \eqref{eq:Ahat1}. Since $a$ was arbitrary, the inequality applies in particular to all $a \in \hat S_\ell$.
    For such $a$, we have $\triangleAa_{\ell a} = \triangleAb_{\ell a}$, see \eqref{eq:Ahat1} and \eqref{eq:Ahat2}, whence
    $\triangleAb_{\ell a} \le \bar{A}_{\ell a} + \kappa_0$. Hence, summing the above inequality over the first $b \in [p_\ell]$ terms of $\hat{\mathcal{S}}_\ell$ (noting that, by our notational convention, these correspond to the $b$ largest terms) yields 
    \[
    \sum_{a=1}^b \triangleAb_{\ell a} \leq \sum_{a=1}^b \bar{A}_{\ell a} + b \kappa_0 \le \sum_{a \in [K]} \bar{A}_{\ell a} + b \kappa_0 = 1 + b\kappa_0. 
    \]
    Rearranging terms, we obtain
    \begin{equation*}
        \frac{1}{b}\Big( -1 + \sum_{a=1}^b \triangleAb_{\ell a} \Big) \le \kappa_0.
    \end{equation*}
    Hence,
    \begin{equation*} 
        \triangleAb_{\ell b} - \frac{1}{b}\Big( -1 + \sum_{a=1}^b \triangleAb_{\ell a} \Big) \ge \triangleAb_{\ell b} - \kappa_0 > \hypersparsity  - \kappa_0 \ge 0,
    \end{equation*}
    where we used that $\triangleAb_{\ell b} > \hypersparsity $ since $b \in \hat{\mathcal{S}}_{\ell}$ and where the last inequality follows from the choice of $\hypersparsity $. 
    Since $b \in [p_\ell]$ was arbitrary, we obtain that $\rho = p_\ell$ and hence
    \begin{equation*}
        \tau(\triangleAb_{\ell \cdot}) = \frac{1}{\rho}\Big(-1 + \sum_{a=1}^\rho \triangleAb_{\ell a} \Big) \le \kappa_0,
    \end{equation*}
    as asserted in \ref{item:lemma_proof_thm_stat_guarantees_Abar_iii}.

We now prove \eqref{item:thm_stat_guarantees_Abar_3}, for which it is sufficient to show that
\begin{align} 
\label{eq:a2-a-support}
\supp(\triangleAb_{\ell \cdot}) = \supp(\triangleA_{\ell \cdot}) \qquad \forall \ell \in J .
\end{align}
Indeed, since $\eta > 2 \hypersparsity $ and $\hypersparsity  \ge \kappa_0$, we have $\supp(\bar{A}_{\ell \cdot}) = \supp(\triangleAb_{\ell\cdot})$ for any $\ell \in J$ by \ref{item:lemma_proof_thm_stat_guarantees_Abar_i} and \ref{item:lemma_proof_thm_stat_guarantees_Abar_ii}. For the proof of \eqref{eq:a2-a-support}, fix $\ell \in J$. By the definition of $\triangleA_{\ell a}$, we have $\triangleA_{\ell a} = 0$ whenever $\triangleAb_{\ell \cdot} = 0$ for some $a \in [K]$, which implies $\supp(\triangleA_{\ell \cdot}) \subseteq \supp(\triangleAb_{\ell \cdot})$. For the reverse inclusion, consider $a \in [K]$ such that $\triangleAb_{\ell a} > 0$. By definition in \eqref{eq:Ahat2}, this is equivalent to $\triangleAb_{\ell a} > \hypersparsity $. By \ref{item:lemma_proof_thm_stat_guarantees_Abar_iii}, we further know that $\tau := \tau(\triangleAb_{\ell \cdot}) \le \kappa_0$. Thus, by definition in \eqref{eq:triangleA}, $\triangleA_{\ell a} = [\triangleAb_{\ell a} - \tau]_+ = \triangleAb_{\ell a} - \tau > \hypersparsity  - \kappa_0 \ge 0$, where the last inequality follows from the choice of $\hypersparsity $.

Finally, the assertion in \eqref{item:thm_stat_guarantees_Abar_5} is an immediate consequence of $\hat{I} = I$, $\hat{K}=K$, and \eqref{item:thm_stat_guarantees_Abar_3}. 
\end{proof}

\begin{proof}[Proof of Proposition~\ref{prop:max-linear-frechet-second-order}]
    Fix $1 \le j < \ell \le d$. Using the properties of $\bm Z$ we have, for any $x_j, x_\ell > 0$,
    \begin{align*}
        F_{j,\ell}(x_j,x_\ell) 
        \equiv 
        \Prob\{ X_j \le x_j, X_\ell \le x_\ell \} 
        &=
        \mathbb{P}\Big\{ \max_{a \in [K]} A_{ja}Z_a \le x_j, \max_{a \in [K]} A_{\ell a}Z_a \le x_\ell \Big\}  \\
        &= 
        \prod_{a \in [K]} \mathbb{P}\Big\{ Z_a \le \min \Big\{ \frac{x_j}{A_{ja}}, \frac{x_\ell}{A_{\ell a}}\Big\} \Big\} 
        \\&= 
        \exp\Big\{ - \sum_{a \in [K]} \max\Big\{ \left( \frac{A_{j a}}{x_j} \right)^\alpha, \Big( \frac{A_{\ell a}}{x_\ell} \Big)^\alpha \Big\}\Big\},
    \end{align*}
with marginal cdfs
    \[
        F_k(x_k)= \Prob\{ X_k \le x_k\}
        =
        \mathbb{P}\Big\{ \max_{a\in[K]}  A_{ka} Z_a \le x_k \Big\} = \exp \Big\{ - \sum_{a \in [K]} \Big(\frac{A_{ka}}{x_k}\Big)^\alpha \Big\}, \quad k \in \{j,\ell\}.
    \]
Hence, for $u_k \in (0,1)$,
    \[
        F_k^{-1}(u_k) = \{ -\ln(u_k) \}^{-1/\alpha} \Big(\sum_{a \in [K]} A_{ka}^\alpha \Big)^{1/\alpha},
    \]
which implies that the copula $C_{j, \ell}$ of $(X_j, X_\ell)$ can be written as
\begin{align*}
        C_{j,\ell}(u_j,u_\ell) 
        = 
       F_{j, \ell} \big(( F_j^{-1}(u_j), F_\ell^{-1}(u_\ell) \big) 
       =
       \exp\Big\{ - \sum_{a \in [K]} \max\left\{ \bar{A}_{ja}(-\ln(u_j)), \bar{A}_{\ell a} (-\ln(u_\ell)) \right\} \Big\}.
\end{align*}
    Since the respective survival copula satisfies $\bar C_{j,\ell}(u_j,u_\ell) = u_j+v_j- 1 +C_{j,\ell}(1-u_j,1-u_\ell)$, we obtain that
    \[
        \chi_t(j, \ell) \equiv t\bar C_{j,\ell}\Big(\frac{1}{t}, \frac{1}{t} \Big)
        = 
        2 - t + t\exp\Big\{ \ln\Big( 1-\frac{1}{t} \Big)\sum_{a \in [K]}  \bar{A}_{ja} \vee \bar{A}_{\ell a}  \Big\}, \quad t > 1.
    \]
    Recalling $\chi(j,\ell) = \sum_{a \in [K]} \bar{A}_{ja} \wedge \bar{A}_{\ell a} = 2 - \sum_{a \in [K]} \bar{A}_{ja} \vee \bar{A}_{\ell a}$ by \eqref{eq:chiA}, we obtain that
    \begin{align*}
        \left| \chi_t(j, \ell) - \chi(j,\ell) \right| 
        &= \Big| \sum_{a \in [K]} \bar{A}_{j a} \vee \bar{A}_{\ell a} - t\Big( 1 -\exp\Big\{ \ln\Big( 1-\frac{1}{t} \Big)\sum_{a \in [K]}  \bar{A}_{ja} \vee \bar{A}_{\ell a}  \Big\}\Big)\Big|.
    \end{align*}
Consider the function $g:[1, \infty) \times [1,2]$ defined by 
    \[
        g(t,c) = t\Big[c - t\Big( 1 -\exp\Big\{ \ln\Big( 1-\frac{1}{t} \Big)c \Big\}\Big)\Big], \quad t \ge 1, c \in [1,2].
    \]
    Since $t \mapsto g(t,c)$ is a non-increasing function with
    $g(1,c) = c-1 \in [0,1]$ 
    we obtain that, for any $t > 1$,
    \begin{align*}
        \left| \chi_t(j, \ell) - \chi(j,\ell) \right|  
        &= 
        t^{-1} g\Big(t,\sum_{a \in [K]} \bar{A}_{ja} \vee \bar{A}_{\ell a} \Big)
        \\ & \le 
        t^{-1}
        g\Big(1, \sum_{a \in [K]} \bar{A}_{ja} \vee \bar{A}_{\ell a} \Big)
        = 
        t^{-1} \Big(\sum_{a \in [K]}\bar{A}_{ja} \vee \bar{A}_{\ell a}-1\Big)
        =
        t^{-1} \big( 1-\chi(j,\ell) \big).
    \end{align*}
    The proof is finished.
\end{proof}

\begin{proof}[Proof of Theorem~\ref{thm:high_dim2}]
Throughout the proof, we write $r=r_n, s=s_n$ etc. Since $C_{j,\ell}^{(n)} \in \mathcal{X}(r, M)$ for all $1\le j < \ell \le d$, we obtain that $D(k/n)$ from \eqref{eq:dkn} satisfies $D(k/n) \le M (k/n)^r$. Recall the definition of $\kappa_0$ from \eqref{eq:kappa0}. Observing that $2 \le \ln(8) \le \ln(4d) \le \ln(4d/\delta)$ for all $\delta \in (0,1)$ and $d \in \N_{\ge 2}$, we have
    \begin{align*}
        \kappa_0 
        &\le 
        M\Big( \frac{k}{n} \Big)^{r} + \frac{\sqrt{2}+ \sqrt{16 \ln(d/\delta)} + \sqrt{16 \ln(4d/\delta)}}{\sqrt{k}} + \frac{6+4\ln(d/\delta) + 8\ln(4d/\delta)}{3k}
        \\&\le
        M\Big( \frac{k}{n} \Big)^{r} + 9 \frac{\sqrt{\ln(4d/\delta)}}{\sqrt k} + 5 \frac{\ln(4d/\delta)}{ k}.
    \end{align*}
    Subsequently, we apply this inequality with $\delta=n^{-2}$ and $k$ as in \eqref{eq:def-k-asymptotic-2}. Since $c_k$ is assumed to be bounded away from zero, we have $k \ge 2$ for sufficiently large $n$ and hence
    \[
    (c_k/2)\{ \ln(4dn^2) \}^{1/(2r+1)} n^{2r/(2r+1)}  \le k \le c_k\{ \ln(4dn^2) \}^{1/(2r+1)} n^{2r/(2r+1)}.
    \]
    As a consequence,  the upper bound in the penultimate display can be upper bounded by
    \[
    M c_k^r \Big( \frac{\log(4dn^2)}{n} \Big)^{r/(2r+1)}
    +
    \frac{9}{\sqrt{c_k}} \Big( \frac{\log(4dn^2)}{n} \Big)^{r/(2r+1)}
    +
    \frac{5}{c_k}\Big( \frac{\log(4dn^2)}{n} \Big)^{2r/(2r+1)}.
    \]
    Since $\log(4dn^2)  \le n$ for all sufficiently large $n$ (by our assumption $\log d = o(n)$), we overall obtain that
    \[
    \kappa_0 
    \le 
    \left( Mc_k^r+ \frac{9}{\sqrt{c}_k} + \frac{5}{c_k}\right) \Big( \frac{\log(4dn^2)}{n} \Big)^{r/(2r+1)} \le  \hyperpurevar,
    \]
with $\hyperpurevar$ as defined in \eqref{eq:def-k-asymptotic-2}. Further, note that $\hyperpurevar$ converges to zero by assumption on $M, d$ and $r$. As a consequence, since $\liminf_{n \to \infty} \eta_n>0$ and since $c_k,c_\hyperpurevar$ are bounded in $n$, the inequality $\hyperpurevar< \eta_n/2$ is met for all sufficiently large $n$. Hence, all assumptions of Theorems~\ref{thm:consistency_K_pure2} and \ref{thm:stat_guarantees_Abar} are met, for sufficiently large $n$. For such $n$, we obtain that 
\[
    \mathbb{P}\big\{ \hat{I}_n = I_n, \hat{K}_n = K_n \big\} \ge 1-n^{-2}, 
    \quad 
    \mathbb{P}\big\{ \hat{s}_n = s_n \big\}  \ge 1-n^{-2}, 
    \quad
    \mathbb{P}\big\{ L_{\infty,2}(\triangleA, \bar{A}) \le 2 \sqrt{s} \hyperpurevar \big\}  \ge 1-n^{-2}.
\]
The assertion then follows from the Borel-Cantelli Lemma, observing that $c_\hyperpurevar$ is bounded in $n$ by assumption. 
\end{proof}

\section{Concentration results for the empirical tail correlation matrix} \label{sec:concentration}

Throughout, let $\bm X_1, \dots, \bm X_n$ be an i.i.d.\ sample of $\bm X$, a $d$-dimensional random vector with continuous marginal cdfs $F_1, \dots, F_d$. 
Recall the empirical pairwise extremal correlation with parameter $k \in [n]$ from \eqref{eq:empirical_correlation}, that is, 
\begin{align*}
    \hat{\chi}_{n, \knon}(j,\ell) = \frac1{\knon} \sum_{i=1}^n  \bm 1 ( R_{i,j}>n-\knon, R_{i,\ell}>n-\knon ), \qquad j, \ell\in[d]. 
\end{align*}
Further, let
\begin{equation*}
    \chi_{n, \knon}(j,\ell) = \frac{n}{\knon} \mathbb{P}\left\{ F_j(X_j) > 1-\knon /n , F_\ell(X_\ell) > 1- \knon/n \right\}.
\end{equation*}

\begin{proposition}
    \label{prop:concentration-tail-correlation}
    Let $\bm X_1, \dots, \bm X_n$ be i.i.d.\ $d$-dimensional random vectors with continuous marginal cdfs $F_1, \dots, F_d$. Fix $k\in[n]$ and $\delta \in (0,1)$. For any $j,\ell \in [d], j \ne \ell$, with probability larger than $1-\delta$, we have
    \begin{equation*}
        |\hat{\chi}_{n, \knon}(j,\ell) - \chi_{n, \knon}(j,\ell)| \leq \frac{\sqrt{2}+ \sqrt{8 \ln(2/\delta)} + \sqrt{8 \ln(8/\delta)}}{\sqrt{\knon}} + \frac{6+2\ln(2/\delta) + 4\ln(8/\delta)}{3\knon}.
    \end{equation*}
\end{proposition}

The proof of Proposition \ref{prop:concentration-tail-correlation} draws inspiration from \cite[Propositions 4.1 and 4.2]{clemencon2024regular}. Their approach relies on concentrating a covariance operator through a Bernstein-type inequality, as established in \cite[Theorem 3.8]{McDiarmid1998}, which we restate here for clarity. Here and throughout we adopt the shorthand notation $z_{i:j} = (z_i,\dots,z_j)$.

\begin{lemma}
    \label{lem:ber_mcd}
    Let $Z = (Z_{1:n})$ with $Z_i$ taking their values in a set $\mathcal{Z}$ and let $f$ be a real-valued function defined on $\mathcal{Z}^n$. Consider the positive deviation functions, defined for $i \in [n]$ and for $z_{1:i} \in \mathcal{Z}^i$
    \begin{equation*}
        g_i(z_{1:i}) = \mathbb{E}\left[ f(Z_{1:n} ) \mid Z_{1:i} = z_{1:i} \right] - \mathbb{E}\left[ f(Z_{1:n}) \mid  Z_{1:i-1} = z_{1:i-1} \right].
    \end{equation*}
    Denote by $b$ the maximum deviation
    \begin{equation*}
        b = \underset{i \in [n]}{\max} \, \underset{z_{1:i} \in \mathcal{Z}^i}{\sup} \, g_i(z_{1:i}),
    \end{equation*}
    and let ${v}$ denote the supremum of the sum of variances,
    \begin{equation*}
        {v} = \underset{(z_1,\dots,z_n) \in \mathcal{Z}^n}{\sup} \sum_{i=1}^n \Var[g_i(z_{1:i-1}, Z_i)]
    \end{equation*}
    If $b$ and ${v}$ are both finite, then, for all $\eps>0$,
    \begin{equation} \label{eq:mcdiarmid}
        \mathbb{P}\left\{ f(Z_{1:n}) - \mathbb{E}[f(Z_{1:n})] \geq \eps \right\} \leq \exp \left\{ -\frac{\eps^2}{2({v}+b\eps/3)}\right\}.
    \end{equation}
\end{lemma}

\begin{proof}[Proof of Proposition \ref{prop:concentration-tail-correlation}] 
Fix $j,\ell\in[d]$ with $j\ne \ell$. Let  $U_j = F_j(X_j)$ and $U_{i,j} = F_j(X_{i,j})$ for $i \in [n]$. Denote by $\hat F_{n,j}$ the empirical distribution function of $U_{1,j},\dots,U_{n,j}$.

In view of the fact the empirical extremal correlation is a rank-based estimator, we have, almost surely,
\begin{equation*}
    \hat{\chi}_{n, \knon}(j,\ell) = \knon^{-1} \sum_{i=1}^n \mathds{1}_{\{ \hat F_{n,j}(U_{i,j}) > 1-k/n \textrm{ or } \hat F_{n,\ell}(U_{i,\ell}) > 1 - k/n  \}}.
\end{equation*}
Next, by the inclusion-exclusion principle and the triangle inequality, we have
\begin{align}
    \label{eq:inc_exc}
        &\phantom{{}={}} \nonumber
        |\hat{\chi}_{n, \knon}(j,\ell) - \chi_{n, \knon}(j,\ell)|  
        \\ &= 
        \Big|\frac{1}{k} \sum_{i=1}^n \mathds{1}_{\{ \hat F_{n,j}(U_{i,j}) > 1-k/n \textrm{ or } \hat F_{n,\ell}(U_{i,\ell}) > 1 - k/n  \}} - \frac{n}{k} \mathbb{P}\left\{ U_{j} > 1-k/n \textrm{ or } U_{\ell} > 1-k/n \right\} \Big| \nonumber \\
        &\le D_1 + D_2,
\end{align}
where 
\begin{align*}
    D_1 
    &= 
    \Big|\frac{1}{k} \sum_{i=1}^n \mathds{1}_{\{ \hat F_{n,j}(U_{i,j}) > 1-k/n \textrm{ or } \hat F_{n,\ell}(U_{i,\ell}) > 1 - k/n  \}} - \frac{1}{k} \sum_{i=1}^n \mathds{1}_{\{ U_{i,j} > 1-k/n \textrm{ or } U_{i,\ell} > 1 - k/n  \}}\Big|, \\ 
    D_2 
    &=
    \Big|\frac{1}{k} \sum_{i=1}^n \mathds{1}_{\{ U_{i,j} > 1-k/n \textrm{ or } U_{i,\ell} > 1 - k/n  \}} - \frac{n}{k} \mathbb{P}\left\{ U_{j} > 1-k/n \textrm{ or } U_{\ell} > 1-k/n \right\} \Big|.
\end{align*}
    We start by treating  $D_2$. Let $\bm{Z}_{1:n} = (\bm{Z}_1,\dots,\bm{Z}_n)$ where $\bm{Z}_i = (U_{i,j}, U_{i,\ell})$ and $\bm{z}_{1:n} = (\bm{z}_1,\dots,\bm{z}_n)$ where $\bm{z}_i = (u_{i,j}, u_{i,\ell}) \in [0,1]^2$ for $i \in [n]$. Consider the map $f_{n} \colon ([0,1]\times [0,1])^n \rightarrow \mathbb{R}$ defined by
    \begin{align*}
f_n(\bm{z}_{1:n}) = \frac{1}{k}\left| \sum_{i=1}^n \left( \mathds{1}_{\{ u_{i,j} > 1-k/n \textrm{ or } u_{i,\ell} > 1 - k/n  \}}  - \mathbb{P}\left\{ U_{j} > 1-k/n \textrm{ or } U_{\ell} > 1-k/n \right\} \right) \right|.
\end{align*}
We will apply Lemma \ref{lem:ber_mcd} to the function $f = f_n$. To do so, we derive an upper bound on the maximum deviation term $b$ and on the maximum sum of variances $v$ from the statement. For that purpose, note that, by the inverse triangle inequality, 
\begin{align}
\label{eq:fnidff}
|f_n(\bm z_{1:n}) - f_n(\bm z'_{1:n})| 
\le 
\frac1k \Big|\sum_{i=1}^n \mathds{1}_{\{u_{i,j} > 1-k/n \textrm{ or } u_{i,\ell} > 1 - k/n\}} - \mathds{1}_{\{u'_{i,j} > 1-k/n \textrm{ or } u'_{i,\ell} > 1 - k/n\}} \Big|
\end{align}
for any $\bm{z}_{1:n}, \bm z'_{1:n} \in ([0,1] \times [0,1])^n$. As a consequence, the maximum deviation $b$ is bounded by $1/k$: indeed, for any $i\in[n]$ and any $\bm{z}_{1:i} \in ([0,1]\times [0,1])^i$, we have, by independence among $\bm{Z}_i$'s,
\begin{multline*}
    g_i(\bm{z}_{1:i}) = \mathbb{E}\left[ f_n(\bm{z}_1,\dots,\bm{z}_{i-1}, \bm{z}_i, \bm{Z}_{i+1},\dots,\bm{Z}_n) - f_n(\bm{z}_1,\dots,\bm{z}_{i-1},\bm{Z}_i,\bm{Z}_{i+1},\dots,\bm{Z}_n) \right] \\
    \leq \frac{1}{k} \mathbb{E}\Big[ \Big| \mathds{1}_{\{ u_{i,j} > 1-k/n \textrm{ or } u_{i,\ell} > 1-k/n \}} - \mathds{1}_{ \{ U_{i,j} > 1-k/n \textrm{ or } U_{i,\ell} > 1-k/n \} }\Big| \Big] 
    \leq \frac{1}{k}.
\end{multline*}
It remains to bound the variance term. Since $\mathbb{E}[g_i(\bm{z}_{1:i-1}, \bm{Z}_i)] = 0$ for any $i\in[n]$, we may write, for $\bm{Z}_i' = (U_{i,j}', U_{i,\ell}')$ an independent copy of $\bm{Z}_i = (U_{i,j}, U_{i,\ell})$ and by \eqref{eq:fnidff},
\begin{align*}
\Var(g_i(\bm{z}_{1:i-1}, \bm{Z}_i)) 
&= 
\mathbb{E}\left[ g_i^2(\bm{z}_{1:i-1},\bm{Z}_i') \right] 
\\&= 
\mathbb{E}\Big[ \Big| \mathbb{E}\big[f_n(\bm{z}_{1:i-1}, \bm{Z}_i', \bm{Z}_{i+1},\dots,\bm{Z}_n) - f_n(\bm{z}_{1:i-1}, \bm{Z}_i, \bm{Z}_{i+1},\dots,\bm{Z}_n)\big] \Big|^2 \Big] 
\\&\leq 
\frac{1}{k^2} \mathbb{E}\Big[ \Big| \mathds{1}_{\{ U_{i,j}' > 1-k/n \textrm{ or } U_{i,\ell}' > 1-k/n \}} - \mathds{1}_{\{ U_{i,j} > 1-k/n \textrm{ or } U_{i,\ell} > 1-k/n \}}\Big|^2 \Big] 
\\&\leq 
\frac{2}{k^2} \mathbb{P}\left\{ U_{i,j} > 1 - k/n \textrm{ or } U_{i,\ell} > 1-k/n \right\} \\
&\leq 
\frac{4}{kn}.
\end{align*}
Hence ${v}$ is bounded from above by $4/k$. As a consequence, since the upper bound in \eqref{eq:mcdiarmid} is increasing in $b$ and $v$, we obtain, by an application of Lemma \ref{lem:ber_mcd},
\begin{equation}
\label{eq:bern1}
    \mathbb{P}\left\{ f(\bm{Z}_{1:n}) - \mathbb{E}[f(\bm{Z}_{1:n})] \geq \eps \right\} 
    \leq 
    \exp\left\{- \frac{k\eps^2}{8+2\eps/3}\right\}.
\end{equation}

Next, let us bound $\mathbb{E}[f(\bm{Z}_{1:n})]$. Writing
\begin{equation*}
    \Delta_i = \mathds{1}_{ \{ U_{i,j} > 1-k/n \textrm{ or } U_{i,\ell} > 1-k/n \} } - \mathbb{P}\left\{ U_{i,j} > 1-k/n \textrm{ or } U_{i,\ell} > 1-k/n \right\}, \quad i \in [n],
\end{equation*}
we have, by Jensen's inequality and independence of $\Delta_1, \dots, \Delta_n$,
\begin{align*}
    \big( \mathbb{E}[f(\bm{Z}_{1:n})]   \big)^2
    = 
    \frac{1}{k^2} \Big( \mathbb{E}\Big[ \sum_{i=1}^n \Delta_i \Big] \Big)^2
    &\leq 
    \frac{1}{k^2} \mathbb{E}\Big[ \Big( \sum_{i=1}^n \Delta_i \Big)^2 
    \Big] 
    =
    \frac{1}{k^2} \Var  \Big( \sum_{i=1}^n \Delta_i \Big) 
    =
    \frac{n}{k^2}  \Var(\Delta_1)  
    \\&\leq 
    \frac{n}{k^2}  \mathbb{P}\left\{ U_{1,j} > 1-k/n \textrm{ or } U_{1,\ell} > 1-k/n \right\} 
    \leq 
    \frac{2}{k}.
\end{align*}
Hence, $|\mathbb{E}[f(\bm{Z}_{1:n})] | \le \sqrt{2/ k}$, so that, 
\begin{align}
\label{eq:bound2}
\Prob\big(D_2 > \eps_2 + \sqrt{2/k} \big)
&= \nonumber
\mathbb{P}\left\{ f(\bm{Z}_{1:n}) - \mathbb{E}[f(\bm{Z}_{1:n})] + \mathbb{E}[f(\bm{Z}_{1:n})] > \eps_2 + \sqrt{2/k} \right\} 
\\&\le 
\mathbb{P}\left\{ f(\bm{Z}_{1:n}) - \mathbb{E}[f(\bm{Z}_{1:n})]  > \eps_2 \right\} \le \delta/2, 
\end{align}
where, in view of \eqref{eq:bern1}, the last inequality is correct if we choose
\[
\eps_2= \sqrt{\frac{8 \ln(2/\delta)}k} + \frac{2\ln(2/\delta)}{3k}.
\]
Indeed, a simple calculation shows that $\eps_2^2 \ge 8 \ln(2/\delta)/k + 2 \eps_2  \log(2/\delta)/(3k)$, which in turn is equivalent to $\exp(-k\eps_2^2/(8+2\eps_2/3)) \le \delta/2$.

It remains to bound $D_1$ from the right-hand side in \eqref{eq:inc_exc}. One has
\begin{align*}
    D_1
    &\leq 
    \frac{1}{k} \sum_{i=1}^n \left| \mathds{1}_{ \{ \hat F_{n,j}(U_{i,j}) > 1-k/n \textrm{ or } \hat F_{n,\ell}(U_{i,\ell}) > 1-k/n \} } -  \mathds{1}_{\{ U_{i,j} > 1-k/n \textrm{ or } U_{i,\ell} > 1 - k/n  \}} \right| \\
    &= 
    \frac{1}{k} \sum_{i=1}^n \left| \mathds{1}_{ \{ \hat F_{n,j}(U_{i,j}) \leq 1-k/n \textrm{ and } \hat F_{n,\ell}(U_{i,\ell}) \leq 1-k/n \} } -  \mathds{1}_{\{ U_{i,j} \leq 1-k/n \textrm{ and } U_{i,\ell} \leq 1 - k/n  \}} \right|
    \leq D_1(j) + D_1(\ell)
\end{align*}
where
\begin{align*}
    D_1(j) = \frac{1}{k} \sum_{i=1}^n \left| \mathds{1}_{ \{ \hat F_{n,j}(U_{i,j}) \leq 1-k/n\} } -  \mathds{1}_{\{ U_{i,j} \leq 1-k/n\}} \right| 
\end{align*}
Hence, by equidistribution, $\Prob(D_1> \eps) \le 2 \Prob\{ D_1(j)>\eps/2 \}$, so it suffices to bound $D_1(j)$. We have $D_1(j) \le E_1 + E_2$,
where
\begin{align*}
    E_1 = 
    \frac{1}{k} \sum_{i=1}^n \left| \mathds{1}_{ \{ \hat F_{n,j}(U_{i,j}) \leq 1-k/n\} } -  \mathds{1}_{\{ U_{i,j} \leq \hat F_{n,j}^\leftarrow(1-k/n)\}} \right|, \qquad 
    E_2
    = 
    \frac{1}{k} \sum_{i=1}^n \left| \mathds{1}_{ \{ U_{i,j} \leq \hat F_{n,j}^\leftarrow(1-k/n)\} } -  \mathds{1}_{\{ U_{i,j} \leq 1-k/n\}} \right|,
\end{align*}
and where $\hat F_{n,j}^\leftarrow$ denotes the left-continuous generalized inverse of $\hat F_{n,j}$.  
Some thoughts reveal that, from the $n$ summands in $E_1$, only one is non-zero, which gives the bound $E_1 \le 1/k$. Regarding $E_2$, note that
\begin{equation*}
    \hat F_{n,j}^\leftarrow(1-k/n) = U_{(n:n-k),j},
\end{equation*}
where $U_{(n:1),j} < \dots < U_{(n:n),j}$ denotes the order statistics of the sample $U_{1,j},\dots,U_{n,j}$. Then, writing $A=\{ 1-k/n \leq U_{(n:n-k),j} \}$,
\begin{align*}
    E_2
    &= 
    \frac{1}{k} \sum_{i=1}^n \left| \mathds{1}_{ \{ U_{i,j} > U_{(n:n-k),j} \} } - \mathds{1}_{ \{ U_{i,j} > 1-k/n \} }\right| \\
    &= 
    \mathds{1}_{A} \times \frac{1}{k} \sum_{i=1}^n \left( \mathds{1}_{ \{ U_{i,j} > 1-k/n \} } - \mathds{1}_{ \{ U_{i,j} > U_{(n:n-k),j} \} }\right) 
    +
    \mathds{1}_{ A^c } \times \frac{1}{k} \sum_{i=1}^n \left( \mathds{1}_{ \{ U_{i,j} > U_{(n:n-k),j} \} } - \mathds{1}_{ \{ U_{i,j} > 1-k//n \} }\right)
    \\&= 
     \mathds{1}_{ A } \times  \Big( \frac{1}{k} \sum_{i=1}^n \mathds{1}_{ \{ U_{i,j} > 1-k/n \} }  - 1 \Big)
    +
    \mathds{1}_{A^c } \times \Big( 1 - \frac{1}{k} \sum_{i=1}^n  \mathds{1}_{ \{ U_{i,j} > 1-k//n \} }\Big),
\end{align*}
where we used that $\sum_{i=1}^n \mathds{1}_{ \{ U_{i,j} > U_{(n:n-k),j} \} } = k$ as a consequence of the fact that there are no ties with probability one. Therefore,
\begin{align*}
    E_2
    &=  \Big| \frac{1}{k} \sum_{i=1}^n \mathds{1}_{ \{ U_{i,j} > 1-k/n \}} - 1  \Big|,
\end{align*}
where the last equality follows from the fact that, on the event $A=\{1-k/n \leq U_{(n:n-k),j}\}$, we have $k^{-1}\sum_{i=1}^n \mathds{1}_{ \{ U_{i,j} \ge 1-k/n \} } \ge k^{-1}\sum_{i=1}^n \mathds{1}_{ \{ U_{i,j} \ge U_{(n,n-k),j} \} } = k$, while on the complementary event, the reverse inequality holds. 
Next, note that $\sum_{i=1}^n \mathds{1}_{ \{ U_{i,j} > 1-k/n \} }$ follows a binomial distribution with success probability $k/n)$. As a consequence, by the classical Bernstein inequality \cite[Theorem 2.7]{McDiarmid1998},
\begin{align*}
    \mathbb{P}\big\{E_2 \geq \eps \big\}
    &\leq 
    \mathbb{P}\Big\{ \Big| \sum_{i=1}^n \mathds{1}_{ \{ U_{i,j} > 1-k/n \} } - k \Big| \geq k\eps \Big\} 
    \leq 
    2 \exp \Big\{ -\frac{k\epsilon^2}{2+2\eps/3} \Big\}.
\end{align*}
Overall, 
\begin{align} \label{eq:bound1}
\Prob\big\{ D_1 > \eps_1 + 2/k \big\}
\le 
2 \Prob\big\{ D_1(j) > \eps_1/2 + 1/k\big\}
\le
2 \Prob \big\{E_2 > \eps_1/2 \big\}
\le
4 \exp \Big\{ -\frac{k\epsilon_1^2}{8+4\eps_1/3} \Big\},
\end{align}
and similar as in the bound for $D_2$, the right-hand side is bounded by $\delta/2$ if we choose
\[
\eps_1= \sqrt{\frac{8 \ln(8/\delta)}{k}} + \frac{4\ln(8/\delta)}{3k}.
\]
The assertion then follows from \eqref{eq:bound2} and \eqref{eq:bound1}.
\end{proof}

\section{Further auxiliary results}
\label{sec:additional}

\begin{lemma} \label{lemma:stable-tail-dependence-spectral-measure-old}
    Suppose $\bm Y \in [0,\infty)^d$ is regularly varying with tail index $\alpha>0$ and spectral measure $\Phi_{\bm Y}$, and has continuous marginal cdfs $G_1, \dots, G_d$. Then the limits 
    \begin{align} \label{eq:t_j-definition}
    t_j :=  \lim_{x \to \infty} \frac{\Prob(Y_j > x)}{\Prob(\|Y\|>x)} \in [0, \infty), \quad j \in [d],
    \end{align}
    exist. 
    If $t_j>0$ for all $j \in [d]$, the stable tail dependence function of $\bm Y$ is given by 
    \begin{align*}
        L(\bm x) 
        =
        \int_{\mathbb S_{+}^{d-1}} \max \Big( \frac{\lambda_j^{\alpha} x_j}{t_j} \Big)\, \diff  
        \Phi_{\bm Y}(\bm \lambda), \qquad \bm x \in [0,\infty)^d.
    \end{align*}
\end{lemma}

\begin{proof}[Proof of Lemma~\ref{lemma:stable-tail-dependence-spectral-measure-old}]
The limit in the definition of $t_j$ exists by regular variation. More precisely, choosing an arbitrary scaling sequence $(c_n)_n$ with associated exponent measure $\nu_{\bm Y}$ such that \eqref{eq:regular_variation} holds, we have
\begin{align} \label{eq:t_j-representation}
t_j 
= 
\frac{\nu_{\bm{Y}}\left( [0,\infty)^{j-1} \times (1, \infty) \times [0,\infty)^{d-j}\right)}{\nu_{\bm{Y}}\left( \{ \bm y \in \Eb_0: \| \bm y \| > 1\}\right)}
= 
\frac{\kappa_j}{\varsigma};
\end{align}
where $\varsigma =\nu_{\bm{Y}}\left( \{ \bm y \in \Eb_0: \| \bm y \| > 1 \}\right)>0$, where
\[
\kappa_j
:=
\nu_{\bm{Y}}\left( [0,\infty)^{j-1} \times (1, \infty) \times [0,\infty)^{d-j}\right) \ge 0,
\]
and where the fraction $\kappa_j / \varsigma$ does not depend on $(c_n)_n$ by property (a) after \eqref{eq:regular_variation}.

Subsequently, assume $t_j>0$ (and hence $\kappa_j>0$) for all $j \in [d]$.
We only prove the assertion regarding the stable tail dependence function $L$ for $\bm x \in(0,\infty)^d$; the remaining points can be treated similarly.
By \eqref{eq:regular_variation} and homogeneity of $\nu_{\bm Y}$, we have, for any $j \in [d]$ and any $x_j>0$,
\begin{align} \label{eq:cdf-inverse-convergence}
    \lim_{n\to\infty}  \frac{1}{n \mathbb{P}\left\{ c_n^{-1} Y_j > x_j \right\}}  
    =
    \frac{1}{\nu_{\bm{Y}}\left( [0,\infty)^{j-1} \times (x_j, \infty) \times [0,\infty)^{d-j}\right)}
    =
    x_j^{\alpha} \kappa_j^{-1}.
\end{align}
Write $Q_j$ for the left-continuous generalized inverse of $1/(1-G_j)$. By Proposition 0.1 in \cite{Res07}, the convergence in \eqref{eq:cdf-inverse-convergence} is equivalent to
\[
    \lim_{n \to \infty}  \frac{Q_j(nx_j)}{c_n} = (x_j\kappa_j)^{1/\alpha}.
\]
By a change of variable, we have shown that, for all $j \in [d]$ and all $x_j>0$,
\[
q_{nj}(x_j) := \frac{Q_j(n/x_j)}{c_n} \to x_j^{-1/\alpha} \kappa_j^{1/\alpha} =: q_j(x_j), \quad n \to \infty.
\]
Then, by eventual continuity of the margins of $\bm{Y}$, we have, for any $\bm x \in(0,\infty)^d$,
\begin{align*}
    L(\bm{x}) &= \underset{n \rightarrow \infty}{\lim} n \mathbb{P}\left\{ G_1(Y_1) > 1-\frac{x_1}{n} \textrm{ or } \dots \textrm{ or } G_d(Y_d) > 1-\frac{x_d}{n} \right\} 
    \\&=
    \underset{n \rightarrow \infty}{\lim} n \mathbb{P}\left\{ Y_1 > Q_1(n/x_1) \textrm{ or } \dots \textrm{ or } Y_d > Q_d(n/x_d) \right\}
    \\&=
    \lim_{n \to \infty} n \Prob \{ c_n^{-1} \bm Y \in [\bm 0, \bm q_n]^c\},
\end{align*}
where $\bm q_n = \bm q_n(\bm x)= (q_{n1}, \dots, q_{nd})^\top$ converges to $\bm q = \bm q(\bm x)=(x_1^{-1/\alpha} \kappa_{1}^{1/\alpha}, \dots, x_d^{-1/\alpha}\kappa_{d}^{1/\alpha})^\top$. 
Since the convergence in the definition of $L$ is locally uniform (see, e.g., Theorem 1(v) in \citealp{SchSta06}), 
we obtain
\begin{align*}
    L(\bm{x}) = 
    \lim_{n \to \infty} n \Prob \{ c_n^{-1} \bm Y \in  [\bm 0, \bm q_n]^c\}
    &=
    \nu_{\bm Y}( [\bm 0, \bm q]^c).
\end{align*}
Finally, by the definition of the spectral measure, by \eqref{eq:relation-exponent-spectral-measure} and by the fact that $\kappa_j = \varsigma t_j$, we obtain
\begin{align*}
        L(\bm x) = \nu_{\bm Y}(K(\bm q(\bm x))) 
        &=
        \nu_{\bm Y} \Big( \Big\{ \bm y: \exists j \in [d] \text{ s.t. } y_j > x_j^{-1/\alpha} \kappa_j^{1/\alpha} \Big\} \Big)
        \\&=
        \nu_{\bm Y} \left(\varsigma^{1/\alpha} \left\{ \bm y: \exists j \in [d] \text{ s.t. } y_j > x_j^{-1/\alpha} t_j^{1/\alpha} \right\} \right)
        \\&=
        \varsigma^{-1} \nu_{\bm Y} \left(\left\{ \bm y: \exists j \in [d] \text{ s.t. } y_j > x_j^{-1/\alpha} t_j^{1/\alpha} \right\} \right)
        \\&=
        (\nu_{\alpha} \otimes \Phi_{\bm Y}) \left(\left\{ (r, \bm \lambda): \exists j \in [d] \text{ s.t. } r \lambda_j > x_j^{-1/\alpha} t_j^{1/\alpha} \right\} \right)
        \\&=
        \int_{\mathbb S_{+}^{d-1}} \max \Big( \frac{\lambda_j^{\alpha} x_j}{t_j} \Big)\, \diff  
        \Phi_{\bm Y}(\bm \lambda)
    \end{align*}
as asserted.
\end{proof}

\begin{lemma}
    \label{lemma:vanishing_noise}
    Let $\bm{X}$ and $\bm{Y}$ be two random vectors with values in $[0,\infty)^d$. Assume $\bm{X}$ is regularly varying with tail index $\alpha>0$, scaling sequence $(c_n)_n$ and associated exponent measure $\nu_{\bm{X}}$. Moreover, assume that 
    $\lim_{n \rightarrow \infty}n\mathbb{P}\left\{ c_n^{-1}\|\bm{Y}\| > a \right\} = 0$ for all $a>0$ 
    (which is a straightforward consequence of regular variation and the assumption $\Prob(\|\bm Y \|>x) / \Prob(\| \bm X\| > x)=o(1)$ as $x \to \infty$). 
    Then $(\bm X, \bm Y)$ is regularly varying  with tail index $\alpha>0$, scaling sequence $(c_n)_n$ and associated exponent measure $\nu_{\bm{X}}  \otimes \delta_{\bm 0}$. As a consequence, both $\bm{X} \vee \bm{Y}$ and $\bm{X} + \bm{Y}$ are regularly varying with tail index $\alpha>0$, scaling sequence $(c_n)_n$ and associated exponent measure $\nu_{\bm{X}}$.
\end{lemma}

\begin{proof}
Write $\nu_n = n \Prob(c_n^{-1}(\bm X, \bm Y) \in \cdot)$. 
We need to show that $\nu_n \to \nu: = \nu_{\bm X} \otimes \delta_{\bm 0}$ in $\mathbb M([0,\infty)^{2d} \setminus \{\bm 0\})$.
By Theorem B.1.17 in \cite{kulik2020heavy}, it is sufficient to show that $\nu_nf \to \nu f$ for all $f: [0,\infty)^{2d} \setminus \{\bm 0\} \to [0, \infty)$ that are Lipschitz continuous with respect to the metric induced by the maximum norm on $[0,\infty)^{2d}\setminus \{\bm 0\}$ (which is \textit{compatible}) and have support bounded away from zero; note that $\nu f = \int_{\Eb_0} f(\bm x, \bm 0) \, \nu_{\bm X}(\diff \bm x)$. 
For that purpose, note that at least one of the following two cases applies: either, the support of $f$ is contained in $A \times [0, \infty)^d$ with $A$ bounded away from zero, or it is contained in $[0, \infty)^d \times B$ with $B$ bounded away from zero. In the latter case, we have $\nu f=0$. Moreover, since $B$ is bounded away from zero, there exists $\kappa>0$ such that $f(\bm x, \bm y)=0$ for all $\bm y$ with $\|\bm y\|_\infty\le \kappa$. As a consequence,
\begin{align*} 
    \nu_n f 
    = 
    n \Exp[ f(c_n^{-1}\bm X, c_n^{-1}\bm Y) ]
    &=
    n \Exp[  f(c_n^{-1}\bm X, c_n^{-1}\bm Y)  \bm 1( c_n^{-1}\| \bm Y \|_\infty > \kappa) ]
    \\&\le
    \sup_{\bm z \in [0,\infty)^{2d} \setminus \{\bm 0\}} f(\bm z) \times  n \Prob(c_n^{-1} \|\bm Y\|_\infty>\kappa)
\end{align*}
which converges to $0=\nu f$ by assumption and equivalence of norms.

Next, consider the case where the support of $f$ is contained in $A \times [0, \infty)^d$ with $A$ bounded away from zero. In that case, we may decompose
\begin{align*}
    \nu_n f 
    &= 
    n \Exp[ f(c_n^{-1}\bm X, c_n^{-1}\bm Y) ]
    \\&=
    n \Exp[f(c_n^{-1}\bm X, \bm 0) ]
    +
    n \Exp[ f(c_n^{-1}\bm X, c_n^{-1}\bm Y) - f(c_n^{-1}\bm X, \bm 0) ] 
    \equiv 
    E_n + R_n.
\end{align*}
Since $\bm x \mapsto f(\bm x, \bm 0)$ is continuous and has support bounded away from zero, the summand $E_{n}$ converges to $\int f(\bm x, \bm 0) \, \nu_{\bm X}(\diff \bm x)=\nu f$ by regular variation of $\bm X$. It remains to show that $R_n$ converges to zero. For that purpose, let $\eps>0$, and decompose
\begin{align*}
    R_n
    &=
    n \Exp[ \{ f(c_n^{-1}\bm X, c_n^{-1}\bm Y) - f(c_n^{-1}\bm X, \bm 0) \}  \bm 1( c_n^{-1}\| \bm Y \|_\infty \le \eps) ] 
    \\& \hspace{2cm}+ 
    n \Exp[ \{ f(c_n^{-1}\bm X, c_n^{-1}\bm Y) - f(c_n^{-1}\bm X, \bm 0) \}  \bm 1( c_n^{-1}\| \bm Y \|_\infty > \eps) ]
    \equiv R_{n1} + R_{n2}.
\end{align*}
Since $f$ is bounded, the second expectation on the right-hand side can be bounded by
\[
R_{n2} \le 2 \sup_{\bm z \in [0,\infty)^{2d} \setminus \{\bm 0\}} f(\bm z) \times  n \Prob(c_n^{-1} \|\bm Y\|_\infty>\eps),
\]
which converges to zero by assumption and equivalence of norms (for any fixed $\eps>0$). 
Regarding $R_{n1}$, writing $g(\bm x) = \sup_{\|\bm y \|_\infty \le \eps} |f(\bm x, \bm y) - f(\bm x, \bm 0)|$, we have
\[
|R_{n1}| \le n \Exp[ g(c_n^{-1}\bm X)].
\]
We will show below that $g(\bm x) \le L \eps \times \bm 1(\bm x \in A)$ where $L>0$ is the Lipschitz constant of $f$. As a consequence, $|R_{n1}| \le L \eps \times n \Prob\{ c_n^{-1} \bm X \in A\} \le L \eps \times \nu_{\bm X}(\overline A) + \eps$ for all sufficiently large $n$ by the Portmanteau theorem (Theorem B.1.17 in \citealp{kulik2020heavy}), which can be made arbitrary small by decreasing $\eps$. It remains to show the bound on $g$. By the assumption on $f$, we have $g(\bm x)=0$ for $\bm x\in A^c$, and otherwise, 
for all $\| \bm y \|_\infty \le \eps$, Lipschitz continuity of $f$ (with Lipschitz constant $L$) implies
\begin{align*}
|f(\bm x, \bm y) - f(\bm x, \bm0)|
&\le
L \| (\bm 0, \bm y)\|_\infty 
\le L \eps.
\end{align*}
This yields $g(\bm x) \le L\eps  \times \bm 1(\bm x \in A)$ as asserted.

Finally, the assertions regrading $(\bm X \vee \bm Y)$ and $(\bm X + \bm Y)$ are straightforward consequences of the mapping theorem, see Theorem B.1.21 in \cite{kulik2020heavy}.
\end{proof}

\section{Additional simulation results}
\label{sec:additional-simulation-results}

Figure~\ref{fig:result_sum_noise} contained simulation results for the sum-linear model with noise only. Similar results were obtained for the sum-linear model without noise (Figure \ref{fig:result_sum_no_noise}), the max-linear model without noise (Figure~\ref{fig:result_max_no_noise}) and the max-linear model with noise (Figure~\ref{fig:result_max_noise}).

\begin{figure}[thp!]
    \centering
    \includegraphics[scale=0.37]{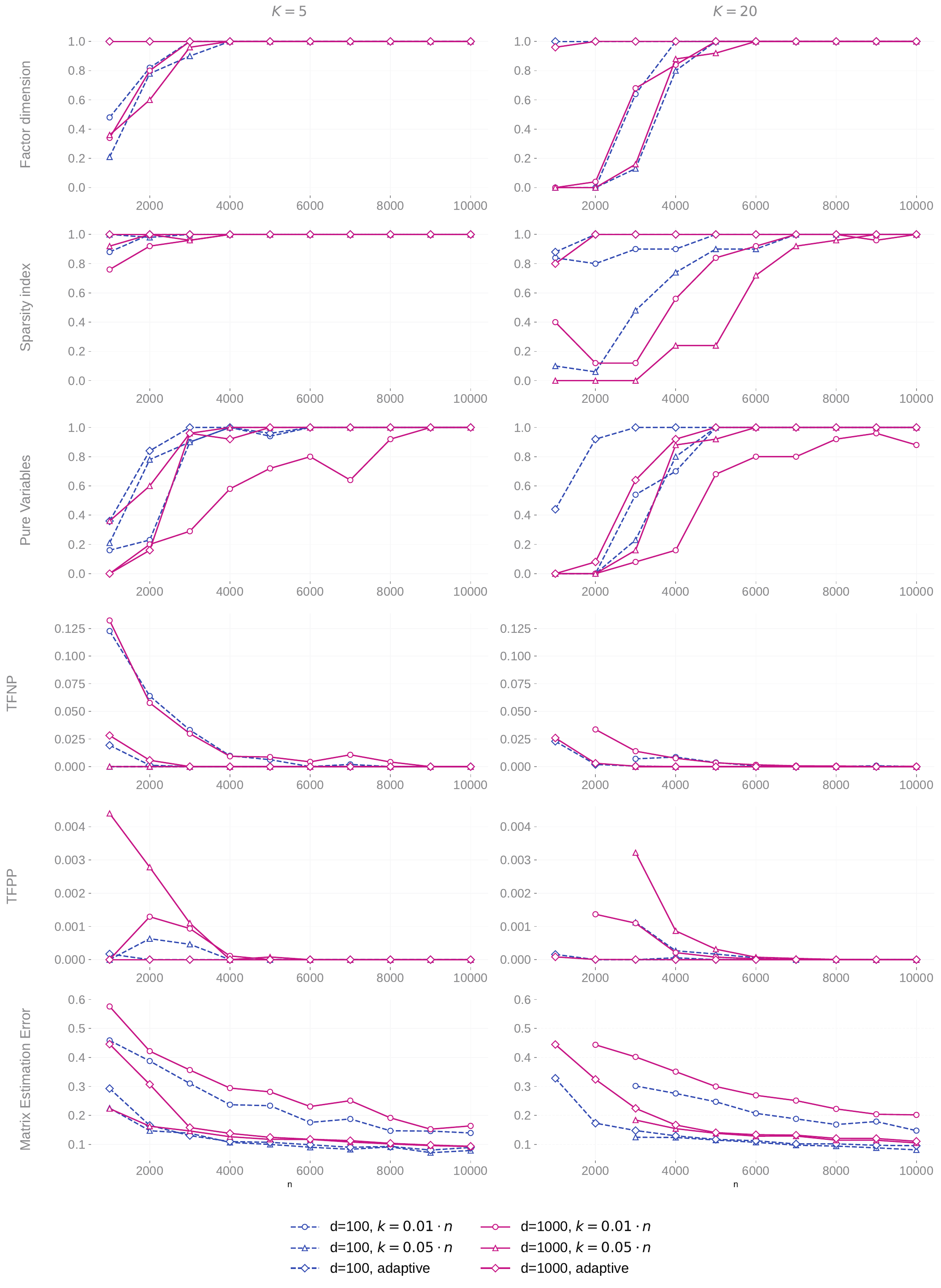}
    \caption{Performance metrics for the \textbf{linear model without noise} across different parameter combinations. Each rows depicts in order: (1) recovery rate of latent factors, 
        (2) recovery rate of sparsity, (3) recovery rate of pure variables, 
        (4) TFNP, (5) TFPP, and (6) matrix estimation Error. Each metric is plotted as a function of sample size $n$, 
        comparing dimensions $d \in \{100, 1000\}$, and the three choices of $(k,\hyperpurevar, \hypersparsity)$ described in the main text. Results are stratified by $K=5$ (left column) and $K=20$ (right column).}
    \label{fig:result_sum_no_noise}
\end{figure}

\begin{figure}[thp!]
    \centering
    \includegraphics[scale=0.37]{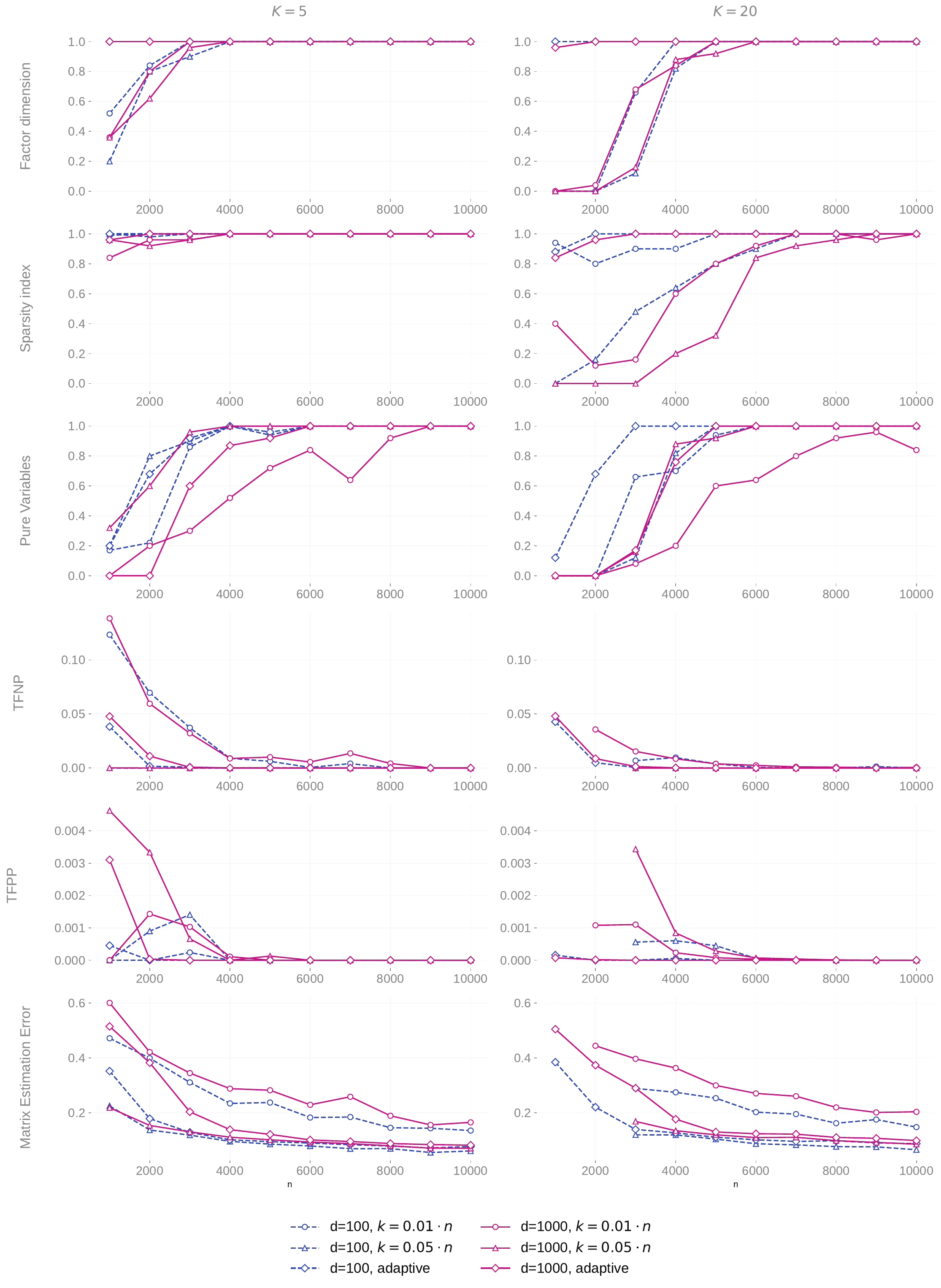}
    \caption{Same as Figure~\ref{fig:result_sum_no_noise}, but for the \textbf{max-linear model without noise}.}
    \label{fig:result_max_no_noise}
\end{figure}

\begin{figure}[thp!]
    \centering
    \includegraphics[scale=0.37]{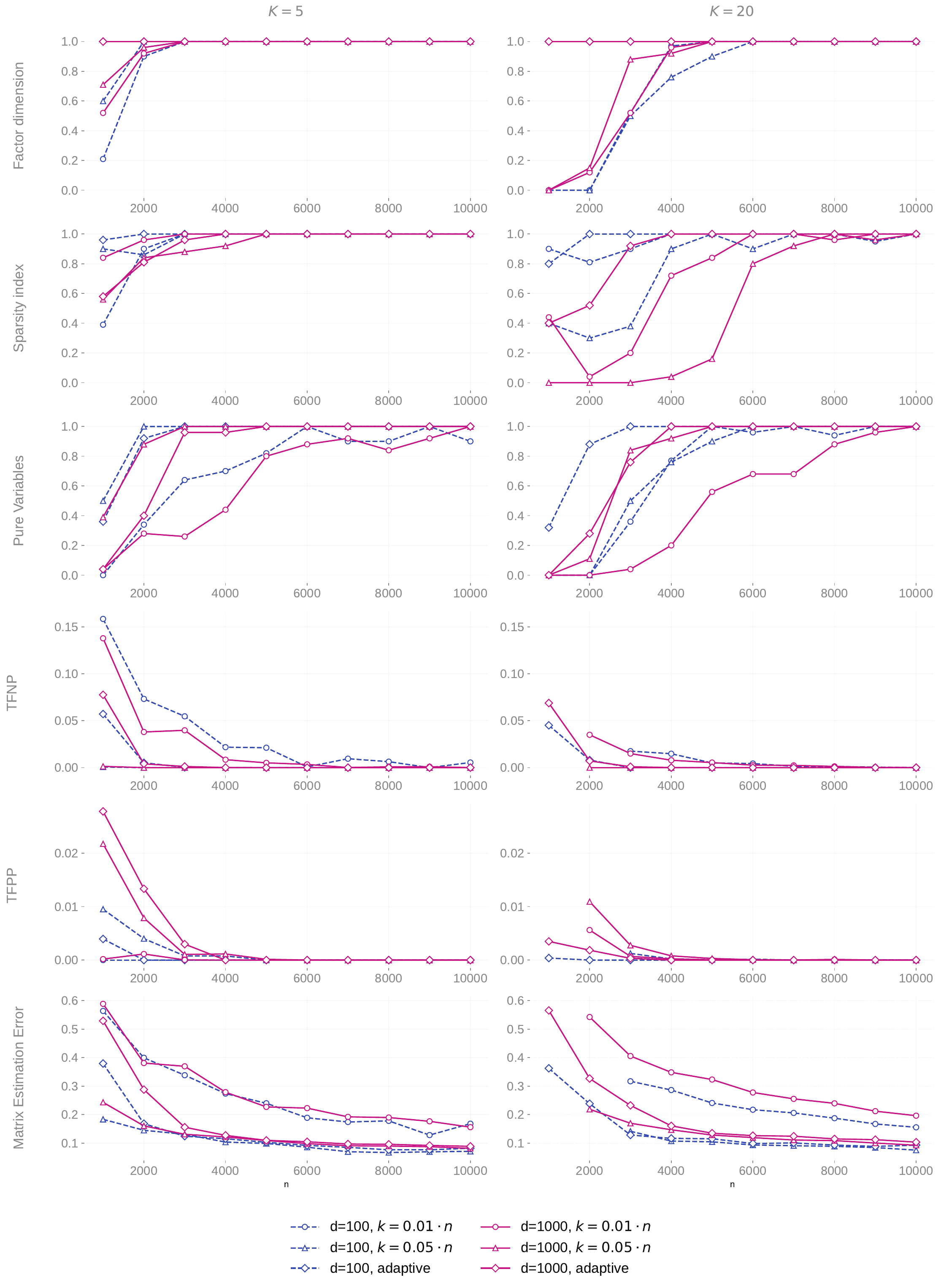}
    \caption{Same as Figure~\ref{fig:result_sum_no_noise}, but for the \textbf{max-linear model with noise}.}
    \label{fig:result_max_noise}
\end{figure}

\section{Additional results for the case study on dietary data}
\label{sec:supplement_case_studies}

Our method is applied to the full $d=39$ dimensional dietary intakes dataset introduced in Section~\ref{sec:case_studies}. We maintain the same $k$ value while fine-tuning $\hyperpurevar$ and $\hypersparsity$ following the procedure outlined in that section which results in $\hyperpurevar^* = 0.1449$ and $\hypersparsity^*=0.035$ with $\hat{K} = \hat{s} = 7$. Figure \ref{fig:Chi_fitted_vs_Chi_colored_full_dietary} compares the empirical extremal correlation to the corresponding extremal correlations implied by the fitted model. Two patterns emerge: (1) \emph{Pure-Pure} and \emph{Pure-Impure} pairs demonstrate excellent agreement, clustering tightly around the identity line, confirming that our method reliably captures dependencies between those pairs. (2) As in the case study on wind speed gusts from Section~\ref{sec:case-study-wind-speed}, \emph{Impure-Impure} pairs exhibit a poor fit, with values scattered broadly below from the identity line. An explanation is provided in Section~\ref{sec:case-study-wind-speed}.

\begin{figure}
    \centering
    \includegraphics[width=0.48\linewidth]{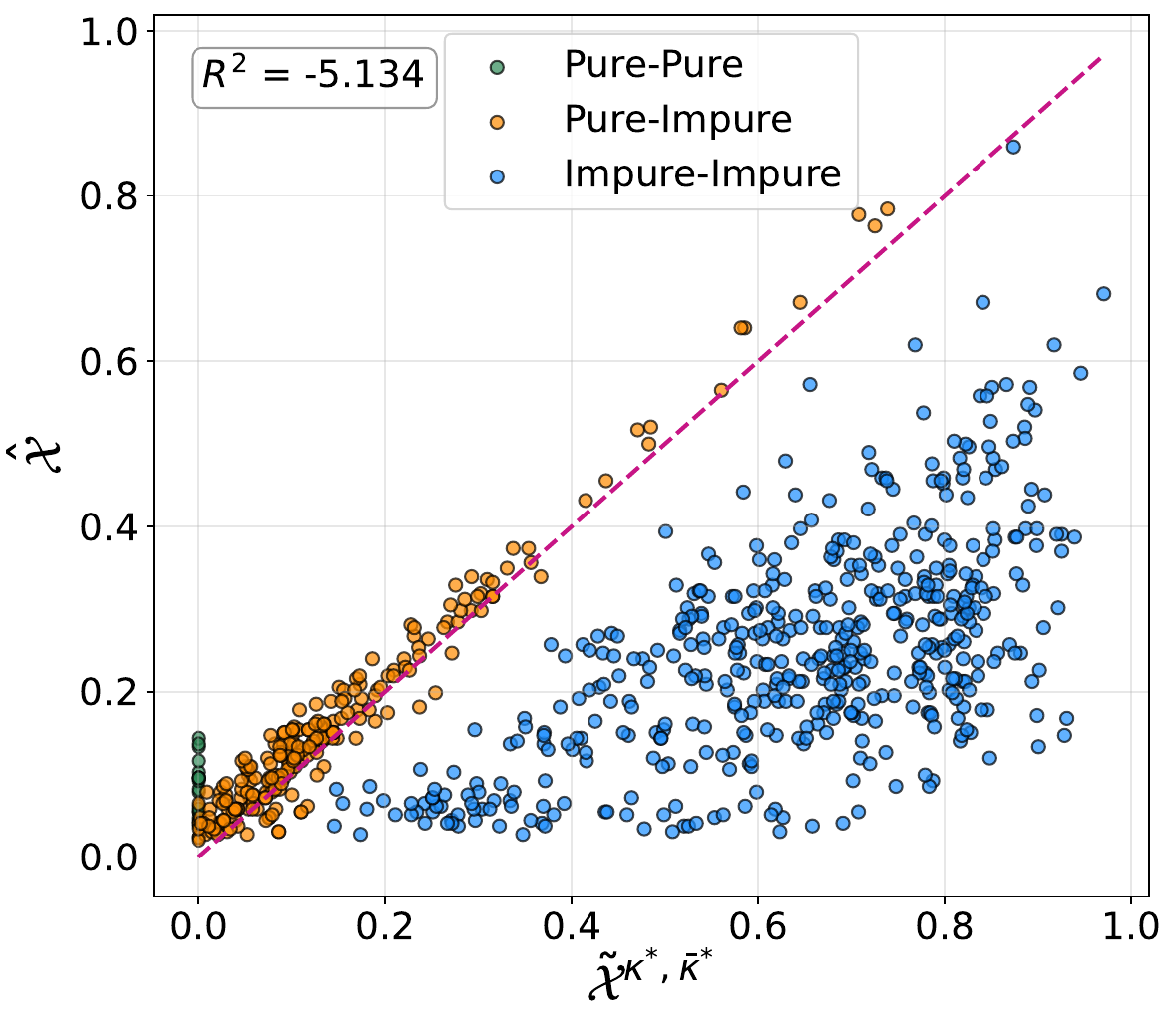}
    \caption{Results for the dietary data set with $d=39$: empirical correlations from $\hat{\mathcal{X}}$ vs. fitted extremal correlations from $\tilde{\mathcal{X}}^{\hyperpurevar^*, \hypersparsity^*}$.}
    \label{fig:Chi_fitted_vs_Chi_colored_full_dietary}
\end{figure}

\section*{Acknowledgements}
The authors are grateful to two unknown referees and an associate editor for their constructive comments that helped to improve the presentation substantially.

\section*{Declarations}

\begin{itemize}
\item\textbf{Ethical Approval.} Not applicable.
\item\textbf{Availability of supporting data. } Code for the simulation study and case studies are available in this Github repository \url{https://github.com/Aleboul/structured_linear_factor_models}. The datasets used in Section~\ref{sec:case_studies} are also publicly available on the websites mentioned in the main text.
\item\textbf{Competing interests.} The authors declare that they have no conflict of interest.
\item\textbf{Funding.} Both authors were supported by the Deutsche For\-schungsgemeinschaft (DFG, German Research Foundation; Project-ID 520388526;  TRR 391:  Spatio-temporal Statistics for the Transition of Energy and Transport) which is gratefully acknowledged. Calculations for this publication were performed on the HPC cluster Elysium of the Ruhr University Bochum, subsidised by the DFG (INST 213/1055-1).
\item\textbf{Authors' contributions.} 
A.\ Boulin conceived the idea for the study, developed the initial draft of the manuscript, and implemented the methods.
A.\ Bücher provided critical revisions to the manuscript and improved the content. 
Both authors contributed to the final version of the manuscript.
\end{itemize}

\bibliographystyle{apalike}
\bibliography{biblio}

\begin{thebibliography}{}

\bibitem[Avella-Medina et~al., 2024]{medina2024spectral}
Avella-Medina, M., Davis, R.~A., and Samorodnitsky, G. (2024).
\newblock Spectral learning of multivariate extremes.
\newblock {\em Journal of Machine Learning Research}, 25(124):1--36.

\bibitem[Beirlant et~al., 2004]{Bei04}
Beirlant, J., Goegebeur, Y., Segers, J., and Teugels, J. (2004).
\newblock {\em Statistics of extremes: Theory and Applications}.
\newblock Wiley Series in Probability and Statistics. John Wiley \& Sons Ltd.,
  Chichester.

\bibitem[Bing et~al., 2020]{Bin20}
Bing, X., Bunea, F., Ning, Y., and Wegkamp, M. (2020).
\newblock Adaptive estimation in structured factor models with applications to
  overlapping clustering.
\newblock {\em Ann. Statist.}, 48(4):2055--2081.

\bibitem[B\"ucher et~al., 2019]{BucherZou2019}
B\"ucher, A., Volgushev, S., and Zou, N. (2019).
\newblock On second order conditions in the multivariate block maxima and peak
  over threshold method.
\newblock {\em J. Multivariate Anal.}, 173:604--619.

\bibitem[Chautru, 2015]{chautru2015dimension}
Chautru, E. (2015).
\newblock {Dimension reduction in multivariate extreme value analysis}.
\newblock {\em Electronic Journal of Statistics}, 9(1):383 -- 418.

\bibitem[Chiapino et~al., 2019]{chiapino2019identifying}
Chiapino, M., Sabourin, A., and Segers, J. (2019).
\newblock Identifying groups of variables with the potential of being large
  simultaneously.
\newblock {\em Extremes}, 22(2):193--222.

\bibitem[Cl\'emen\c{c}on et~al., 2024]{clemencon2024regular}
Cl\'emen\c{c}on, S., Huet, N., and Sabourin, A. (2024).
\newblock Regular variation in {H}ilbert spaces and principal component
  analysis for functional extremes.
\newblock {\em Stochastic Process. Appl.}, 174:104375, 22.

\bibitem[Cooley and Thibaud, 2019]{cooley2019decompositions}
Cooley, D. and Thibaud, E. (2019).
\newblock Decompositions of dependence for high-dimensional extremes.
\newblock {\em Biometrika}, 106(3):587--604.

\bibitem[Drees and Huang, 1998]{DreesHuang1998}
Drees, H. and Huang, X. (1998).
\newblock Best attainable rates of convergence for estimators of the stable
  tail dependence function.
\newblock {\em J. Multivariate Anal.}, 64(1):25--47.

\bibitem[Drees and Sabourin, 2021]{drees2021principal}
Drees, H. and Sabourin, A. (2021).
\newblock {Principal component analysis for multivariate extremes}.
\newblock {\em Electronic Journal of Statistics}, 15(1):908 -- 943.

\bibitem[Duchi et~al., 2008]{duchi2008efficient}
Duchi, J., Shalev-Shwartz, S., Singer, Y., and Chandra, T. (2008).
\newblock Efficient projections onto the $\ell_1$-ball for learning in high
  dimensions.
\newblock In {\em Proceedings of the 25th international conference on Machine
  learning}, pages 272--279.

\bibitem[Einmahl et~al., 2018]{einmahl2018continuous}
Einmahl, J.~H., Kiriliouk, A., and Segers, J. (2018).
\newblock A continuous updating weighted least squares estimator of tail
  dependence in high dimensions.
\newblock {\em Extremes}, 21:205--233.

\bibitem[Einmahl et~al., 2012]{einmahl2012mestimator}
Einmahl, J. H.~J., Krajina, A., and Segers, J. (2012).
\newblock {An M-estimator for tail dependence in arbitrary dimensions}.
\newblock {\em The Annals of Statistics}, 40(3):1764 -- 1793.

\bibitem[Engelke and Ivanovs, 2021]{Eng21}
Engelke, S. and Ivanovs, J. (2021).
\newblock Sparse structures for multivariate extremes.
\newblock {\em Annu. Rev. Stat. Appl.}, 8:241--270.

\bibitem[Engelke et~al., 2021]{engelke2021learning}
Engelke, S., Lalancette, M., and Volgushev, S. (2021).
\newblock Learning extremal graphical structures in high dimensions.
\newblock {\em arXiv preprint arXiv:2111.00840}.

\bibitem[Gissibl et~al., 2018]{gissibl2018tail}
Gissibl, N., Kl\"uppelberg, C., and Otto, M. (2018).
\newblock Tail dependence of recursive max-linear models with regularly varying
  noise variables.
\newblock {\em Econom. Stat.}, 6:149--167.

\bibitem[Goix et~al., 2017]{goix2017sparse}
Goix, N., Sabourin, A., and Cl{\'e}men{\c{c}}on, S. (2017).
\newblock Sparse representation of multivariate extremes with applications to
  anomaly detection.
\newblock {\em Journal of Multivariate Analysis}, 161:12--31.

\bibitem[Hult and Lindskog, 2006]{hult2006regular}
Hult, H. and Lindskog, F. (2006).
\newblock Regular variation for measures on metric spaces.
\newblock {\em Publications de l'Institut Math{\'e}matique}, 80(94):121--140.

\bibitem[Izenman, 2008]{izenman2008modern}
Izenman, A.~J. (2008).
\newblock {\em Modern multivariate statistical techniques}, volume~1.
\newblock Springer.

\bibitem[Jan{\ss}en and Wan, 2020]{janssen2020k}
Jan{\ss}en, A. and Wan, P. (2020).
\newblock $ k $-means clustering of extremes.
\newblock {\em Electronic Journal of Statistics}, 14(1):1211--1233.

\bibitem[Katz et~al., 2002]{Kat02}
Katz, R.~W., Parlange, M.~B., and Naveau, P. (2002).
\newblock Statistics of extremes in hydrology.
\newblock {\em Advances in Water Resources}, 25(8):1287 -- 1304.

\bibitem[Kiriliouk and Zhou, 2022]{kiriliouk2022estimating}
Kiriliouk, A. and Zhou, C. (2022).
\newblock Estimating probabilities of multivariate failure sets based on
  pairwise tail dependence coefficients.
\newblock {\em arXiv preprint arXiv:2210.12618}.

\bibitem[Kl{\"u}ppelberg and Krali, 2021]{kluppelberg2021estimating}
Kl{\"u}ppelberg, C. and Krali, M. (2021).
\newblock Estimating an extreme {B}ayesian network via scalings.
\newblock {\em Journal of Multivariate Analysis}, 181:104672.

\bibitem[Krali et~al., 2025]{krali2025heavy}
Krali, M., Davison, A.~C., and Klüppelberg, C. (2025).
\newblock Heavy-tailed max-linear structural equation models in networks with
  hidden nodes.
\newblock {\em arXiv preprint arXiv:2306.15356}.

\bibitem[Kuhn, 1955]{kuhn1955hungarian}
Kuhn, H.~W. (1955).
\newblock The {H}ungarian method for the assignment problem.
\newblock {\em Naval research logistics quarterly}, 2(1-2):83--97.

\bibitem[Kulik and Soulier, 2020]{kulik2020heavy}
Kulik, R. and Soulier, P. (2020).
\newblock {\em Heavy-tailed time series}.
\newblock Springer.

\bibitem[Kyrillidis et~al., 2013]{kyrillidis2013sparse}
Kyrillidis, A., Becker, S., Cevher, V., and Koch, C. (2013).
\newblock Sparse projections onto the simplex.
\newblock In {\em International Conference on Machine Learning}, pages
  235--243. PMLR.

\bibitem[McDiarmid, 1998]{McDiarmid1998}
McDiarmid, C. (1998).
\newblock Concentration.
\newblock In {\em Probabilistic methods for algorithmic discrete mathematics},
  volume~16 of {\em Algorithms Combin.}, pages 195--248. Springer, Berlin.

\bibitem[Meyer and Wintenberger, 2024]{meyer2024multivariate}
Meyer, N. and Wintenberger, O. (2024).
\newblock Multivariate sparse clustering for extremes.
\newblock {\em Journal of the American Statistical Association},
  119(547):1911--1922.

\bibitem[Mourahib et~al., 2025]{Mou25}
Mourahib, A., Kiriliouk, A., and Segers, J. (2025).
\newblock Multivariate generalized {P}areto distributions along extreme
  directions.
\newblock {\em Extremes}, 28(2):239--272.

\bibitem[Oh and Patton, 2017]{OhPat17}
Oh, D.~H. and Patton, A.~J. (2017).
\newblock Modeling dependence in high dimensions with factor copulas.
\newblock {\em J. Bus. Econom. Statist.}, 35(1):139--154.

\bibitem[Resnick, 2024]{resnick2024art}
Resnick, S. (2024).
\newblock {\em The Art of Finding Hidden Risks: Hidden Regular Variation in the
  21st Century}.
\newblock Springer Nature.

\bibitem[Resnick, 1987]{Resnick1987}
Resnick, S.~I. (1987).
\newblock {\em Extreme values, regular variation, and point processes},
  volume~4 of {\em Applied Probability. A Series of the Applied Probability
  Trust}.
\newblock Springer-Verlag, New York.

\bibitem[Resnick, 2007]{Res07}
Resnick, S.~I. (2007).
\newblock {\em Heavy-tail phenomena}.
\newblock Springer Series in Operations Research and Financial Engineering.
  Springer, New York.
\newblock Probabilistic and statistical modeling.

\bibitem[Sabourin, 2025]{MLExtreme}
Sabourin, A. (2025).
\newblock Mlextreme v0.1.2.
\newblock {\em GitHub repository}.

\bibitem[Schmidt and Stadtm\"uller, 2006]{SchSta06}
Schmidt, R. and Stadtm\"uller, U. (2006).
\newblock Non-parametric estimation of tail dependence.
\newblock {\em Scand. J. Statist.}, 33(2):307--335.

\bibitem[Simpson et~al., 2020]{Sim20}
Simpson, E.~S., Wadsworth, J.~L., and Tawn, J.~A. (2020).
\newblock Determining the dependence structure of multivariate extremes.
\newblock {\em Biometrika}, 107(3):513--532.

\bibitem[Wang and Stoev, 2011]{Wan11}
Wang, Y. and Stoev, S.~A. (2011).
\newblock Conditional sampling for spectrally discrete max-stable random
  fields.
\newblock {\em Adv. in Appl. Probab.}, 43(2):461--483.

\end{thebibliography}

\end{document}